\documentclass[12pt,reqno]{amsart}
\usepackage{amssymb}
\usepackage{mathtools}
\usepackage{txfonts}
\usepackage{eucal}
\usepackage{extarrows}
\usepackage{amscd}
\usepackage[all]{xy}
\usepackage{amsmath}
\usepackage{tikz}
\usepackage{graphicx}
\usepackage{float}
\usepackage{multirow}
\usepackage{makecell}
\usepackage{xspace}
\usepackage{diagbox}

\usepackage[a4paper,body={16.3cm,22.8cm},centering]{geometry}
\usepackage[colorlinks,final,backref=page,hyperindex,pdftex]{hyperref}

\newcommand{\nc}{\newcommand}

\newcommand{\delete}[1]{}

\newcommand{\mbibitem}[1]{\bibitem{#1}} 

\delete{

\newcommand{\mbibitem}[1]{\bibitem[\bf #1]{#1}} 
}
\newtheorem{thm}{Theorem}[section]
\newtheorem{prop}[thm]{Proposition}
\newtheorem{lem}[thm]{Lemma}
\newtheorem{cor}[thm]{Corollary}

\theoremstyle{definition}
\newtheorem{defn}[thm]{Definition}
\newtheorem{rem}[thm]{Remark}
\newtheorem{exam}[thm]{Example}
\newtheorem{prop-def}{Proposition-Definition}[section]

\nc{\tred}[1]{\textcolor{red}{#1}} \nc{\tgreen}[1]{\textcolor{green}{#1}}
\nc{\tblue}[1]{\textcolor{blue}{#1}} \nc{\tpurple}[1]{\textcolor{purple}{#1}}

\nc{\yuan}[1]{\tred{\underline{Yuan:}#1 }}
\nc{\wen}[1]{\tblue{\underline{Wen:}#1 }}

\newcommand{\cee}{(E,\succ_{E},\prec_{E})}
\newcommand{\dnv}{D\,\natural V}
\newcommand{\dnvp}{D\,\natural^{\prime} V}
\newcommand{\ceep}{(E,\succ_{E}^{\prime},\prec_{E}^{\prime})}

\begin{document}

\title[Extending Structures for Dendriform Algebras]{Extending Structures for Dendriform Algebras}
\author{Yuanyuan Zhang}
\address{School of Mathematics and Statistics,  Henan University,  Henan,  Kaifeng 475004,  P.\, R. China}
\email{zhangyy17@henu.edu.cn}
\author{Junwen Wang}
\address{School of Mathematics and Statistics,  Henan University,  Henan,  Kaifeng 475004,  P.\, R. China}
\email{3111435107@qq.com}

\date{\today}

\begin{abstract}
In this paper, we devote to extending structures for dendriform algebras. First, we define extending datums and unified products of dendriform algebras, and theoretically solve the extending structure problem. As an application, we consider flag datums as a special case of extending structures, and give an example of the extending structure problem. Second, we introduce matched pairs and bicrossed products of dendriform algebras and theoretically solve the factorization problem for dendriform algebras. Moreover, we also introduce cocycle semidirect products and nonabelian semidirect products as special cases of unified products. Finally, we define the deformation map on a dendriform extending structure (more general case), not necessary a matched pair, which is more practical in the classifying complements problem.
\end{abstract}

\makeatletter
\@namedef{subjclassname@2020}{\textup{2020} Mathematics Subject Classification}
\makeatother


\maketitle

\tableofcontents

\setcounter{section}{0}

\allowdisplaybreaks

\section{Introduction}
A dendriform algebra is a vector space $D$ with two binary operations $\prec$ and $\succ$, satisfying three relations (see Eqs.~\eqref{definition:dendriform:formula1}-\eqref{definition:dendriform:formula3}). The concept was introduced by Loday~\cite{L01} in 1995 with the motivation from periodicity of algebraic K-theory and operads. There is a close relationship between Rota-Baxter algebras and dendriform algebras. In ~\cite{Agu00, EG02,EG08}, the authors proved that the Rota-Baxter algebra can induce dendriform algebras. Moreover, dendriform algebras have been studied quite extensively with connections to several areas in mathematics and physics, such as arithmetics~\cite{L02}, homology~\cite{F98}, preLie algebras~\cite{Agu002,AB07,M12}, Hopf algebras~\cite{L98} and combinatorics~\cite{EMP08}.
\smallskip

The extending structure problem has been studied extensively, it is first studied by Agore and Militaru in group theory ~\cite{AM14}. Later, they also studied Lie algebras~\cite{AM14L}, Leibniz algebras~\cite{AM13}, Hopf algbras\cite{AM13B}, poisson algebras~\cite{AM15} and associative algebras~\cite{AM16}.
In recent years, many researchers studied the problem for various algebraic structures. Zhang studied extending structures of 3-Lie algebras~\cite{Z22b} and unified products for braided Lie bialgebras~\cite{Z22a}, Hong studied extending structures of left-symmetric algebras~\cite{H19b}, associative conformal algebras~\cite{H19a}, Lie conformal algebras~\cite{H17} and Lie bialgebras~\cite{H21}, Zhao et.al studied extending structures of Lie conformal superalgebras~\cite{ZCY}, Peng et.al studied extending structures of Rota-Baxter Lie algebras~\cite{PZ}, Hou studied extending structures of perm algebras and perm bialgebras~\cite{HB23}, and so on. The scholars who care about this problem often define extending structures and unified products of different algebraic structures. In this paper, we will define dendriform extending structures and unified products as a useful tool to solve the problem. Moreover, for different algebra extending structures (resp. matched pairs), we have the following diagram:
\begin{align*}
\CD
  \makecell[c]{\text{dendriform}\\ \text{extending structure}} @>\text{by Eq.~\eqref{formulas:dendriform to associative}}>> \makecell[c]{\text{associative}\\ \text{extending structure~\cite{AM16}}} \\
  @V\text{by Eq.~\eqref{formulas:dendriform to prelie}}VV @V\text{by Eq.~\eqref{formulas:asso to lie}}VV  \\
  \makecell[c]{\text{preLie}\\ \text{extending structure~\cite{H19b}}} @>\text{by Eq.~\eqref{formulas:preLie to lie}}>> \makecell[c]{\text{Lie}\\ \text{extending structure~\cite{AM14L}}}
\endCD
\end{align*}

\begin{defn}
If $D$ is a subalgebra of dendriform algebra $E$, then $E$ is also called an {\bf extension} of $D$, and is denoted by $D \subset E$. In this case, we also say $D\subset E$ is an extension of dendriform algebras. A subspace $V$ of $E$ is called a {\bf space complement} of $D$ in $E$ if $E=D+V$, and $D\cap V={0}$.
\end{defn}

Our first aim of this paper is to study the extending structure problem for dendriform algebras.
{\bf The Extending Structure (ES) Problem}: Let $D$ be a dendriform algebra, $E$ a vector space containing $D$ as a subspace. Describe and classify all extensions $\cee$ of $D$.
The ES problem can be rephrased with the language of ``subalgebra": Let $D$ be a dendriform algebra, $E$ a vector space containing $D$ as a subspace. Describe and classify all dendriform algebra structures that can be defined on $E$ containing $D$ as a subalgebra.
\smallskip

Our second aim of this paper is to study the factorization problem for dendriform algebras, which is the ES problem with an additional assumption ``the space complement $V$ of $D$ in $E$ to be also a subalgebra of $E$".

{\bf The Factorization Problem}: Let $D$ and $V$ be two dendriform algebras. Describe and classify all dendriform algebra structures that can be defined on $E$ such that $E$ {\bf factorizes through} $D$ and $V$, i.e., $E$ contains $D$ and $V$ as subalgebras such that $E=D+V$ and $D\cap V=\{0\}$.

\smallskip
Our third aim of this paper is to study various products by unified products as a powerful tool. Various products are indeed special cases of unified products, which is the origin of the name ``unified product",  refer to the following diagram.

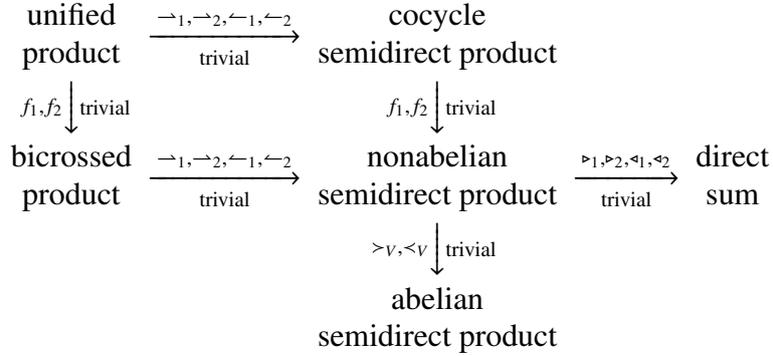
\begin{align*}
\CD
  \makecell[c]{\text{unified}\\ \text{product}} @>\rightharpoonup_{1},\rightharpoonup_{2},\leftharpoonup_{1},\leftharpoonup_{2}>\text{trivial}> \makecell[c]{\text{cocycle}\\ \text{semidirect product}} \\
  @V f_{1},f_{2} V \text{trivial} V @V f_{1},f_{2} V \text{trivial} V  \\
  \makecell[c]{\text{bicrossed}\\ \text{product}} @>\rightharpoonup_{1},\rightharpoonup_{2},\leftharpoonup_{1},\leftharpoonup_{2}>\text{trivial}> \makecell[c]{\text{nonabelian}\\ \text{semidirect product}} @>\triangleright_{1},\triangleright_{2},\triangleleft_{1},\triangleleft_{2}>\text{trivial}> \makecell[c]{\text{direct}\\ \text{sum}} \\
  && @V \succ_{V},\prec_{V} V \text{trivial} V   \\
  &&\makecell[c]{\text{abelian}\\ \text{semidirect product}}
\endCD
\end{align*}

\begin{defn}
  Let $D \subset E$ be an extension of dendriform algebras. A subalgebra $V$ of $E$ is called a {\bf dendriform complement} of $D$ in $E$ (or a  {\bf $D$ dendriform complement} of $E$) if $E=D+V$ and $D\cap V=\{0\}$. We know that $V$ is a dendriform complement of $D$ in $E$ if and only if $E$ factorizes through $D$ and $V$.
\end{defn}

Our fourth aim of this paper is to study the classifying complements problem for dendriform algebras.
Generally speaking, the classifying complements problem is an inverse problem of the factorization problem, the factorization problem is to solve the problem $D+V=?$, while the classifying complements problem is to solve the problem $D+?=E$.

{\bf The Classifying Complements Problem(CCP)}: Let $D \subset E$ be an extension of dendriform algebras.
If a dendriform complement of $D$ in $E$ exists, describe and classify all dendriform complements of $D$ in $E$, and compute the cardinal of the isomorphism classes of all $D$ dendriform complements of $E$, which is called the {\bf index} of $D$ in $E$ and will be denoted by $[E:D]$.

The CCP problem has been studied in many algebraic structures, such as groups~\cite{AM15GC}, Lie algebras~\cite{AM14LC}, associative algebras~\cite{A14}, Hopf algbras\cite{AM13HC}, left-symmetric algebras~\cite{H19b}, perm algebras~\cite{HB23}, and so on.
 The scholars who care about this problem often define deformation maps on the given algebraic complement (i.e., a matched pair) of different algebraic structures. In this paper, we define deformation maps on a dendriform extending structure (more general case), not necessary a matched pair. In this case, the condition ``if a dendriform complement of $D$ in $E$ exists" is no longer required.

\smallskip
The paper is organized as follows. In Section 2, we recall some basic concepts and results. In Section 3, we define extending datums and unified products, and theoretically solve the extending structure problem. Moreover, we prove that the flag datum is a special case of the dendriform extending structure, and give an example of classification for the extending structure problem. In Section 4, we introduce matched pairs and bicrossed products of dendriform algebras to solve the factorization problem, and we also introduce cocycle semidirect products and nonabelian semidirect products of dendriform algebras as special cases of unified products. Then we consider the splitting of short exact sequences. In Section 5, we introduce the deformation map to solve the classifying complements problem for dendriform algebras. In Section 6, we give two problems that may be worthy to study in the future.

\smallskip
{\bf Notation.}
Throughout this paper,  we fix ${\bf k}$ the complex number field,
which will be the base field of all vectors, algebras, as well as linear maps.

\section{Preliminary}

In this section, we mainly recall some basic concepts of dendriform algebras, which will be used later. Meanwhile, we give some examples.
\subsection{Dendriform algebras}
\begin{defn}~\cite{Guo12,L01}\label{defn:dendriform algebra}
A {\bf dendriform algebra} is a vector space $D$, together with bilinear operations $\succ,\prec:D\times D\rightarrow D$, such that the following conditions hold for all $x,y,z\in D$
\begin{align}
(x\succ y+x\prec y)\succ z=&x\succ (y\succ z);\label{definition:dendriform:formula1}\\
(x\prec y)\prec z=&x\prec (y\succ z+y\prec z);\label{definition:dendriform:formula2}\\
(x\succ y)\prec z=&x\succ (y\prec z).\label{definition:dendriform:formula3}
\end{align}
We denote it by $(D,\succ,\prec)$, or simply $D$ if there is no confusion.
\end{defn}

\begin{rem}\label{rem:asso}
Let $D$ be a dendriform algebra.
\begin{enumerate}
   \item \label{asso} Define a new operation $\star:=\succ+\prec$, then $(D,\star)$ is an associative algebra~\cite{Guo12}.
   \item \label{preLie} Define a new operation $\diamond:=\succ-\prec$, then $(D,\diamond)$ is a left preLie algebra~\cite{Agu002}.
   \end{enumerate}
   \end{rem}

\begin{exam}\label{exam:dendriform algebra}
\begin{enumerate}
\item\label{exam:dendriform algebra1} Suppose that $D={\bf k}\{e_{1}\}$, $e_{1}\succ e_{1}=e_{1}$, $e_{1}\prec e_{1}=0$, then $(D,\succ,\prec)$ is a dendriform algebra.
\item\label{exam:dendriform algebra2} Suppose that $B={\bf k}\{e_{2}\}$, $e_{2}\succ e_{2}=0$, $e_{2}\prec e_{2}=e_{2}$, then $(B,\succ,\prec)$ is a dendriform algebra.
\item Suppose that $E={\bf k}\{e_{1},e_{2}\}$ satisfying the following conditions
\begin{align*}
\begin{split}
e_{1}\succ e_{1}=&e_{1};\\
e_{1}\prec e_{1}=&0;\\
\end{split}
\begin{split}e_{1}\succ e_{2}&=e_{2};\\
e_{1}\prec e_{2}&=2e_{1}-e_{2};\\
\end{split}\,\,\,\,\,\,\,\,\,
\begin{split}e_{2}\succ e_{1}&=2e_{1};\\
e_{2}\prec e_{1}&=0;\\
\end{split}
\begin{split}e_{2}\succ e_{2}&=2e_{2};\\
e_{2}\prec e_{2}&=4e_{1}-2e_{2}.
\end{split}
\end{align*}
Then $(E,\succ,\prec)$ is a dendriform algebra.
\end{enumerate}
\end{exam}
\begin{defn}~\cite{Guo12,L01}
 Let $D$ and $B$ be two dendriform algebras. A linear map $\varphi:D\rightarrow B$ is called a {\bf morphism} of dendriform algebras if the following identities hold for all $x,y\in D$,
    \begin{align*}
   \varphi(x\succ y)&=\varphi(x)\succ_{B}\varphi(y);\\
  \varphi(x\prec y)&=\varphi(x)\prec_{B}\varphi(y).
    \end{align*}
\end{defn}
A morphism $\varphi$ is called an {\bf isomorphism} if it is also a bijective map. We call $D$ and $B$ {\bf isomorphic} if there exists an isomorphism of dendriform algebras $\varphi:D$ $\rightarrow B$ , we denote it by $D\cong B$.


\subsection{The module of dendriform algebras}
\begin{defn}~\cite{D21} Let $D$ be a dendriform algebra, and $V$ a vector space. Suppose that $\triangleright_{1},\triangleright_{2}:D\times V\rightarrow V$ and $\triangleleft_{1},\triangleleft_{2}:V\times D\rightarrow V$ are four bilinear maps. We define two new bilinear maps $\triangleright:=\triangleright_{1}+\triangleright_{2}$ and $\triangleleft:=\triangleleft_{1}+\triangleleft_{2}$.
\allowdisplaybreaks{
\begin{enumerate}
  \item The triple $(V,\triangleright_{1},\triangleright_{2})$ is called a {\bf left $D$ module} if the following identities hold for all $a,b\in D$ and $x\in V$,
    \begin{align}
    \label{formulas:left module definition}
    \begin{split}
   &(a\star b)\triangleright_{1} x=a\triangleright_{1}(b\triangleright_{1} x);\\
                &(a\prec b)\triangleright_{2} x=a\triangleright_{2}(b\triangleright x);\\
                &(a\succ b)\triangleright_{2} x=a\triangleright_{1}(b\triangleright_{2} x).
  \end{split}
    \end{align}
  \item The triple $(V,\triangleleft_{1},\triangleleft_{2})$ is called a {\bf right $D$ module} if the following identities hold for all $a,b\in D$ and $x\in V$,
    \begin{align}\label{formulas:right module definition}
      \begin{split}
      (x\triangleleft a)\triangleleft_{1} b=&x\triangleleft_{1}(a\succ b);\\
      (x\triangleleft_{2} a)\triangleleft_{2} b=&x\triangleleft_{2}(a\star b);\\
      (x\triangleleft_{1} a)\triangleleft_{2} b=&x\triangleleft_{1}(a\prec b).
  \end{split}
    \end{align}
  \item The 5-tuple $(V,\triangleright_{1},\triangleright_{2},\triangleleft_{1},\triangleleft_{2})$ is called a {\bf $D$ bimodule} if the triple $(V,\triangleright_{1},\triangleright_{2})$ is a left $D$ module, $(V,\triangleleft_{1},\triangleleft_{2})$ is a right $D$ module, and the following identities hold for all $a,b\in D$ and $x\in V$,
    \begin{align}\label{formulas:bimodule definition}
      \begin{split}
         (a\triangleright x)\triangleleft_{1} b=&a\triangleright_{1}(x\triangleleft_{1} b);\\
         (a\triangleright_{2} x)\triangleleft_{2} b=&a\triangleright_{2}(x\triangleleft b);\\
         (a\triangleright_{1} x)\triangleleft_{2} b=&a\triangleright_{1}(x\triangleleft_{2} b).
    \end{split}
    \end{align}
  \item Let $(V,\triangleright_{1},\triangleright_{2})$ and $(V^{\prime},\triangleright_{1}^{\prime},\triangleright_{2}^{\prime})$ be
two left $D$ modules. A linear map $\varphi:V\rightarrow V^{\prime}$ is called a {\bf left $D$ module morphism} if the following identities hold for all $a\in D$ and $x\in V$,
    \begin{align*}
  \varphi(a\triangleright_{1} x)&=a\triangleright^{\prime}_{1}\varphi(x);\\
  \varphi(a\triangleright_{2} x)&=a\triangleright^{\prime}_{2}\varphi(x).
    \end{align*}
    The right $D$ module morphism can be defined similarly. The morphism of $D$ bimodules is a linear map which is both a left $D$ module morphism and a right $D$ module morphism.
   \end{enumerate}
   }
\end{defn}

\begin{rem} Let $D$ be a dendriform algebra.
 \begin{enumerate}
  \item The triple $(D,\succ,\prec)$ is a left (right) $D$ module, and we call it the {\bf left (right) regular module}.
  \item A left $D$ module $(V,\triangleright_{1},\triangleright_{2})$ is called {\bf trivial}, if the action maps $\triangleright_{1}$ and $\triangleright_{2}$ are trivial. A linear (or bilinear) map is called trivial if it is 0.
  \item Let $(V,\triangleright_{1},\triangleright_{2},\triangleleft_{1},\triangleleft_{2})$ be a $D$ {\bf bimodule}. By Remark~\ref{rem:asso}~\ref{asso}, define two new operations $\triangleright:=\triangleright_{1}+\triangleright_{2}$ and $\triangleleft:=\triangleleft_{1}+\triangleleft_{2}$, then $(V,\triangleright,\triangleleft)$ is a $(D,\star)$ {\bf bimodule}.
   \end{enumerate}
\end{rem}

\begin{exam} \label{exam:bimodule}
To continue Example~\ref{exam:dendriform algebra}~\ref{exam:dendriform algebra1} and suppose $V={\bf k}\{e_{2}\}$. Now we give some $D$ bimodules by defining the actions as follows:
\allowdisplaybreaks{
\begin{align*}
\begin{aligned}
e_{1}\triangleright_{1}e_{2}&=\bar{l}_{1}e_{2};\\
e_{2}\triangleleft_{1}e_{1}&=\bar{r}_{1}e_{2};\\
\end{aligned}\,\,\,\,\,\,\,
\begin{aligned}
e_{1}\triangleright_{2}e_{2}&=\bar{l}_{2}e_{2};\\
e_{2}\triangleleft_{2}e_{1}&=\bar{r}_{2}e_{2}.
\end{aligned}
\end{align*}
}By direct computation, we obtain all possible solutions of $\bar{l}_1,\bar{l}_2, \bar{r}_1,\bar{r}_2$ in Table~\ref{table:bimodules} as follows. In fact, each column defines a $D$ bimodule $(V,\triangleright_{1},\triangleright_{2},\triangleleft_{1},\triangleleft_{2})$.

\begin{table}[H]
  \centering
  \renewcommand\arraystretch{1.6}
\begin{tabular}{|c|c|c|c|c|c|c|c|c|}
  \hline
$\bar{l}_{1}$ & $1$ & $0$ & $1$ & $1$ & $1$ & $1$ & $1$ & $0$\\
  \hline
$\bar{l}_{2}$ & $-1$ & $0$ & $0$ & $0$ & $-1$ & $0$ & $0$ & $0$\\
  \hline
$\bar{r}_{1}$ & $0$ & $1$ & $0$ & $0$ & $0$ & $1$ & $0$ & $0$\\
  \hline
$\bar{r}_{2}$ & $0$ & $0$ & $0$ & $1$ & $1$ & $0$ & $1$ & $0$\\
  \hline
\end{tabular}
  \caption{$D$ bimodules}\label{table:bimodules}
\end{table}

\end{exam}

\subsection{Abelian semidirect product}
\begin{defn}\label{defn:direct sum algebra}
Let $D$ and $B$ be two dendriform algebras. A {\bf direct sum} of $D$ and $B$, denoted by $D\oplus B$, is the vector space $D\times B$, together with two bilinear maps defined as follows with $a,b\in D$ and $x,y\in B$,
\begin{align*}
(a,x)\succeq(b,y)&:=(a\succ b,x\succ_{B} y);\\
(a,x)\preceq(b,y)&:=(a\prec b,x\prec_{B} y).
\end{align*}
\end{defn}

\begin{exam}\label{exam:direc sum}
Suppose that $D$ and $B$ are dendriform algebras defined in Example ~\ref{exam:dendriform algebra}~\ref{exam:dendriform algebra1} and ~\ref{exam:dendriform algebra2}. The direct sum $D\oplus B$ is defined as follows:
\begin{align*}
\begin{aligned}
(e_{1},0)\succeq (e_{1},0)&=(e_{1},0);\\
(e_{1},0)\succeq (0,e_{2})&=(0,0);\\
(0,e_{2})\succeq (e_{1},0)&=(0,0);\\
(0,e_{2})\succeq (0,e_{2})&=(0,0);\\
\end{aligned}\,\,\,\,\,\,\,
\begin{aligned}
(e_{1},0)\preceq (e_{1},0)&=(0,0);\\
(e_{1},0)\preceq (0,e_{2})&=(0,0);\\
(0,e_{2})\preceq (e_{1},0)&=(0,0);\\
(0,e_{2})\preceq (0,e_{2})&=(0,e_{2}).
\end{aligned}
\end{align*}
\end{exam}

\begin{prop}\label{prop:direct sum}
The direct sum $D\oplus B:=(D\times B,\succeq,\preceq)$ is a dendriform algebra.
\end{prop}
\begin{proof} It is a direct computation.
\end{proof}
Das introduced semidirect products of dendriform algebras in \cite{D21}. In this paper, we call it abelian semidirect product, the word ``abelian" is to emphasize $V$ has a trivial dendriform algebra structure $(V,\succ_{V},\prec_{V})$, i.e., the maps $\succ_{V}$ and $\prec_{V}$ are trivial. In the sequel, the word ``nonabelian" means that $V$ has a dendriform algebra structure $(V,\succ_{V},\prec_{V})$ which may be not trivial.
\begin{defn}\cite{D21}\label{defn:abelian semidirect product}
Let $D$ be a dendriform algebra and $(V,\triangleright_{1},\triangleright_{2},\triangleleft_{1},\triangleleft_{2})$ a $D$ bimodule. An {\bf (abelian) semidirect product} of $D$ and $(V,\triangleright_{1},\triangleright_{2},\triangleleft_{1},\triangleleft_{2})$, denoted by $D\ltimes V$, is the vector space $D\times B$, together with two bilinear maps defined as follows with $a,b\in D$ and $x,y\in B$,
\begin{align*}
(a,x)\succeq(b,y)&:=(a\succ b,a\triangleright_{1} y+x\triangleleft_{1} b);\\
(a,x)\preceq(b,y)&:=(a\prec b,a\triangleright_{2} y+x\triangleleft_{2} b).
\end{align*}
\end{defn}
By direct computing, we have the following proposition.
\begin{prop}\cite{D21}\label{prop:weak semidirect product bimodule}
The abelian semidirect product $D\ltimes V:=(D\times V,\succeq,\preceq)$ is a dendriform algebra.
\end{prop}

\section{Extending structures for dendriform algebras}

In this section, we define extending datums and unified products of dendriform algebras, and theoretically classify the dendriform extending structures. Moreover, we define flag datums and give an example of the ES problem.

\subsection{Classification criteria for the ES problem}
In this subsection, we define two relations $\equiv$ and $\approx$ on the set of extensions $\mathbf{Exts}(E,D)$, for the classification of the ES problem.


\begin{defn}\label{defn:equivalent cohomolous}
Let $D$ be a dendriform algebra, and $E$ a vector space containing $D$ as a subspace. 
We denote by $\mathbf{Exts}(E,D)$ the set of all extensions $\cee$ of $D$. Suppose that $V$ is a space complement of $D$ in $E$,
An element of $E$ will be denoted by the form $a+x$ with $a\in D$ and $x\in V$.
In the following Diagram ~(\ref{diagram:equivalent cohomologous ee}), suppose that two extensions $(E,\succ_{E},\prec_{E}),(E,\succ_{E}^{\prime},\prec_{E}^{\prime})\in \mathbf{Exts}(E,D)$, the map $i$ and $\pi$ are defined by
 $$i(a)=a,\,\pi(a+x)=x,\,a\in D\, \text{and}\,x\in V.$$
Then two rows are short exact sequences in the category $\mathrm{Vect}_{\bf k}$ of vector spaces. We say that {\bf a linear map $\psi:E \rightarrow E$ stabilizes $D$ }if the left square is commutative. Similarly, we say that {\bf $\psi$ co-stabilizes $V$} if the right square is commutative.
\begin{align}\label{diagram:equivalent cohomologous ee}
\begin{split}
\xymatrix{
  0 \ar[r]^{} & D \ar[d]_{\mathrm {Id}} \ar[r]^{i\,\,\,\,\,\,\,\,\,\,\,\,\,\,\,\,\,\,} & (E,\succ_{E},\prec_{E}) \ar[d]_{\psi} \ar[r]^{\,\,\,\,\,\,\,\,\,\,\,\,\,\,\,\,\,\,\pi} & V \ar[d]_{\mathrm {Id}} \ar[r]^{} & 0  \\
  0 \ar[r]^{} & D \ar[r]^{i\,\,\,\,\,\,\,\,\,\,\,\,\,\,\,\,\,\,} & (E,\succ_{E}^{\prime},\prec_{E}^{\prime}) \ar[r]^{\,\,\,\,\,\,\,\,\,\,\,\,\,\,\,\,\,\,\pi} & V \ar[r]^{} & 0   }
\end{split}
\end{align}



\begin{enumerate}
\item The two extensions are called {\bf equivalent} if there exists an isomorphism of dendriform algebras $\psi:(E,\succ_{E},\prec_{E}) \rightarrow (E,\succ_{E}^{\prime},\prec_{E}^{\prime})$ which stabilizes $D$ in Diagram ~(\ref{diagram:equivalent cohomologous ee}). We denote it by $(E,\succ_{E},\prec_{E})\equiv (E,\succ_{E}^{\prime},\prec_{E}^{\prime})$.
\item The two extensions are called {\bf cohomologous} if there exists an isomorphism of dendriform algebras $\psi:(E,\succ_{E},\prec_{E}) \rightarrow (E,\succ_{E}^{\prime},\prec_{E}^{\prime})$ which stabilizes $D$ and co-stabilizes $V$ in Diagram ~(\ref{diagram:equivalent cohomologous ee}). We denote it by $(E,\succ_{E},\prec_{E})\approx (E,\succ_{E}^{\prime},\prec_{E}^{\prime})$.
\end{enumerate}
\end{defn}

\begin{rem}\label{rem:extd and extdp}
The relations $\equiv$ and $\approx$ defined on the set $\mathbf{Exts}(E,D)$ are both equivalence relations. By specifically computing $\mathbf{Exts}(E,D)$, we obtain a parameterization of the dendriform algebra structures that can be defined on $E$ containing $D$ as a subalgebra. Suppose that $\mathbf{Extd}(E,D):=\mathbf{Exts}(E,D)/\equiv$, then $\mathbf{Extd}(E,D)$ gives a classification of the $ES$ problem. The set $\mathbf{Extd^{\prime}}(E,D):=\mathbf{Exts}(E,D)/\approx$ is another more strictly classifying object for the ES problem. Since any two cohomologous dendriform algebra structures that defined on $E$ are obviously equivalent, then there exists a canonical projection $\mathbf{Extd^{\prime}}(E,D)\twoheadrightarrow \mathbf{Extd}(E,D)$.
 \end{rem}

\subsection{Unified products of dendriform algebras}

\allowdisplaybreaks{
\begin{defn}
Let $D$ be a dendriform algebra, and $V$ a vector space. An {\bf extending datum} of $D$ through $V$ is a system $\Omega(D,V)=(\triangleright_{1},\triangleright_{2},\triangleleft_{1},\triangleleft_{2},\rightharpoonup_{1},\rightharpoonup_{2},\leftharpoonup_{1},\leftharpoonup_{2},f_{1},f_{2},\succ_{V},\prec_{V})$ consisting of twelve bilinear maps:
\begin{align*}
\begin{split}
\triangleright_{1},\triangleright_{2}&:\,D\times V\rightarrow V;\\
\rightharpoonup_{1},\rightharpoonup_{2}&:\,V\times D\rightarrow D;\\
f_{1},f_{2}&:\,V\times V\rightarrow D;
\end{split}
\begin{split}
\triangleleft_{1},\triangleleft_{2}&:\,V\times D\rightarrow V;\\
\leftharpoonup_{1},\leftharpoonup_{2}&:\,D\times V\rightarrow D;\\
\succ_{V},\prec_{V}&:\,V\times V\rightarrow V.
\end{split}
\end{align*}
\end{defn}
}



\begin{defn}\label{defn:unified product}
    The {\bf unified product} of $D$ and $\Omega(D,V)$, denoted by $D\natural_{\Omega(D,V)} V=\dnv$, is the vector space $D\times V$, together with two bilinear maps:
\allowdisplaybreaks{
\begin {align*}
\succeq,\preceq&: (D\times V)\times (D\times V)\rightarrow (D\times V).
\end{align*}
defined as follows with $a,b\in D$ and $x,y\in V$,
\begin{align}\label{formulas:unified product}
\begin{split}
(a,x)\succeq(b,y)&:=(a\succ b+a\leftharpoonup_{1} y+x\rightharpoonup_{1} b+f_{1}(x,y),\,a\triangleright_{1} y+x\triangleleft_{1} b+x\succ_{V} y);\\
(a,x)\preceq(b,y)&:=(a\prec b+a\leftharpoonup_{2} y+x\rightharpoonup_{2} b+f_{2}(x,y),\,a\triangleright_{2} y+x\triangleleft_{2} b+x\prec_{V} y).
\end{split}
\end{align}}such that the triple $\big(D\times V,\succeq,\preceq\big)$ is a dendriform algebra. In this case, the extending datum $\Omega(D,V)$ is called a {\bf dendriform extending structure} of $D$ through $V$. The maps $\triangleright_{1},\triangleright_{2},\triangleleft_{1},\triangleleft_{2},
    \rightharpoonup_{1},\rightharpoonup_{2},\leftharpoonup_{1}$ and $\leftharpoonup_{2}$ are called {\bf actions} of $\Omega(D,V)$ and $f_{1},f_{2}$ are called {\bf cocycles} of $\Omega(D,V)$.
  \end{defn}
We denote by $\mathcal{O}(D,V)$ the set of all dendriform extending structures of $D$ through $V$.
For the sake of simplification of symbols, we give the following conventions:
 \begin{align}\label{formulas:dendriform to associative}
 \begin{aligned}
 \begin{split}
 \triangleright:=&\triangleright_{1}+\triangleright_{2};\\
 \rightharpoonup:=&\rightharpoonup_{1}+\rightharpoonup_{2};\\
 f:=&f_{1}+f_{2};\\
 \star_{E}:=&\succ_{E}+\prec_{E}\\
 \end{split}
 \end{aligned}
  \begin{aligned}
 \begin{split}
  \triangleleft:=&\triangleleft_{1}+\triangleleft_{2};\\
  \leftharpoonup:=&\leftharpoonup_{1}+\leftharpoonup_{2};\\
\star_{V}:=&\succ_{V}+\prec_{V};\\
\underline{\star}:=&\succeq+\preceq.
 \end{split}
 \end{aligned}
 \end{align}

\begin{thm}\label{thm:datum and unified product}
Let $D$ be a dendriform algebra, $V$ a vector space, and $\Omega(D,V)$ an extending datum of $D$ through $V$. Then the following statements are equivalent:
\begin{enumerate}
\item The object $\dnv$ is a unified product;
\item The following conditions hold for all $a,b\in D$ and $x,y,z\in V$,
 \allowdisplaybreaks{
      \begin{align*}
        (D1)\,\,&\text{$\big(V,\triangleright_{1},\triangleright_{2},\triangleleft_{1},\triangleleft_{2}\big)$ is a $D$ bimodule, i.e.,  satisfying {\rm Eqs. ~\eqref{formulas:left module definition}-\eqref{formulas:bimodule definition};}}\\
        (D2)\,\,&(a\star b)\leftharpoonup_{1} x=a\succ(b\leftharpoonup_{1} x)+a\leftharpoonup_{1}(b\triangleright_{1} x);\\
        &(a\prec b)\leftharpoonup_{2} x=a\prec (b\leftharpoonup x)+a\leftharpoonup_{2}(b\triangleright x);\\&(a\succ b)\leftharpoonup_{2} x=a\succ(b\leftharpoonup_{2} x)+a\leftharpoonup_{1}(b\triangleright_{2} x);\\
        (D3)\,\,&(x\rightharpoonup a)\succ b+(x\triangleleft a)\rightharpoonup_{1} b=x\rightharpoonup_{1}(a\succ b);\\
        &(x\rightharpoonup_{2} a)\prec b+(x\triangleleft_{2} a)\rightharpoonup_{2} b=x\rightharpoonup_{2}(a\star b);\\&(x\rightharpoonup_{1} a)\prec b+(x\triangleleft_{1} a)\rightharpoonup_{2} b=x\rightharpoonup_{1}(a\prec b);\\
        (D4)\,\,&(a\leftharpoonup x)\succ b+(a\triangleright x)\rightharpoonup_{1} b=a\succ(x\rightharpoonup_{1} b)+a\leftharpoonup_{1}(x\triangleleft_{1} b);\\
        &(a\leftharpoonup_{2} x)\prec b+(a\triangleright_{2} x)\rightharpoonup_{2} b=a\prec (x\rightharpoonup b)+a\leftharpoonup_{2}(x\triangleleft b);\\&(a\leftharpoonup_{1} x)\prec b+(a\triangleright_{1} x)\rightharpoonup_{2} b=a\succ(x\rightharpoonup_{2} b)+a\leftharpoonup_{1}(x\triangleleft_{2} b);\\
        (D5)\,\,&f(x,y)\succ a+(x\star_{V} y)\rightharpoonup_{1} a=x\rightharpoonup_{1}(y\rightharpoonup_{1} a)+f_{1}\big( x,y\triangleleft_{1} a\big);\\
        &f_{2}(x,y)\prec a+(x\prec_{V} y)\rightharpoonup_{2} a=x\rightharpoonup_{2}(y\rightharpoonup a)+f_{2}\big( x,y\triangleleft a\big);\\&f_{1}(x,y)\prec a+(x\succ_{V} y)\rightharpoonup_{2} a=x\rightharpoonup_{1}(y\rightharpoonup_{2} a)+f_{1}\big( x,y\triangleleft_{2} a\big);\\
        (D6)\,\,&(x\star_{V} y)\triangleleft_{1} a=x\triangleleft_{1}(y\rightharpoonup_{1} a)+x\succ_{V}(y\triangleleft_{1} a);\\
                &(x\prec_{V} y)\triangleleft_{2} a=x\triangleleft_{2}(y\rightharpoonup a)+x\prec_{V}(y\triangleleft a);\\
                &(x\succ_{V} y)\triangleleft_{2} a=x\triangleleft_{1}(y\rightharpoonup_{2} a)+x\succ_{V}(y\triangleleft_{2} a);\\
        (D7)\,\,&(a\leftharpoonup x)\leftharpoonup_{1} y+f_{1}\big(a\triangleright x, y\big)= a\succ f_{1}(x,y)+ a\leftharpoonup_{1} (x\succ_{V} y) ;\\
                &(a\leftharpoonup_{2} x)\leftharpoonup_{2} y+f_{2}\big(a\triangleright_{2} x, y\big)= a\prec f(x,y)+ a\leftharpoonup_{2} (x\star_{V} y) ;\\                        &(a\leftharpoonup_{1} x)\leftharpoonup_{2} y+f_{2}\big(a\triangleright_{1} x, y\big)= a\succ f_{2}(x,y)+ a\leftharpoonup_{1} (x\prec_{V} y) ;\\
        (D8)\,\,&(a\leftharpoonup x)\triangleright_{1} y+(a\triangleright x)\succ_{V} y= a\triangleright_{1} (x\succ_{V} y);\\
                &(a\leftharpoonup_{2} x)\triangleright_{2} y+(a\triangleright_{2} x)\prec_{V} y= a\triangleright_{2} (x\star_{V} y);\\
                &(a\leftharpoonup_{1} x)\triangleright_{2} y+(a\triangleright_{1} x)\prec_{V} y= a\triangleright_{1} (x\prec_{V} y);\\
        (D9)\,\,&(x\rightharpoonup a)\leftharpoonup_{1} y+f_{1}\big(x\triangleleft a,y\big)=x\rightharpoonup_{1}(a\leftharpoonup_{1} y)+f_{1}\big(x,a\triangleright_{1} y\big);\\
               &(x\rightharpoonup_{2} a)\leftharpoonup_{2} y+f_{2}\big(x\triangleleft_{2} a,y\big)=x\rightharpoonup_{2}(a\leftharpoonup y)+f_{2}\big(x,a\triangleright y\big);\\
        &(x\rightharpoonup_{1} a)\leftharpoonup_{2} y+f_{2}\big(x\triangleleft_{1} a,y\big)=x\rightharpoonup_{1}(a\leftharpoonup_{2} y)+f_{1}\big(x,a\triangleright_{2} y\big);\\
        (D10)\,&(x\rightharpoonup a)\triangleright_{1} y+(x\triangleleft a)\succ_{V} y=x\triangleleft_{1}(a\leftharpoonup_{1} y)+ x\succ_{V}(a\triangleright_{1} y);\\
        &(x\rightharpoonup_{2} a)\triangleright_{2} y+(x\triangleleft_{2} a)\prec_{V} y=x\triangleleft_{2}(a\leftharpoonup y)+ x\prec_{V}(a\triangleright y);\\
        &(x\rightharpoonup_{1} a)\triangleright_{2} y+(x\triangleleft_{1} a)\prec_{V} y=x\triangleleft_{1}(a\leftharpoonup_{2} y)+ x\succ_{V}(a\triangleright_{2} y);\\
        (D11)\,&f(x,y)\leftharpoonup_{1} z+f_{1}\big(x\star_{V} y,z\big)=x\rightharpoonup_{1} f_{1}(y,z)+f_{1}\big(x,y\succ_{V} z\big);\\
        &f_{2}(x,y)\leftharpoonup_{2} z+f_{2}\big(x\prec_{V} y,z\big)=x\rightharpoonup_{2} f(y,z)+f_{2}\big(x,y\star_{V} z\big);\\
        &f_{1}(x,y)\leftharpoonup_{2} z+f_{2}\big(x\succ_{V} y,z\big)=x\rightharpoonup_{1} f_{2}(y,z)+f_{1}\big(x,y\prec_{V} z\big);\\
        (D12)\,&f(x,y)\triangleright_{1} z+(x\star_{V} y)\succ_{V} z=x\triangleleft_{1} f_{1}(y,z)+x\succ_{V} (y\succ_{V} z); \\
        &f_{2}(x,y)\triangleright_{2} z+(x\prec_{V} y)\prec_{V} z=x\triangleleft_{2} f(y,z)+x\prec_{V} (y\star_{V} z);\\
        &f_{1}(x,y)\triangleright_{2} z+(x\succ_{V} y)\prec_{V} z=x\triangleleft_{1} f_{2}(y,z)+x\succ_{V} (y\prec_{V} z).
    \end{align*}
    }
  \end{enumerate}
    \end{thm}

\begin{proof} The object $\dnv$ is a unified product if and only if $\Omega(D,V)$ is an extending datum such that the following conditions hold for all $a,b,c\in D$ and $x,y,z\in V$,
       \allowdisplaybreaks{
        \begin{align}
        \begin{split}\label{zformula(not):unified product}
    \big((a,x)\underline{\star} (b,y)\big)&\succeq (c,z)=(a,x)\succeq \big((b,y)\succeq (c,z)\big);\\
    \big((a,x)\preceq (b,y)\big)&\preceq (c,z)=(a,x)\preceq \big((b,y)\underline{\star} (c,z)\big);\\
    \big((a,x)\succeq (b,y)\big)&\preceq (c,z)=(a,x)\succeq \big((b,y)\preceq (c,z)\big).
    \end{split}
    \end{align}

Since $\dnv$ is a direct sum of $D$ and $V$ as vector spaces, Eq. ~(\ref{zformula(not):unified product}) holds if and only if it holds for all generators of $\dnv$, i.e., for the set $\{(a,0)|a\in D\}\cup \{(0,x)|x\in V\}$. By Eq. ~(\ref{formulas:unified product}), we have
 \begin{align*}
0=&\big((a,0)\underline{\star} (b,0)\big)\succeq (0,x)-(a,0)\succeq \big((b,0)\succeq (0,x)\big)\\
=&(a\star b,0)\succeq \big(0,x)-(a,0)\succeq (b\leftharpoonup_{1} x,b\triangleright_{1} x)\\
=&\big((a\star b)\leftharpoonup_{1} x,(a\star b)\triangleright_{1} x\big)
-\big(a\succ(b\leftharpoonup_{1} x)+a\leftharpoonup_{1}(b\triangleright_{1} x),a\triangleright_{1}(b\triangleright_{1} x)\big);\\
0=&\big((a,0)\preceq (b,0)\big)\preceq (0,x)-(a,0)\preceq \big((b,0)\underline{\star} (0,x)\big)\\
=&(a\prec b,0)\preceq \big(0,x)-(a,0)\succeq (b\leftharpoonup x,b\triangleright x)\\
=&\big((a\prec b)\leftharpoonup_{1} x,(a\prec b)\triangleright_{2} x\big)
-\big(a\prec (b\leftharpoonup x)+a\leftharpoonup_{2}(b\triangleright x),a\triangleright_{2}(b\triangleright x)\big);\\
0=&\big((a,0)\succeq (b,0)\big)\preceq (0,x)-(a,0)\succeq \big((b,0)\preceq (0,x)\big)\\
=&(a\succ b,0)\preceq \big(0,x)-(a,0)\succeq (b\leftharpoonup_{2} x,b\triangleright_{2} x)\\
=&\big((a\succ b)\leftharpoonup_{2} x,(a\succ b)\triangleright_{2} x\big)
-\big(a\succ(b\leftharpoonup_{2} x)+a\leftharpoonup_{1}(b\triangleright_{2} x),a\triangleright_{1}(b\triangleright_{2} x)\big).
    \end{align*}

Hence, (D2) and Eq. ~\eqref{formulas:left module definition} hold if and only if Eq. ~(\ref{zformula(not):unified product}) holds for the triple $(a,0)$, $(b,0)$, $(0,x)$ with $a,b\in D$ and $x\in V$.

 Similarly, we can obtain: (D3) and Eq. ~\eqref{formulas:right module definition} hold if and only if Eq. ~(\ref{zformula(not):unified product}) holds for the triple $(0,x)$, $(a,0)$, $(b,0)$ with $a,b\in D$ and $x\in V$; (D4) and Eq.~\eqref{formulas:bimodule definition} hold if and only if Eq. ~(\ref{zformula(not):unified product}) holds for the triple $(a,0)$, $(0,x)$, $(b,0)$ with $a,b\in D$ and $x\in V$; (D5) and (D6) hold if and only if Eq. ~(\ref{zformula(not):unified product}) holds for the triple $(0,x)$, $(0,y)$, $(a,0)$ with $a\in D$ and $x,y\in V$; (D7) and (D8) hold if and only if Eq. ~(\ref{zformula(not):unified product}) holds for the triple $(a,0)$, $(0,x)$, $(0,y)$ with $a\in D$ and $x,y\in V$; (D9) and (D10) hold if and only if Eq. ~(\ref{zformula(not):unified product}) holds for the triple $(0,x)$, $(a,0)$, $(0,y)$ with $a\in D$ and $x,y\in V$; (D11) and (D12) hold if and only if Eq. ~(\ref{zformula(not):unified product}) holds for the triple $(0,x)$, $(0,y)$, $(0,z)$ with $x,y,z\in V$.
}
 \end{proof}



%
\begin{rem}\label{rem:relations of extending structure}
\begin{enumerate}
\item By Theorem ~\ref{thm:datum and unified product}, a dendriform extending structure of $D$ through $V$ can be viewed as an extending datum $\Omega(D,V)$ satisfying the conditions (D1)-(D12).

\item By Remark~\ref{rem:asso}~\ref{asso}, we know that $(D,\star)$ is an associative algebra. By Theorem ~\ref{thm:datum and unified product}, then $\big(\triangleright,\triangleleft,\rightharpoonup,\leftharpoonup,f,\star_{V}\big)$ is an associative algebra extending structure of $(D,\star)$ through $V$, more details we refer to ~\cite{AM16}.
\item By Remark~\ref{rem:asso}~\ref{preLie}, we know that $(D,\diamond)$ is a preLie algebra. Define
\begin{align}\label{formulas:dendriform to prelie}
 \begin{aligned}
 \begin{split}
 a\vdash x:=&a\triangleright_{1} x-x\triangleleft_{2} a;\\
 x\gg a:=&x\rightharpoonup_{1} a-a \leftharpoonup_{2} x;\\
 \tilde{f}(x,y):=&f_{1}(x,y)-f_{2}(y,x);\\
 \end{split}
 \end{aligned}
  \begin{aligned}
 \begin{split}
  x\dashv a:=&x\triangleleft_{1} a-a\triangleright_{2} x;\\
  a\ll x:=&a \leftharpoonup_{1} x-x\rightharpoonup_{2} a;\\
x\diamond_{V} y:=&x\succ_{V}y-y\prec_{V}x,\,a\in D,x,y\in V.
 \end{split}
 \end{aligned}
 \end{align}
By Theorem ~\ref{thm:datum and unified product}, then $\big(\vdash,\dashv,\gg,\ll,\tilde{f},\diamond_{V}\big)$ is a preLie algebra extending structure of $(D,\diamond)$ through $V$, more details we refer to ~\cite{H19b}.

\item Define $[x,y]:=x\star y-y\star x=x\succ y+x\prec y-y\succ x-y\prec x$, we know that $(D,[,])$ is a Lie algebra in ~\cite{H12}. For all $a\in D$, $x,y\in V$, define
 \begin{align}\label{formulas:asso to lie}
 \begin{split}
 x\blacktriangleleft a:=&x\triangleleft a-a\triangleright x=x\triangleleft_{1}a+x\triangleleft_{2}a-a\triangleright_{1}x-a\triangleright_{2}x;\\
 x \blacktriangleright a:=&x\rightharpoonup a-a\leftharpoonup x=x\rightharpoonup_{1}a+x\rightharpoonup_{2}a-a\leftharpoonup_{1}x-a\leftharpoonup_{2}x;\\
\hat{f}(x,y):=&f(x,y)-f(y,x)=f_{1}(x,y)+f_{2}(x,y)-f_{1}(y,x)-f_{2}(y,x);\\
\{x,y\}:=&x\star_{V} y-y\star_{V} x=x\succ_{V} y+x\prec_{V} y-y\succ_{V} x-y\prec_{V} x.
\end{split}
 \end{align}
By Theorem ~\ref{thm:datum and unified product}, then $\big(\blacktriangleright,\blacktriangleleft,\hat{f},\{,\}\big)$ is a Lie algebra extending structure of $(D,[,])$ through $V$, more details we refer to ~\cite{AM14L}.
\item Define $[x,y]:=x\diamond y-y\diamond x=x\succ y+x\prec y-y\succ x-y\prec x$, we know that $(D,[,])$ is a Lie algebra in ~\cite{Agu002}. For all $a\in D$, $x,y\in V$, define
 \begin{align}\label{formulas:preLie to lie}
 \begin{split}
 x\blacktriangleleft a:=&a\vdash x-x\dashv a=x\triangleleft_{1}a+x\triangleleft_{2}a-a\triangleright_{1}x-a\triangleright_{2}x;\\
 x \blacktriangleright a:=&x \gg a-a\ll x=x\rightharpoonup_{1}a+x\rightharpoonup_{2}a-a\leftharpoonup_{1}x-a\leftharpoonup_{2}x;\\
\hat{f}(x,y):=&\tilde{f}(x,y)-\tilde{f}(y,x)=f_{1}(x,y)+f_{2}(x,y)-f_{1}(y,x)-f_{2}(y,x);\\
\{x,y\}:=&x\diamond_{V} y-y\diamond_{V} x=x\succ_{V} y+x\prec_{V} y-y\succ_{V} x-y\prec_{V} x.
\end{split}
 \end{align}
By Theorem ~\ref{thm:datum and unified product}, then $\big(\blacktriangleright,\blacktriangleleft,\hat{f},\{,\}\big)$ is a Lie algebra extending structure of $(D,[,])$ through $V$, more details we refer to ~\cite{AM14L,H19b}.
\end{enumerate}
 \end{rem}
{\em Note:} In the rest of this paper, we will delete the trivial maps of $\Omega(D,V)$ for simplicity. For example, if the maps $\rightharpoonup_{1},\rightharpoonup_{2},\leftharpoonup_{1}$ and $\leftharpoonup_{2}$ are all trivial, then the extending datum will be denoted by $\Omega(D,V)=\big(\triangleright_{1},\triangleright_{2},\triangleleft_{1},\triangleleft_{2},f_{1},f_{2},\succ_{V},\prec_{V})$ for simplicity. Now we will give some special cases of unified products as examples.
\begin{exam}\label{exam:direct sum and unified}
The extending datum $\Omega(D,V)=(\succ_{V},\prec_{V})$ is a dendriform extending structure if and only if the triple $(V,\succ_{V},\prec_{V})$ is a dendriform algebra
in Theorem ~\ref{thm:datum and unified product}. In this case, the
corresponding unified product is a direct sum defined in Definition ~\ref{defn:direct sum algebra}
and Proposition ~\ref{prop:direct sum} is a corollary of Theorem ~\ref{thm:datum and unified product}.
\end{exam}

\begin{exam}\label{exam:weak semidirect product and unified}
The extending datum $\Omega(D,V)=\big(\triangleright_{1},\triangleright_{2},\triangleleft_{1},\triangleleft_{2}\big)$ is a dendriform extending structure if and only if the 5-tuple $(V,\triangleright_{1},\triangleright_{2},\triangleleft_{1},\triangleleft_{2})$ is a $D$ bimodule in Theorem ~\ref{thm:datum and unified product}. In this case, the
corresponding unified product is an abelian semidirect product defined in Definition ~\ref{defn:abelian semidirect product} and Proposition ~\ref{prop:weak semidirect product bimodule} is a corollary of Theorem ~\ref{thm:datum and unified product}.
\end{exam}

\subsection{The relations between dendriform extending structures and extensions of dendriform algebras}
In this subsection, we establish a bijection between the set of dendriform extending structures $\mathcal{O}(D,V)$ and the set of extensions $\mathbf{Exts}(E,D)$.

Let $D$ be a dendriform algebra, $E$ a vector space containing $D$ as a subspace and $V$ a space complement of $D$ in $E$. Suppose that dendriform extending structure $\Omega(D,V)\in \mathcal{O}(D,V)$, and $\dnv$ is the corresponding unified product. The two rows in the following Diagram ~(\ref{diagram:equivalent cohomologous ue}) are short exact sequences in the category $\mathrm{Vect}_{\bf k}$, where $i$ and $\pi$ are defined in Diagram ~(\ref{diagram:equivalent cohomologous ee}), $i^{\prime}$ and $\pi^{\prime}$ are defined by $$i^{\prime}(a)=(a,0),\quad \pi^{\prime}(a,x)=x,\, a\in D\, \text{and}\, x\in V.$$
 Similar to Definition~\ref{defn:equivalent cohomolous},
we define the relation $\equiv$ and $\approx$ between the unified product $\dnv$ and the extension $(E,\succ_{E},\prec_{E})\in \mathbf{Exts}(E,D)$ in the following Diagram ~(\ref{diagram:equivalent cohomologous ue}), and denoted by $\dnv\equiv (E,\succ_{E},\prec_{E})$ and $\dnv\approx (E,\succ_{E},\prec_{E})$ respectively.
\begin{align}\label{diagram:equivalent cohomologous ue}
\begin{split}
\xymatrix{
  0 \ar[r]^{} & D \ar[d]_{\mathrm {Id}} \ar[r]^{i^{\prime}\,\,\,\,\,} & \dnv \ar[d]_{\varphi} \ar[r]^{\,\,\,\,\,\pi^{\prime}} & V \ar[d]_{\mathrm {Id}} \ar[r]^{} & 0  \\
  0 \ar[r]^{} & D \ar[r]^{i\,\,\,\,\,\,\,\,\,\,\,\,\,\,\,\,\,\,} & (E,\succ_{E},\prec_{E}) \ar[r]^{\,\,\,\,\,\,\,\,\,\,\,\,\,\,\,\,\,\,\pi} & V \ar[r]^{} & 0   }\,\,\,\,\,\,\,\,\,\,\,\,
\end{split}
\end{align}

\begin{thm}\label{thm:extension to unified}
Let $D\subset E $ be an extension of dendriform algebras. Then there exists a dendriform extending structure $\Omega(D,V)$ of $D$ through $V$, such that $\dnv\approx \cee$ in {\rm Diagram ~(\ref{diagram:equivalent cohomologous ue})}.
\end{thm}

\begin{proof}
 In the category $\mathrm{Vect}_{\mathbf{\bf k}}$, given an inclusion map $i:D\rightarrow E$, there exists a linear retraction $\rho:E\rightarrow D$, i.e., there exists a linear map $\rho:E\rightarrow D$, such that $\rho\, i=\mathrm{Id}_{D}$. Then $V:=\mathrm{ker}\,\rho$ is a space complement of $D$ in $E$, and we get a short exact sequence in Diagram ~(\ref{diagram:equivalent cohomologous ue}).
Let's list some useful facts: $\succ_{E}|_{D}=\succ$ and $\prec_{E}|_{D}=\prec$, $\rho(a+x)=a$, $a\in D,\,x\in V$.

First, we define the extending datum $\Omega(D,V)=\big(\triangleright_{1},\triangleright_{2},\triangleleft_{1},\triangleleft_{2},
\rightharpoonup_{1},\rightharpoonup_{2},\leftharpoonup_{1},\leftharpoonup_{2},f_{1},f_{2},\succ_{V},\prec_{V})$ of $D$ through $V$ as follows with $a\in D$ and $x,y\in V$,
\allowdisplaybreaks{
      \begin{align}\label{formulas:extension to unified}
      \begin{split}
        a\triangleright_{1} x&:=a\succ_{E} x-\rho(a\succ_{E} x);\\
        x\triangleleft_{1} a&:=x\succ_{E} a-\rho(x\succ_{E} a);\\
        a\leftharpoonup_{1} x&:=\rho(a\succ_{E} x);\\
        x\rightharpoonup_{1} a&:=\rho(x\succ_{E} a);\\
        f_{1}(x,y)&:=\rho(x\succ_{E} y);\\
        x\succ_{V} y&:=x\succ_{E} y-\rho(x\succ_{E} y);\\
       \end{split}
             \begin{split}
        a\triangleright_{2} x&:=a\prec_{E} x-\rho(a\prec_{E} x);\\
        x\triangleleft_{2} a&:=x\prec_{E} a-\rho(x\prec_{E} a);\\
        a\leftharpoonup_{2} x&:=\rho(a\prec_{E} x);\\
        x\rightharpoonup_{2} a&:=\rho(x\prec_{E} a);\\
        f_{2}(x,y)&:=\rho(x\prec_{E} y);\\
        x\prec_{V} y&:=x\prec_{E} y-\rho(x\prec_{E} y).
       \end{split}
        \end{align}
        }
Second, we prove that the extending datum $\Omega(D,V)$ is a dendriform extending structure. By Theorem~\ref{thm:datum and unified product}, we only need prove that $\dnv$ is a unified product by defining a linear map $\varphi$ as follows:
\begin{align}\label{formulas:map eu}
\begin{split}
\varphi:\,\,\dnv&\rightarrow \cee\\
(a,x)&\mapsto a+x,a \in D,x\in V.
\end{split}
\end{align}

Since $E=D+V$, $D\cap V=0$, suppose $\varphi(a,x)=a+x=0$, $a\in D,x\in V$, we have $a=0,x=0$, therefore $\varphi$ is injective. For any $u\in E$, $\varphi\big(\rho(u),u-\rho(u)\big)=u$, therefore $\varphi$ is surjective. Hence, $\varphi$ is a linear isomorphism. By Eq.~(\ref{formulas:unified product}), for all $a,b\in D$ and $x,y \in V$, we have
\allowdisplaybreaks{
\begin{align}
\varphi\big((a,x)\succeq(b,y)\big)
&=\varphi\big(a\succ b+a\leftharpoonup_{1} y+x\rightharpoonup_{1} b+f(x,y),a\triangleright_{1} y+x\triangleleft_{1} b+x\succ_{V} y\big)\notag\\
&=a\succ b+a\leftharpoonup_{1} y+x\rightharpoonup_{1} b+f_{1}(x,y)+a\triangleright_{1} y+x\triangleleft_{1} b+x\succ_{V} y\notag\\
&=a\succ b+(a\leftharpoonup_{1} y+a\triangleright_{1} y)+(x\rightharpoonup_{1} b+x\triangleleft_{1} b)+\big(f_{1}(x,y)+x\succ_{V} y\big)\notag\\
&=a\succ_{E} b+a\succ_{E} y+x\succ_{E} b+x\succ_{E} y\notag\\
&=(a+x)\succ_{E}(b+y)\notag\\
&=\varphi(a,x)\succ_{E}\varphi(b,y);\label{formulas:preserve multiplication1}\\
\varphi\big((a,x)\preceq(b,y)\big)
&=\varphi\big(a\prec b+a\leftharpoonup_{2} y+x\rightharpoonup_{2} b+f_{2}(x,y),a\triangleright_{2} y+x\triangleleft_{2} b+x\prec_{V} y\big)\notag\\
&=a\prec b+a\leftharpoonup_{2} y+x\rightharpoonup_{2} b+f_{2}(x,y)+a\triangleright_{2} y+x\triangleleft_{2} b+x\prec_{V} y\notag\\
&=a\prec b+(a\leftharpoonup_{2} y+a\triangleright_{2} y)+(x\rightharpoonup_{2} b+x\triangleleft_{2} b)+\big(f_{2}(x,y)+x\prec_{V} y\big)\notag\\
&=a\prec_{E} b+a\prec_{E} y+x\prec_{E} b+x\prec_{E} y\notag\\
&=(a+x)\prec_{E}(b+y)\notag\\
&=\varphi(a,x)\prec_{E}\varphi(b,y).\label{formulas:preserve multiplication2}
\end{align}
}
Since the map $\varphi$ is a linear isomorphism, there exists a linear inverse $\varphi^{-1}$, then we have
\begin{align*}
    \big((a,x)\underline{\star} (b,y)\big)\succeq (c,z)
    =&\varphi^{-1}\varphi\Big(\big((a,x)\underline{\star} (b,y)\big)\succeq (c,z)\Big)\\
    =&\varphi^{-1}\Big(\big((a+x)\star_{E} (b+y)\big)\succ_{E}(c+z)\Big)\quad\text{(by Eqs.~(\ref{formulas:preserve multiplication1}) and ~(\ref{formulas:preserve multiplication2}))}\\
    =&\varphi^{-1}\Big((a+x)\succ_{E} \big((b+y)\succ_{E} (c+z)\big)\Big)
    \quad\text{(by $E$ being a dendriform algebra.})\\
    =&\varphi^{-1}\varphi\Big((a,x)\succeq \big((b,y)\succeq (c,z)\big)\Big)
       \quad\text{(by Eqs.~(\ref{formulas:preserve multiplication1}))}\\
    =&(a,x)\succeq \big((b,y)\succeq (c,z)\big).
    \end{align*}
Analogously we can prove,
\begin{align*}
    \big((a,x)\preceq (b,y)\big)&\preceq (c,z)=(a,x)\preceq \big((b,y)\underline{\star} (c,z)\big);\\
    \big((a,x)\succeq (b,y)\big)&\preceq (c,z)=(a,x)\succeq \big((b,y)\preceq (c,z)\big).
    \end{align*}
Hence, $\dnv$ is a unified product of $D$ and $\Omega(D,V)$.
Moreover, the linear isomorphism $\varphi$ is an isomorphism of dendriform algebras by Eqs.~(\ref{formulas:preserve multiplication1}) and ~(\ref{formulas:preserve multiplication2}).

Finally, we prove that the map $\varphi$ stabilizes $D$ and co-stabilizes $V$ in Diagram ~(\ref{diagram:equivalent cohomologous ue}). Since $$\varphi\big(i^{\prime}(a)\big)=\varphi(a,0)=a=i(a),a\in D,$$
we obtain $\varphi$ stabilizes $D$. Moreover, we have
 $$\pi\big(\varphi(a,x)\big)=\pi(a+x)=x=\pi^{\prime}(a,x),a\in D,x\in V,$$
so $\varphi$ co-stabilizes $V$. This completes the proof.
\end{proof}

By the proof of Theorem~\ref{thm:extension to unified}, we give the following remark.
\begin{rem}\label{remark:retraction and V}
Let $D\subset E $ be an extension of dendriform algebras.
\begin{enumerate}
\item \label{remark:retraction and V1}There exists a one-to-one correspondence between the linear retraction $\rho:E\rightarrow D$ and $V$. In detail, when the linear retraction $\rho:E\rightarrow D$ is given, we immediately obtain $V:=\mathrm{ker}\,\rho$ is a space complement of $D$ in $E$. Conversely, when $V$ is a space complement of $D$ in $E$, define a linear retraction $\rho:E\rightarrow D$, $\rho(a+x)=a$, for all $a\in D,x\in V$, we call it {\bf the retraction associated to $V$}.

\item\label{rem:eu map} Fix a linear retraction $\rho:E\rightarrow D$ and $V:=\mathrm{ker}\,\rho$. Now we establish a map as follows:
\begin{align*}
\Upsilon_{1}:\mathbf{Exts}(E,D)&\rightarrow \mathcal{O}(D,V)\\
\cee&\mapsto \Omega(D,V),
\end{align*}
where the dendriform extending structure $\Omega(D,V)$ is defined by Eq.~(\ref{formulas:extension to unified}).
\end{enumerate}
\end{rem}
In the sequel, the dendriform extending structure $(\triangleright_{1}^{\prime},\triangleright_{2}^{\prime},\triangleleft_{1}^{\prime},\triangleleft_{2}^{\prime},\rightharpoonup_{1}^{\prime},\rightharpoonup_{2}^{\prime},\leftharpoonup_{1}^{\prime},\leftharpoonup_{2}^{\prime},f_{1}^{\prime},f_{2}^{\prime},\succ_{V}^{\prime},\prec_{V}^{\prime})$ will be denoted by $\Omega^{\prime}(D,V)$ for convenience.
\begin{prop}\label{prop:bijection of extension and unified product}
Let $D\subset E $ be an extension of dendriform algebras with a linear retraction $\rho:E\rightarrow D$ and $V:=\mathrm{ker}\,\rho$. Then the map $\Upsilon_{1}:\mathbf{Exts}(E,D)\rightarrow \mathcal{O}(D,V)$ defined by {\rm Remark ~\ref{remark:retraction and V} ~\ref{rem:eu map}} is a bijection.
\end{prop}
\begin{proof}
First, we define a map as follows: 
\begin{align*}
\Upsilon_{2}:\mathcal{O}(D,V)&\rightarrow \mathbf{Exts}(E,D)\\
\Omega(D,V)&\mapsto \cee,
\end{align*}
where the maps $\succ_{E}$ and $\prec_{E}$ are defined as follows with $a,b\in D$ and $x,y\in V$,
\begin{align}\label{formulas:unified to extension}
\begin{split}
(a+x)\succ_{E}(b+y)&:=a\succ b+a\leftharpoonup_{1} y+x\rightharpoonup_{1} b+f_{1}(x,y)+a\triangleright_{1} y+x\triangleleft_{1} b+x\succ_{V} y;\\
(a+x)\prec_{E}(b+y)&:=a\prec b+a\leftharpoonup_{2} y+x\rightharpoonup_{2} b+f_{2}(x,y)+a\triangleright_{2} y+x\triangleleft_{2} b+x\prec_{V} y.
\end{split}
\end{align}
Similar to Theorem ~\ref{thm:datum and unified product}, we can prove that the triple $\cee$ is an extension of $D$, hence the map $\Upsilon_{2}$ is well defined.

Second, we consider the composition map $\Upsilon_{2}\circ\Upsilon_{1}$,
\begin{align*}
\Upsilon_{2}\circ\Upsilon_{1}:\mathbf{Exts}(E,D)\rightarrow\mathcal{O}(D,V)&\rightarrow \mathbf{Exts}(E,D)\\
\cee\mapsto \Omega(D,V)&\mapsto \ceep.
\end{align*}
By substituting Eq. ~(\ref{formulas:extension to unified}) to Eq. ~(\ref{formulas:unified to extension}), we immediately obtain $\succ_{E}^{\prime}=\succ_{E},\prec_{E}^{\prime}=\prec_{E}$, then $\Upsilon_{2}\circ\Upsilon_{1}=\mathrm{Id}_{\mathbf{Exts}(E,D)}$.

Finally, we consider the composition map $\Upsilon_{1}\circ\Upsilon_{2}$,
\begin{align*}
\Upsilon_{1}\circ\Upsilon_{2}:\mathcal{O}(D,V)&\rightarrow \mathbf{Exts}(E,D)\rightarrow \mathcal{O}(D,V)\\
\Omega(D,V)&\mapsto \cee\mapsto \Omega^{\prime}(D,V).
\end{align*}
By substituting Eq. ~(\ref{formulas:unified to extension}) to Eq. ~(\ref{formulas:extension to unified}), we obtain $\Omega^{\prime}(D,V)=\Omega(D,V)$, then $\Upsilon_{1}\circ\Upsilon_{2}=\mathrm{Id}_{\mathcal{O}(D,V)}$.
\end{proof}


\subsection{The classification of dendriform extending structures}
In this subsection, we define two relations $\equiv$ and $\approx$ on the set of dendriform extending structures $\mathcal{O}(D,V)$ and obtain the main result of classifying the dendriform extending structures.

\begin{defn}\label{defn:datum equivalent}
Let $D$ be a dendriform algebra, and $V$ a vector space. Two dendriform extending structures $\Omega(D,V)$ and $\Omega^{\prime}(D,V)$ are called {\bf equivalent}, and we denote it by $\Omega(D,V)\equiv\Omega^{\prime}(D,V)$, if there exists a pair of linear maps $(g,h)$, where $g:V\rightarrow D$ and $h\in \mathrm{Aut}_{{\bf k}}(V)$, satisfying the following conditions with $a\in D$ and $x,y\in V$,
\allowdisplaybreaks{
\begin{align*}
        (E1)\,\,&\text{The map $h$ is a $D$ bimodule morphism, i.e., }\\
                    &\begin{aligned}
                    \begin{split}
                    &(E1.1)\,\,h(a\triangleright_{1} x)=a\triangleright^{\prime}_{1} h(x);\\
                    &(E1.2)\,\,h(x\triangleleft_{1} a)=h(x)\triangleleft^{\prime}_{1} a;\\
                    \end{split}
                    \begin{split}
                    h(a\triangleright_{2} x)=a\triangleright^{\prime}_{2} h(x);\\
                    h(x\triangleleft_{2} a)=h(x)\triangleleft^{\prime}_{2} a.\\
                    \end{split}
                    \end{aligned}\\
        (E2)\,\,&x\rightharpoonup_{1} a+g(x\triangleleft_{1} a)=g(x)\succ a+h(x)\rightharpoonup^{\prime}_{1} a ;\\
        &x\rightharpoonup_{2} a+g(x\triangleleft_{2} a)=g(x)\prec\,a+h(x)\rightharpoonup^{\prime}_{2} a ;\\
        (E3)\,\,&a\leftharpoonup_{1} x+g(a\triangleright_{1} x)=a\succ g(x)+a\leftharpoonup^{\prime}_{1}h(x);\\
        &a\leftharpoonup_{2} x+g(a\triangleright_{2} x)=a\prec\,g(x)+a\leftharpoonup^{\prime}_{2}h(x);\\
        (E4)\,\,&f_{1}(x,y)+g(x\succ_{V} y)=g(x)\succ g(y)+g(x)\leftharpoonup^{\prime}_{1} h(y)+h(x)\rightharpoonup^{\prime}_{1} g(y)+f^{\prime}_{1}\big(h(x),h(y)\big);\\
        &f_{2}(x,y)+g(x\prec_{V} y)=g(x)\prec\,g(y)+g(x)\leftharpoonup^{\prime}_{2} h(y)+h(x)\rightharpoonup^{\prime}_{2} g(y)+f^{\prime}_{2}\big(h(x),h(y)\big);\\
        (E5)\,\,&h(x\succ_{V} y)=g(x)\triangleright^{\prime}_{1} h(y)+h(x)\triangleleft^{\prime}_{1} g(y)+h(x)\succ_{V}^{\prime} h(y).\\
        &h(x\prec_{V} y)=g(x)\triangleright^{\prime}_{2} h(y)+h(x)\triangleleft^{\prime}_{2} g(y)+h(x)\prec_{V}^{\prime} h(y).
\end{align*}
}
Moreover, two equivalent dendriform extending structures are called {\bf cohomologous} if $h=\mathrm{Id}_{V}$, and denoted by $\Omega(D,V)\approx\Omega^{\prime}(D,V)$.
\end{defn}

Let $D$ be a dendriform algebra, and $V$ a vector space. Suppose that two dendriform extending structures $\Omega(D,V),\Omega^{\prime}(D,V)\in \mathcal{O}(D,V)$, and $\dnv$, $\dnvp$ are corresponding unified products. Two rows in the following Diagram ~(\ref{diagram:equivalent cohomologous uu}) are short exact sequences in the category $\mathrm{Vect}_{\bf k}$, where $i^{\prime}$ and $\pi^{\prime}$ are defined in Diagram ~(\ref{diagram:equivalent cohomologous ue}). Similar to Definition~\ref{defn:equivalent cohomolous}, we define the relation $\equiv$ and $\approx$ between these two unified products $\dnv$ and $\dnvp$ in the following Diagram ~(\ref{diagram:equivalent cohomologous uu}), and denoted by $\dnv\equiv \dnvp$ and $\dnv\approx \dnvp$ respectively.
\begin{align}\label{diagram:equivalent cohomologous uu}
\begin{split}
\xymatrix{
  0 \ar[r]^{} & D \ar[d]_{\mathrm {Id}} \ar[r]^{i^{\prime}\,\,\,\,\,} & \dnv \ar[d]_{\phi} \ar[r]^{\,\,\,\,\,\pi^{\prime}} & V \ar[d]_{\mathrm {Id}} \ar[r]^{} & 0  \\
  0 \ar[r]^{} & D \ar[r]^{i^{\prime}\,\,\,\,\,} & \dnvp \ar[r]^{\,\,\,\,\,\pi^{\prime}} & V \ar[r]^{} & 0   }
\end{split}
\end{align}
\begin{lem}\label{lem:datum equivalent cohomologous}
Let $D$ be a dendriform algebra, and $V$ a vector space. Suppose that $\Omega(D,V)$ and $\Omega^{\prime}(D,V)$ are two dendriform extending structures of $D$ through $V$, and $\dnv$, $\dnvp$ are the corresponding unified products. We denote by $\mathcal{M}$ the set of all morphisms of dendriform algebras $\phi:\dnv\rightarrow \dnvp$ which stabilizes $D$ in {\rm Diagram ~(\ref{diagram:equivalent cohomologous uu})} and denote by $\mathcal{N}$ the set of all pairs $(g,h)$, where the linear maps $g:V\rightarrow D$ and $h:V\rightarrow V$ satisfies the conditions {\rm (E1)-(E5)} in {\rm Definition ~\ref{defn:datum equivalent}}. Then
\begin{enumerate}
 \item \label{It:bijection} there exists a bijection
\begin{align*}
\Upsilon_3:\,\,\,\mathcal{N}\,\,\,&\rightarrow \mathcal{M}\\
(g,h)&\mapsto \phi,
\end{align*}
where $\phi(a,x)=\big(a+g(x),h(x)\big)$, for all $a \in D,x \in V$.
\item \label{It:isomorphism} under the bijection $\Upsilon_3$, the map $\phi$ is bijective if and only if the map $h$ is bijective and the map $\phi$ co-stabilizes $V$ if and only if $h=\mathrm{Id}_{V}$.
    \end{enumerate}
\end{lem}

\allowdisplaybreaks{
\begin{proof}
\begin{enumerate}
\item[(1)] The proof of Item~\ref{It:bijection}.

First, we need to prove the map $\Upsilon_{3}$ is well defined. Suppose that the pair $(g,h)\in \mathcal{N}$, and $\phi=\Upsilon_{3}(g,h)$. In Diagram ~(\ref{diagram:equivalent cohomologous uu}), we have
$$\phi\big(i^{\prime}(a)\big)=\phi(a,0)=\big(a+g(0),h(0)\big)=(a,0)
=i^{\prime}(a),\,\text{for}\, a\in D.$$
Hence $\phi$ stabilizes $D$. Moreover, the map $\phi$ is a morphism of dendriform algebras if and only if the following Eq. ~(\ref{zformula(not):algebra morphism}) holds for all $a,b\in D$ and $x,y\in V$,
\begin{align}\label{zformula(not):algebra morphism}
\begin{split}
   \phi\big((a,x)\succeq (b,y)\big)&=\phi(a,x)\succeq^{\prime} \phi(b,y);\\
      \phi\big((a,x)\preceq (b,y)\big)&=\phi(a,x)\preceq^{\prime} \phi(b,y).
      \end{split}
\end{align}
In fact, Eq. ~(\ref{zformula(not):algebra morphism}) holds if and only if it holds for the generating set $\{(a,0)|a\in D\}\cup \{(0,x)|x\in V\}$ of $\dnv$. By Eq. ~(\ref{formulas:unified product}), we have
 \begin{align*}
0=&\phi\big((a,0)\succeq (0,x)\big)-\phi(a,0)\succeq^{\prime} \phi(0,x)\\
    =&\phi(a\leftharpoonup_{1} x,a\triangleright_{1} x)-(a,0)\succeq^{\prime}\big(g(x),h(x)\big)\\
    =&\big(a\leftharpoonup_{1} x+g(a\triangleright_{1} x),h(a\triangleright_{1} x)\big)-\big(a\succ g(x)+a\leftharpoonup^{\prime}_{1}h(x),a\triangleright^{\prime}_{1} h(x)\big);\\
0=&\phi\big((a,0)\preceq (0,x)\big)-\phi(a,0)\preceq^{\prime} \phi(0,x)\\
    =&\phi(a\leftharpoonup_{2} x,a\triangleright_{2} x)-(a,0)\preceq^{\prime}\big(g(x),h(x)\big)\\
    =&\big(a\leftharpoonup_{2} x+g(a\triangleright_{2} x),h(a\triangleright_{2} x)\big)-\big(a\prec g(x)+a\leftharpoonup^{\prime}_{2}h(x),a\triangleright^{\prime}_{2} h(x)\big).
    \end{align*}
Hence, the relations (E3) and (E1.1) hold if and only if Eq. ~(\ref{zformula(not):algebra morphism}) holds for the pair $(a,0)$, $(0,x)$ with $a\in D$, $x\in V$. Similarly, the relations (E2) and (E1.2) hold if and only if Eq. ~(\ref{zformula(not):algebra morphism}) holds for the pair $(0,x)$, $(a,0)$ with $a\in D$, $x\in V$. The relations (E4) and (E5) hold if and only if Eq. ~(\ref{zformula(not):algebra morphism}) holds for the pair $(0,x)$, $(0,y)$ with $x,y\in V$. Hence, $\phi=\Upsilon_{3}(g,h)$ is a morphism of dendriform algebras if and only if (E1)-(E5) hold.
So the pair $(g,h)\in \mathcal{N}$ if and only if the map $\phi=\Upsilon_{3}(g,h)\in \mathcal{M}$. Hence the map $\Upsilon_{3}$ is well defined.

Second, we prove that $\Upsilon_{3}$ is bijective. In Diagram ~(\ref{diagram:equivalent cohomologous uu}), define the canonical projection $\pi_{D}:\dnv\,(\dnvp)\rightarrow D$ and the injection $i_{V}:V\rightarrow \dnv\,(\dnvp)$ as follows:
$$\pi_{D}(a,x)=a,\quad i_{V}(x)=(0,x),\,\, a\in D,\,x\in V.$$
Then we immediately obtain $\pi_{D}\,i^{\prime}=\mathrm{Id}_{D}$, $\pi^{\prime}\,i_{V}=\mathrm{Id}_{V}$ and $i^{\prime}\,\pi_{D}+i_{V}\,\pi^{\prime}=\mathrm{Id}_{\dnv\,(\dnvp)}$.

In order to prove that $\Upsilon_{3}$ is bijective, now we define a map $\Upsilon_{4}$ as follows:
\begin{align*}
\Upsilon_{4}:\mathcal{M}&\rightarrow \mathcal{N}\\
\phi&\mapsto (g,h),
\end{align*}
where $g(x)=\pi_{D}\big(\phi(0,x)\big)$ and $h(x)=\pi^{\prime}\big(\phi(0,x)\big)$, $x \in V$.
Also, we need to prove that the map $\Upsilon_{4}$ is well defined, i.e., the pair $(g,h)=\Upsilon_{4}(\phi)\in \mathcal{N}, \phi\in \mathcal{M}$. To avoid the tedious computation, suppose that
\begin{align*}
\Upsilon_{3}\circ\Upsilon_{4}:\mathcal{M}&\rightarrow \,\,\,\mathcal{N}\,\,\,\rightarrow \mathcal{M}\\
\phi&\mapsto (g,h)\mapsto \phi^{\prime}.
\end{align*}
By the proof of Item~\ref{It:bijection}, we only need to prove that $\phi^{\prime}\in \mathcal{M}$. In fact, we have
\begin{align*}
\phi^{\prime}(a,x)
=&\Upsilon_{3}(g,h)(a,x)\\
=&(a+g(x),h(x))\\
=&(a+\pi_{D}\big(\phi(0,x)\big),\pi^{\prime}\big(\phi(0,x)\big))\\
=&(a,0)+(\pi_{D}\big(\phi(0,x)\big),0)+(0,\pi^{\prime}\big(\phi(0,x)\big))\\
=&\phi(a,0)+i^{\prime}\pi_{D}\big(\phi(0,x)\big)+i_{V}\pi^{\prime}\big(\phi(0,x)\big)\\
=&\phi(a,0)+(i^{\prime}\pi_{D}+i_{V}\pi^{\prime})\big(\phi(0,x)\big)\\
=&\phi(a,0)+\phi(0,x)\\
=&\phi(a,x).
\end{align*}
Hence $\phi^{\prime}=\phi\in \mathcal{M}$ and $\Upsilon_{3}\circ\Upsilon_{4}=\mathrm{Id}_{\mathcal{M}}$.
Conversely, suppose that
\begin{align*}
\Upsilon_{4}\circ\Upsilon_{3}:\,\,\,\mathcal{N}\,\,\,&\rightarrow \mathcal{M}\rightarrow \,\,\,\mathcal{N}\\
(g,h)&\mapsto \phi\mapsto (g^{\prime},h^{\prime})
\end{align*}
for all $x\in V$, we have
\begin{align*}
g^{\prime}(x)=&\pi_{D}\big(\phi(0,x)\big)=\pi_{D}(g(x),\,h(x))=g(x);\\ h^{\prime}(x)=&\pi^{\prime}\big(\phi(0,x)\big)=\pi^{\prime}(g(x),\,h(x))=h(x).
\end{align*}
Hence $(g^{\prime},h^{\prime})=(g,h)$, i.e., $\Upsilon_{4}\circ\Upsilon_{3}=\mathrm{Id}_{\mathcal{N}}$.

\item[(2)] 
We divide the proof of Item ~\ref{It:isomorphism} into two steps.
 \begin{enumerate}
 \item Under the bijection of Item~\ref{It:bijection}, we prove that the map $\phi$ is bijective if and only if the map $h$ is bijective. First, suppose that the linear map $\phi$ is bijective, it's easy to find that
$h$ is surjective, because the map $\phi$ can't be surjective if $h$ is not surjective.  Meanwhile, we prove that $h$ is injective by contradiction. Assume that $h$ is not injective, then there exists $0\neq x\in V$, such that $h(x)=0$, so $\phi(-g(x),x)=(-g(x)+g(x),h(x))=(0,0)$, which contradicts with the injectivity of $\phi$.
Second, suppose that the map $h$ is bijective. For any element $(a,x)\in D\times V$, we have $\phi\big(a-g\big(h^{-1}(x)\big),h^{-1}(x)\big)=(a,x)$, so $\phi$ is surjective. Moreover, suppose that $\phi(a,x)=\big(a+g(x),h(x)\big)=0$, we have,
\begin{align*}
   \begin{cases}
    a+g(x)=0\\
    h(x)=0
    \end{cases}
\end{align*}
Hence, $a=0$ and $x=0$, i.e., the map $\phi$ is injective.

\item Under the bijection of Item~\ref{It:bijection}, we prove that $\phi$ co-stabilizes $V$ in Diagram ~(\ref{diagram:equivalent cohomologous uu}) if and only if $h=\mathrm{Id}_{V}$. In fact, for all $a\in D$ and $x\in V$, we have
\begin{align*}
    \pi^{\prime}\big(\phi(a,x)\big)=&\pi^{\prime}\big(a+g(x),h(x)\big)=h(x);\\
    \pi^{\prime}(a,x)=&x.
\end{align*}
Hence, $\phi$ co-stabilizes $V$ if and only if $ \pi^{\prime}\phi=\pi^{\prime}$, i.e., $h(x)=x$, for all $x\in V$, so we have $h=\mathrm{Id}_{V}$. This complete the proof.
\end{enumerate}
\end{enumerate}
\end{proof}
}

According to Definition ~\ref{defn:datum equivalent} and Lemma ~\ref{lem:datum equivalent cohomologous}, we have the following proposition.
\begin{prop}\label{prop:datum equivalent cohomologous}
Let $D$ be a dendriform algebra, and $V$ a vector space. Suppose that $\Omega(D,V)$ and $\Omega^{\prime}(D,V)$ are two dendriform extending structures of $D$ through $V$, and $\dnv$, $\dnvp$ are corresponding unified products. Then $\Omega(D,V)\equiv \Omega^{\prime}(D,V)$ (resp. $\Omega(D,V)\approx \Omega^{\prime}(D,V)$) if and only if $\dnv\equiv \dnvp$ (resp. $\dnv\approx \dnvp$) in {\rm Diagram ~(\ref{diagram:equivalent cohomologous uu})}.

\end{prop}

Now we arrive at our main result of classifying the dendriform structures.
\allowdisplaybreaks{
\begin{thm}\label{thm:classification of extension}
Let $D$ be a dendriform algebra. Suppose that $E$ is a vector space containing $D$ as a subspace, and $V$ is a space complement of $D$ in $E$. Then
\begin{enumerate}
\item \label{thm:classification of extension1}the relation $\equiv$ is an equivalence relation on the set $\mathcal{O}(D,V)$. Moreover, suppose that $H^{2}(V,D):=\mathcal{O}(D,V)/\equiv$, then there exists a bijection
\begin{align*}
 \Upsilon:\mathbf{Extd}(E,D)&\rightarrow H^{2}(V,D);\\
 [\cee]&\mapsto [\Omega(D,V)],
 \end{align*}
where the dendriform extending structure $\Omega(D,V)$ is obtained by {\rm Eq. ~(\ref{formulas:extension to unified})}, and $[\Omega(D$, $V)]$ (resp. $[\cee]$) denote the equivalence class of $\Omega(D,V)$ (resp. $\cee$) via $\equiv$.
 \item \label{thm:classification of extension2}the relation $\approx$ is an equivalence relation on the set $\mathcal{O}(D,V)$. Moreover, suppose that $H_{V}^{2}(V,D):=\mathcal{O}(D,V)/\approx$, then there exists a bijection
\begin{align*}
 \Upsilon^{\prime}:\mathbf{Extd^{\prime}}(E,D)&\rightarrow H_{V}^{2}(V,D);\\
 \overline{\cee}&\mapsto \overline{\Omega(D,V)},
\end{align*}
where the dendriform extending structure $\Omega(D,V)$ is obtained by {\rm Eq. ~(\ref{formulas:extension to unified})}, and $\overline{\Omega(D,V)}$ (resp. $\overline{\cee}$) denote the equivalence class of $\Omega(D,V)$ (resp. $\cee$) via $\approx$.
\end{enumerate}
\end{thm}
}

\allowdisplaybreaks{
\begin{proof} Here we only prove Item~\ref{thm:classification of extension1} and Item ~\ref{thm:classification of extension2} can be proved similarly.
By Proposition ~\ref{prop:bijection of extension and unified product}, we have a bijection
\begin{align*}
\Upsilon_{1}:\mathbf{Exts}(E,D)&\rightarrow \mathcal{O}(D,V)\\
\cee&\mapsto \Omega(D,V);\\
\ceep&\mapsto \Omega^{\prime}(D,V).
\end{align*}
Suppose that $\dnv$ and $\dnvp$ are corresponding unified products. By Theorem ~\ref{thm:extension to unified}, we know that $\cee\approx \dnv$ and $\ceep\approx \dnvp$, which implies that $\cee\equiv \dnv$ and $\ceep\equiv \dnvp$, so $\cee\equiv \ceep$ if and only if $\dnv\equiv \dnvp$.
Moreover, by Proposition ~\ref{prop:datum equivalent cohomologous}, we know that $\Omega(D,V)\equiv \Omega^{\prime}(D,V)$ if and only if $\dnv\equiv \dnvp$. Hence $\Omega(D,V)\equiv\Omega^{\prime}(D,V)$ if and only if $\cee\equiv \ceep$, so we have an equivalence relation $\equiv$  on the set $\mathcal{O}(D,V)$.
The map $\Upsilon_{1}$ induces a bijection $\Upsilon:H^{2}(V,D)\rightarrow \mathbf{Extd}(E,D)$. This completes the proof.
\end{proof}
}

\subsection{Flag extending structures of dendriform algebras}\label{subsection:flag}
In this subsection, we apply our main Theorem~\ref{thm:classification of extension} to the flag extending structure of dendriform algebras as a special case.

\begin{defn}
Let $D$ be a dendriform algebra, $E$ a vector space containing $D$ as a subspace. A dendriform algebra structure defined on $E$ is called a {\bf flag extending structure} of $D$ to $E$ if there exists a finite chain of subalgebras:
$$D=E_{0}\subset E_{1}\subset\cdots\subset E_{n}=E.$$
such that $E_{i}$ has codimension 1 in $E_{i+1}$, for $0\leq i \leq n-1$.
\end{defn}

Suppose that $D$ has finite codimension in $E$, $V$ is a space complement of $D$ in $E$ with basis $\{x_{1},x_{2},\cdots,x_{n}\}$, then the flag extending structure of $D$ to $E$ can be defined recursively. In fact, we can first define all the dendriform structures on $E_{1}:=D+{\bf k}\{x_{1}\}$ through dendriform extending structures $\Omega(D,{\bf k}\{x_{1}\})$. Second, we can define all the dendriform structures on $E_{2}:=E_{1}+{\bf k}\{x_{2}\}$ through dendriform extending structures $\Omega(E_{1},{\bf k}\{x_{2}\}),\cdots$, by finite steps, we can define all the dendriform structures on $E_{n}:=E_{n-1}+{\bf k}\{x_{n}\}$. Since each step is similar to the first one, we mainly consider the dendriform extending structure of $D$ through a 1-dimensional vector space $V$.
\begin{defn}~\label{def:flag}
Let $D$ be a dendriform algebra. A {\bf flag datum} of $D$ is a 12-tuple $\big(l_{1},l_{2},r_{1},r_{2},p_{1}$, $p_{2},q_{1},q_{2},a_{1},a_{2},\bar{k}_{1},\bar{k}_{2}\big)$ consisting of eight linear maps $l_{1},l_{2},r_{1},r_{2}:D\rightarrow {\bf k}$ and $p_{1},p_{2},q_{1},q_{2}:D\rightarrow D$, elements $a_{1},a_{2}\in D$ and $\bar{k}_{1},\bar{k}_{2}\in {\bf k}$, satisfying the following conditions for all $a,b\in D$,
 \allowdisplaybreaks{
      \begin{align*}
        (F1)\,\,
                &l_{1}(a\star b)=l_{1}(a)l_{1}(b);\quad
                l_{2}(a\prec b)=l_{2}(a)(l_{1}(b)+l_{2}(b));\\
                &l_{2}(a\succ b)=l_{1}(a)l_{2}(b);\quad
                (r_{1}(a)+r_{2}(a))r_{1}(b)=r_{1}(a\succ b);\\
                &r_{2}(a)r_{2}(b)=r_{2}(a\star b);\quad
                r_{1}(a)r_{2}(b)=r_{1}(a\prec b);\\
                &l_{2}(a)r_{1}(b)=0;\\
        (F2)\,\,&q_{1}(a\star b)=a\succ q_{1}(b)+q_{1}(a)l_{1}(b);\\
        &q_{2}(a\prec b)=a\prec (q_{1}(b)+q_{2}(b))+q_{2}(a)(l_{1}(b)+l_{2}(b));\\
        &q_{2}(a\succ b)=a\succ q_{2}(b)+q_{1}(a)l_{2}(b);\\
        (F3)\,\,&(p_{1}(a)+p_{2}(a))\succ b+(r_{1}(a)+r_{2}(a))p_{1}(b)=p_{1}(a\succ b);\\
        &p_{2}(a)\prec b+r_{2}(a)p_{2}(b)=p_{2}(a\star b);\\
        &p_{1}(a)\prec b+r_{1}(a)p_{2}(b)=p_{1}(a\prec b);\\
        (F4)\,\,&(q_{1}(a)+q_{2}(a))\succ b+(l_{1}(a)+l_{2}(a))p_{1}(b)=a\succ p_{1}(b)+q_{1}(a)r_{1} (b);\\
        &q_{2}(a)\prec b+l_{2}(a)p_{2}(b)=a\prec (p_{1}(b)+p_{2}(b))+q_{2}(a)(r_{1}(b)+r_{2}(b));\\
        &q_{1}(a)\prec b+l_{1}(a)p_{2}(b)=a\succ p_{2}(b)+q_{1}(a)r_{2}(b);\\
        (F5)\,\,&(a_{1}+a_{2})\succ a+(\bar{k}_{1}+\bar{k}_{2})p_{1}(a)=p_{1}(p_{1}(a))+a_{1}r_{1}(a);\\
        &a_{2}\prec a+\bar{k}_{2}p_{2}(a)=p_{2}(p_{1}(a)+p_{2}(a))+a_{2}(r_{1}(a)+r_{2}(a));\\
        &a_{1}\prec a+\bar{k}_{1}p_{2}(a)=p_{1}(p_{2}(a))+a_{1}r_{2}(a);\\
        (F6)\,\,&\bar{k}_{2}r_{1}(a)=r_{1}(p_{1}(a));\\
                &r_{2}(p_{1}(a)+p_{2}(a))+\bar{k}_{2}r_{1}(a)=0;\\
                &r_{1}(p_{2}(a))=0;\\
        (F7)\,\,&q_{1}(q_{1}(a)+q_{2}(a))+a_{1}(l_{1}(a)+l_{2}(a))=a\succ a_{1}+ q_{1}(a)\bar{k}_{1};\\
                &q_{2}(q_{2}(a))+a_{2}l_{2}(a)=a\prec (a_{1}+a_{2})+ q_{2}(a)(\bar{k}_{1}+\bar{k}_{2});\\
                &q_{2}(q_{1}(a))+a_{2}l_{1}(a)=a\succ a_{2}+q_{1}(a)\bar{k}_{2};\\
        (F8)\,\,&l_{1}(q_{1}(a)+q_{2}(a))+l_{2}(a)\bar{k}_{1}=0;\\
                &l_{2}(q_{2}(a))= l_{2}(a)\bar{k}_{1};\\
                &l_{2}(q_{1}(a))=0;\\
        (F9)\,\,&q_{1}(p_{1}(a)+p_{2}(a))+a_{1}(r_{1}(a)+r_{2}(a))=p_{1}(q_{1}(a))+a_{1}l_{1}(a);\\
               &q_{2}(p_{2}(a))+a_{2}r_{2}(a)=p_{2}(q_{1}(a)+q_{2}(a))+a_{2}(l_{1}(a)+l_{2}(a));\\
        &q_{2}(p_{1}(a))+a_{2}r_{1}(a)=p_{1}(q_{2}(a))+a_{1}l_{2}(a);\\
        (F10)\,&l_{1}(p_{1}(a)+p_{2}(a))+(r_{1}(a)+r_{2}(a))\bar{k}_{1}=r_{1}(q_{1}(a))+ \bar{k}_{1}l_{1}(a);\\
        &l_{2}(p_{2}(a))+r_{2}(a)\bar{k}_{2}=r_{2}(q_{1}(a)+q_{2}(a))+ \bar{k}_{2}(l_{1}(a)+l_{2}(a));\\
        &l_{2}(p_{1}(a))+r_{1}(a)\bar{k}_{2}=r_{1}(q_{2}(a))+ \bar{k}_{1}l_{2}(a);\\
        (F11)\,&q_{1}(a_{1}+a_{2})+a_{1}\bar{k}_{2}=p_{1}(a_{1});\\
        &q_{2}(a_{2})=p_{2}(a_{1}+a_{2})+a_{2}\bar{k}_{1};\\
        &q_{2}(a_{1})+a_{2}\bar{k}_{1}=p_{1}(a_{2})+a_{1}\bar{k}_{2};\\
        (F12)\,&l_{1}(a_{1}+a_{2})+\bar{k}_{2}\bar{k}_{1}=r_{1}(a_{1});\\
        &l_{2}(a_{2})=r_{2}(a_{1}+a_{2})+\bar{k}_{2}\bar{k}_{1};\\
        &l_{2}(a_{1})=r_{1}(a_{2}). \end{align*}
    }
\end{defn}

Denote by $\mathcal{F}(D)$ the set of all flag datums of $D$, then we have the following result.

\begin{thm}\label{thm:flag to extending}
Let $D$ be a dendriform algebra of codimension $1$ in the vector space $E$. Suppose that $V$ is a space complement of $D$ in $E$ with basis $\{x\}$. Then there exists a bijection
\begin{align*}
 \Phi:\,\,\,\,\,\,\,\,\,\,\,\,\,\,\,\,\,\,\,\,\mathcal{F}(D)\,\,\,\,\,\,\,\,\,\,\,\,\,\,\,\,\,\,\,\,\,&\rightarrow \mathcal{O}(D,V)\\
 (l_{1},l_{2},r_{1},r_{2},p_{1},p_{2},q_{1},q_{2},a_{1},a_{2},\bar{k}_{1},\bar{k}_{2})&\mapsto \Omega(D,V),
 \end{align*}
where the dendriform extending structure $\Omega(D,V)$ is defined as follows:
  \allowdisplaybreaks{
\begin{align}\label{formulas:flag and extending structure}
\begin{aligned}
\begin{split}
        a\triangleright_{1} x&=l_{1}(a)\,x;\\
        x\triangleleft_{1} a&=r_{1}(a)\,x;\\
        x\rightharpoonup_{1} a&=p_{1}(a);\\
        a\leftharpoonup_{1} x&=q_{1}(a);\\
        f_{1}(x,x)&=a_{1};\\
        x\succ_{V} x&=\bar{k}_{1}\,x;\\
\end{split}
\end{aligned}
\begin{aligned}
\begin{split}
        a\triangleright_{2} x&=l_{2}(a)\,x;\\
        x\triangleleft_{2} a&=r_{2}(a)\,x;\\
        x\rightharpoonup_{2} a&=p_{2}(a);\\
        a\leftharpoonup_{2} x&=q_{2}(a);\\
        f_{2}(x,x)&=a_{2};\\
        x\prec_{V} x&=\bar{k}_{2}\,x,\, a\in D.
\end{split}
\end{aligned}
\end{align}
}
\end{thm}

\begin{proof}
We can directly check that Eqs. (F1)-(F12) of the flag datum in Definition~\ref{def:flag}, which are equivalent to Eqs. (D1)-(D12) of the dendriform extending structure $\Omega(D,V)$ by Eq.~\eqref{formulas:flag and extending structure}, so $\Phi$ is a well defined bijection.
\end{proof}

Under the conditions of Theorem ~\ref{thm:flag to extending}, suppose that
\begin{align}\label{formulas:flag and extending equivalent}
g(x)=g_{0};\quad h(x)=\bar{h}_{0}x.
\end{align}
where the linear maps $g,h$ are defined in Definition ~\ref{defn:datum equivalent} and $g_{0}\in D$, $\bar{h}_{0}\in {\bf k}$.
By Eq.~\eqref{formulas:flag and extending structure} and Definition ~\ref{defn:datum equivalent}, we give the following definition.
\begin{defn}\label{defn:flag datum equivalent}
Let $D$ be a dendriform algebra. Suppose that $\Omega(D)=(l_{1},l_{2},r_{1},r_{2},p_{1},p_{2},q_{1},q_{2},$ $a_{1},a_{2},\bar{k}_{1}, \bar{k}_{2}) $ and $\Omega^{\prime}(D)=(l^{\prime}_{1},l^{\prime}_{2},r^{\prime}_{1},r^{\prime}_{2},
p^{\prime}_{1},p^{\prime}_{2},q^{\prime}_{1},q^{\prime}_{2},a^{\prime}_{1},a^{\prime}_{2},
\bar{k_1}^{\prime}, \bar{k_2}^{\prime})$ are two flag datums of $D$. They are called {\bf equivalent}, denoted by $\Omega(D)\equiv \Omega^{\prime}(D)$ if there exist an element $g_{0}\in D$ and a number $\bar{h}_{0}\in {\bf k}\,(\bar{h}_{0}\neq 0)$, satisfying the following conditions with $a\in D$,
\allowdisplaybreaks{
\begin{align}\label{formulas:flag datum equivalent}
\begin{split}
        l^{\prime}_{1}=&l_{1};\,\,\,\,l^{\prime}_{2}=l_{2};\,\,\,\,r^{\prime}_{1}=r_{1};\,\,\,\,r^{\prime}_{2}=r_{2};\\
        p^{\prime}_{1}(a)=&\frac{1}{\bar{h}_{0}}(p_{1}(a)+g_{0}r_{1}(a)-g_{0}\succ a);\,\,\,\,
        p^{\prime}_{2}(a)=\frac{1}{\bar{h}_{0}}(p_{2}(a)+g_{0}r_{2}(a)-g_{0}\prec a);\\
        q^{\prime}_{1}(a)=&\frac{1}{\bar{h}_{0}}(q_{1}(a)+g_{0}l_{1}(a)-a\succ g_{0});\,\,\,\,
        q^{\prime}_{2}(a)=\frac{1}{\bar{h}_{0}}(q_{2}(a)+g_{0}l_{2}(a)-a\prec g_{0});\\
        a^{\prime}_{1}=&\frac{1}{h^{2}_{0}}(a_{1}+g_{0}\bar{k}_{1}-q_{1}(g_{0})-g_{0}l_{1}(g_{0})+g_{0}\succ g_{0}-p_{1}(g_{0})-g_{0}r_{1}(g_{0}));\\
        a^{\prime}_{2}=&\frac{1}{h^{2}_{0}}(a_{2}+g_{0}\bar{k}_{2}-q_{2}(g_{0})-g_{0}l_{2}(g_{0})+g_{0}\prec g_{0}-p_{2}(g_{0})-g_{0}r_{2}(g_{0}));\\
        \bar{k}^{\prime}_{1}=&\frac{1}{\bar{h}_{0}}(\bar{k}_{1}-l_{1}(g_{0})-r_{1}(g_{0}));\,\,\,\,
        \bar{k}^{\prime}_{2}=\frac{1}{\bar{h}_{0}}(\bar{k}_{2}-l_{2}(g_{0})-r_{2}(g_{0})).
\end{split}
\end{align}
}
 Moreover, if $\bar{h}_{0}=1$, these two flag datums are called {\bf cohomologous}, denoted by $\Omega(D)\approx\Omega^{\prime}(D)$.
\end{defn}


\begin{thm}
\allowdisplaybreaks{
\label{classification of extension}
Let $D$ be a dendriform algebra of codimension $1$ in the vector space $E$. Suppose that $V$ is a space complement of $D$ in $E$ with basis $\{x\}$.
Then
\begin{enumerate}
\item \label{thm:classification of extension1}the relation $\equiv$ is an equivalence relation on the set $\mathcal{F}(D)$. Moreover, suppose that $H^{2}(D):=\mathcal{F}(D)/\equiv$, then there exists a bijection
\begin{align*}
 \Upsilon:\mathbf{Extd}(E,D)&\rightarrow H^{2}(D);\\
 [\cee]&\mapsto [\Omega(D)],
 \end{align*}
where the flag datum $\Omega(D)$ is obtained by {\rm Eqs. ~(\ref{formulas:extension to unified}) and~(\ref{formulas:flag and extending structure})}, and $[\Omega(D)]$ (resp. $[\cee]$) denote the equivalence class of $\Omega(D)$ (resp. $\cee$) via $\equiv$.
 \item \label{thm:classification of extension2}the relation $\approx$ is an equivalence relation on the set $\mathcal{F}(D)$. Moreover, suppose that $H_{V}^{2}(D):=\mathcal{F}(D)/\approx$, then there exists a bijection
\begin{align*}
 \Upsilon^{\prime}:\mathbf{Extd^{\prime}}(E,D)&\rightarrow H_{V}^{2}(D);\\
 \overline{\cee}&\mapsto \overline{\Omega(D)},
\end{align*}
where the flag datum $\Omega(D)$ is obtained by {\rm Eqs. ~(\ref{formulas:extension to unified}) and~(\ref{formulas:flag and extending structure})}, and $\overline{\Omega(D)}$ (resp. $\overline{\cee}$) denote the equivalence class of $\Omega(D)$ (resp. $\cee$) via $\approx$.
\end{enumerate}
}
\end{thm}

\begin{proof}
It can be directly obtained from Theorem ~\ref{thm:classification of extension} by some simple computations.
\end{proof}
For better understanding, now we give an example to compute $H^{2}(D)$ and $H_{V}^{2}(D)$.
\begin{exam}\label{exam:ES problem}
To continue Example~\ref{exam:dendriform algebra}~\ref{exam:dendriform algebra1} and suppose $V={\bf k}\{e_{2}\}$.
Let  $(l_{1},l_{2},r_{1},r_{2},p_{1},p_{2},q_{1},q_{2},$ $a_{1},a_{2},\bar{k}_{1},\bar{k}_{2}) $ be a flag datum of $D$, and $g_{0}$ defined in Definition ~\ref{defn:flag datum equivalent}. Suppose that
  \allowdisplaybreaks{
\begin{align}\label{formulas:flag and D}
\begin{aligned}
\begin{split}
        l_{1}(e_{1})&=\bar{l}_{1}e_{1};\\
        p_{1}(e_{1})&=\bar{p}_{1}e_{1};\\
        a_{1}&=\bar{a}_{1}e_{1};\\
\end{split}
\end{aligned}
\begin{aligned}
\begin{split}
        l_{2}(e_{1})&=\bar{l}_{2}e_{1};\\
        p_{2}(e_{1})&=\bar{p}_{2}e_{1};\\
        a_{2}&=\bar{a}_{2}e_{1};\\
\end{split}
\end{aligned}
\begin{aligned}
\begin{split}
        r_{1}(e_{1})&=\bar{r}_{1}e_{1};\\
        q_{1}(e_{1})&=\bar{q}_{1}e_{1};\\
        g_{0}&=\bar{g}_{0}e_{1};
\end{split}
\end{aligned}
\begin{aligned}
\begin{split}
        r_{2}(e_{1})&=\bar{r}_{2}e_{1};\\
        q_{2}(e_{1})&=\bar{q}_{2}e_{1};\\
        a=b&=e_{1}.
\end{split}
\end{aligned}
\end{align}
}where elements $\bar{l}_{1},\bar{l}_{2},\bar{r}_{1},\bar{r}_{2},\bar{p}_{1},\bar{p}_{2},\bar{q}_{1},\bar{q}_{2},\bar{a}_{1},\bar{a}_{2},\bar{k}_{1},\bar{k}_{2},\bar{g}_{0}\in {\bf k}$, and the 12-tuple $(\bar{l}_{1},\bar{l}_{2},\bar{r}_{1},\bar{r}_{2},\bar{p}_{1},\bar{p}_{2}$, $\bar{q}_{1},\bar{q}_{2},\bar{a}_{1},\bar{a}_{2},\bar{k}_{1},\bar{k}_{2})$ is also called a flag datum of $D$ for convenience.

\begin{table}[htbp]
  \centering
  \renewcommand\arraystretch{1.8}
\begin{tabular}{|c|c|c|c|c|}
 \hline
  case&$(\bar{l}_{1},\bar{l}_{2},\bar{r}_{1},\bar{r}_{2},\bar{p}_{1},\bar{p}_{2},\bar{q}_{1},\bar{q}_{2},
  \bar{a}_{1},\bar{a}_{2},\bar{k}_{1},\bar{k}_{2})$ & $\bar{h}_{0}$ & $\bar{g}_{0}$ & equivalent (cohomologous) class  \\

  \hline
  1&$(1,-1,0,0,\bar{p}_{1},0,0,\bar{p}_{1},0,\bar{p}^{2}_{1},\bar{p}_{1},-\bar{p}_{1})$ & $1$ & $-\bar{p}_{1}$ & $(1,-1,0,0,0,0,0,0,0,0,0,0)$  \\

      \hline
   \multirow{2}*{2}& $(1,-1,0,0,\bar{p}_{1},0,0,\bar{p}_{1},0,\bar{p}^{2}_{1}-\bar{k}_{2}\bar{p}_{1},$ & $\frac{1}{\bar{k}_{2}}$ ~& $-\frac{1}{\bar{k}_{2}}\bar{p}_{1}$ & $(1,-1,0,0,0,0,0,0,0,0,0,1)$  \\
  \cline{3-5}
      ~&$\bar{p}_{1},\bar{k}_{2}-\bar{p}_{1})$, $\bar{k}_{2}\neq 0$ & $1$ & $-\bar{p}_{1}$ & $(1,-1,0,0,0,0,0,0,0,0,0,\bar{k}_{2})$\\
    \hline
   3& $(0,0,1,0,0,0,\bar{q}_{1},0,0,0,\bar{q}_{1},0)$ & $1$ & $-\bar{q}_{1}$ & $(0,0,1,0,0,0,0,0,0,0,0,0)$  \\

  \hline
  4&$(0,0,0,1,\bar{p}_{1},-\bar{p}_{1},\bar{p}_{1},0,\bar{p}^{2}_{1},-\bar{p}^{2}_{1},0,\bar{p}_{1})$ & $1$ & $-\bar{p}_{1}$ & $(0,0,0,1,0,0,0,0,0,0,0,0)$  \\
\hline
  5&$(1,0,0,0,\bar{p}_{1},0,0,0,0,0,\bar{p}_{1},0)$ & $1$ & $-\bar{p}_{1}$ & $(1,0,0,0,0,0,0,0,0,0,0,0)$  \\
\hline
 6& $(1,0,1,0,0,0,0,0,-\frac{1}{4}\bar{k}^{2}_{1},0,\bar{k}_{1},0)$ & $1$ & $-\frac{\bar{k}_{1}}{2}$ & $(1,0,1,0,0,0,0,0,0,0,0,0)$  \\
 \hline
   \multirow{2}*{7}& $(1,0,1,0,0,0,0,0,\bar{a}_{1}-\frac{1}{4}\bar{k}^{2}_{1},0$, & $\frac{1}{\sqrt{\bar{a}_{1}}}$ ~& $-\frac{\bar{k}_{1}}{2\sqrt{\bar{a}_{1}}}$ & $(1,0,1,0,0,0,0,0,1,0,0,0)$  \\
  \cline{3-5}
     ~&$\bar{k}_{1},0)$, $\bar{a}_{1}\neq 0$  & $1$ & $-\frac{\bar{k}_{1}}{2}$ & $(1,0,1,0,0,0,0,0,\bar{a}_{1},0,0,0)$\\
\hline
  8 &$(1,0,0,1,\bar{p}_{1},-\bar{p}_{1},0,0,0,-\bar{p}_{1}^{2},\bar{p}_{1},\bar{p}_{1})$ & $1$ & $-\bar{p}_{1}$ & $(1,0,0,1,0,0,0,0,0,0,0,0)$  \\

  \hline
   \multirow{2}*{9}& $(1,0,0,1,\bar{p}_{1},-\bar{p}_{1},0,0,0,-\bar{p}_{1}\bar{k}_{2}-\bar{p}_{1}^{2}$,& $\frac{1}{\bar{k}_{2}}$ ~& $-\frac{1}{\bar{k}_{2}}\bar{p}_{1}$ & $(1,0,0,1,0,0,0,0,0,0,0,1)$  \\
   \cline{3-5}
      ~&$\bar{p}_{1},\bar{k}_{2}+\bar{p}_{1})$, $\bar{k}_{2}\neq 0$  & $1$ & $-\bar{p}_{1}$ & $(1,0,0,1,0,0,0,0,0,0,0,\bar{k}_{2})$\\
  \hline
   \multirow{2}*{10} & $(1,0,0,1,\bar{p}_{1},-\bar{p}_{1},0,0,-\bar{p}_{1}\bar{k}_{1},-\bar{p}_{1}^{2}$, & $\frac{1}{\bar{k}_{1}}$ ~& $-\frac{1}{\bar{k}_{1}}\bar{p}_{1}$ & $(1,0,0,1,0,0,0,0,0,0,1,0)$  \\
    \cline{3-5}
    ~&  $\bar{k}_{1}+\bar{p}_{1},\bar{p}_{1})$, $\bar{k}_{1}\neq 0$  & $1$ & $-\bar{p}_{1}$ & $(1,0,0,1,0,0,0,0,0,0,\bar{k}_{1},0)$\\
\hline
 11 & $(1,-1,0,1,\bar{p}_{1},-\bar{p}_{1},0,\bar{p}_{1},0,0,\bar{p}_{1},0)$ & $1$ & $-\bar{p}_{1}$ & $(1,-1,0,1,0,0,0,0,0,0,0,0)$  \\
\hline
  12 & $(0,0,0,0,\bar{q}_{1},0,\bar{q}_{1},0,\bar{q}^{2}_{1},0,0,0)$ & $1$ & $-\bar{q}_{1}$ & $(0,0,0,0,0,0,0,0,0,0,0,0)$  \\
\hline
  \multirow{2}*{13} &  $(0,0,0,0,\bar{q}_{1},0,\bar{q}_{1},0,\bar{q}_{1}(\bar{q}_{1}-\bar{k}_{1}),0$,  & $\frac{1}{\bar{k}_{1}}$ ~& $-\frac{1}{\bar{k}_{1}}\bar{q}_{1}$ & $(0,0,0,0,0,0,0,0,0,0,1,0)$  \\
    \cline{3-5}
    ~&  $\bar{k}_{1},0)$, $\bar{k}_{1}\neq 0$ & $1$ & $-\bar{q}_{1}$ & $(0,0,0,0,0,0,0,0,0,0,\bar{k}_{1},0)$\\
  \hline
   \multirow{2}*{14} & $(0,0,0,0,\bar{q}_{1}+\bar{k}_{1},0,\bar{q}_{1},\bar{k}_{1},\bar{q}^{2}_{1},\bar{q}_{1}\bar{k}_{1}$, & $\frac{1}{\bar{k}_{1}}$ ~& $-\frac{1}{\bar{k}_{1}}\bar{q}_{1}$ & $(0,0,0,0,1,0,0,1,0,0,1,0)$  \\
    \cline{3-5}
    ~&  $\bar{k}_{1},0)$, $\bar{k}_{1}\neq 0$ & $1$ & $-\bar{q}_{1}$ & $(0,0,0,0,\bar{k}_{1},0,0,\bar{k}_{1},0,0,\bar{k}_{1},0)$\\
   \hline
   \multirow{2}*{15} &  $(0,0,0,0,\bar{q}_{1},0,\bar{q}_{1},0,\bar{q}^{2}_{1},-\bar{q}_{1}\bar{k}_{2}$,  & $\frac{1}{\bar{k}_{2}}$ ~& $-\frac{1}{\bar{k}_{2}}\bar{q}_{1}$ & $(0,0,0,0,0,0,0,0,0,0,0,1)$  \\
    \cline{3-5}
   ~&   $0,\bar{k}_{2})$, $\bar{k}_{2}\neq 0$ & $1$ & $-\bar{q}_{1}$ & $(0,0,0,0,0,0,0,0,0,0,0,\bar{k}_{2})$\\
    \hline
   \multirow{2}*{16} & $(0,0,0,0,\bar{q}_{1}+\bar{k}_{2},0,\bar{q}_{1},\bar{k}_{2},\bar{q}^{2}_{1}+\bar{q}_{1}\bar{k}_{2},0$, & $\frac{1}{\bar{k}_{2}}$ ~& $-\frac{1}{\bar{k}_{2}}\bar{q}_{1}$ & $(0,0,0,0,1,0,0,1,1,0,0,1)$  \\
   \cline{3-5}
    ~&  $0,\bar{k}_{2})$, $\bar{k}_{2}\neq 0$ & $1$ & $-\bar{q}_{1}$ & $(0,0,0,0,\bar{k}_{2},0,0,\bar{k}_{2},\bar{k}_{2},0,0,\bar{k}_{2})$\\
    \hline
\end{tabular}
\\[5pt]
  \caption{classifying flag datum}\label{table:classifying flag datum}
  \vspace{-1em}
\end{table}

Computing Eq.~(F1) by Eq.~(\ref{formulas:flag and D}), we obtain all possible solutions of
$\bar{l}_{1},\bar{l}_{2},\bar{r}_{1},\bar{r}_{2}$ in Table~\ref{table:bimodules} of Example ~\ref{exam:bimodule}. Computing Eqs. (F2)-(F12) by Eq.~(\ref{formulas:flag and D}), we totally get $16$ different cases of all the flag datums, refer to the second column of Table~\ref{table:classifying flag datum}. Computing Eq. ~(\ref{formulas:flag datum equivalent}) by Eq.~(\ref{formulas:flag and D}), we obtain the third and fourth columns of Table~\ref{table:classifying flag datum}, then we classify all the flag datums, the result is in the fifth column of Table~\ref{table:classifying flag datum}.

First, we analyze the data in Table~\ref{table:classifying flag datum}. For Case 1, all the flag datums are equivalent to the flag datum $(1,-1,0,0$, $0,0,0,0,0,0,0,0)$. Moreover, the cohomologous class coincides with the equivalent class.
For Case 2, all the flag datums are equivalent to the flag datum $(1,-1,0,0,0,0,0,0,0,0,$ $0,1)$. When $\bar{k}_{2}$ is given, they are cohomologous to the flag datum $\big(1,-1,0,0,0,0,0,0,0,0,0,\bar{k}_{2}\big)$, then there exists many cohomologous classes, but they belong to the same equivalent class. In the case, two classifying methods are different by remark~\ref{rem:extd and extdp}.

Second,
by Eqs.~(\ref{formulas:unified to extension}),~(\ref{formulas:flag and extending structure}) and~(\ref{formulas:flag and D}) and for each flag datum $(\bar{l}_{1},\bar{l}_{2},\bar{r}_{1},\bar{r}_{2},\bar{p}_{1},\bar{p}_{2},\bar{q}_{1},\bar{q}_{2},
  \bar{a}_{1},\bar{a}_{2},\bar{k}_{1},$ $\bar{k}_{2})$, we can obtain an extension $\cee$ defined on $E={\bf k}\{e_{1},e_{2}\}$, where $\succ_{E}$ and $\prec_{E}$ is defined as follows:
\begin{align}\label{formulas:flag to extension}
\begin{aligned}
e_{1}\succ_{E}e_{1}&=e_{1},\\
e_{1}\succ_{E}e_{2}&=\bar{q}_{1}e_{1}+\bar{l}_{1}e_{2},\\
e_{2}\succ_{E}e_{1}&=\bar{p}_{1}e_{1}+\bar{r}_{1}e_{2},\\
e_{2}\succ_{E}e_{2}&=\bar{a}_{1}e_{1}+\bar{k}_{1}e_{2},\\
\end{aligned}\,\,\,\,\,\,\,
\begin{aligned}
e_{1}\prec_{E}e_{1}&=0,\\
e_{1}\prec_{E}e_{2}&=\bar{q}_{2}e_{1}+\bar{l}_{2}e_{2},\\
e_{2}\prec_{E}e_{1}&=\bar{p}_{2}e_{1}+\bar{r}_{2}e_{2},\\
e_{2}\prec_{E}e_{2}&=\bar{a}_{2}e_{1}+\bar{k}_{2}e_{2}.
\end{aligned}
\end{align}
For each case, by Eq.~(\ref{formulas:flag to extension}), we obtain extensions $\cee$ from the second column and extension $\ceep$ from the fifth column, the isomorphism of dendriform algebras $\psi:\ceep\rightarrow \cee$ stabilizing $D$ (and costabilizing $V$) in Diagram ~(\ref{diagram:equivalent cohomologous ee}) can be defined by \[\psi(e_{1})=e_{1},\psi(e_{2})=\bar{g}_{0}e_{1}+\bar{h}_{0}e_{2}.\]
 More specifically, by Eqs. ~\eqref{formulas:flag and extending equivalent} and ~\eqref{formulas:flag and D}, we have $g(e_{2})=\bar{g}_{0}e_{1},h(e_{2})=\bar{h}_{0}e_{2}.$ By Lemma ~\ref{lem:datum equivalent cohomologous}, the isomorphism of dendriform algebras $\phi:\dnvp\rightarrow \dnv$ stabilizing $D$ (and costabilizing $V$) in Diagram ~(\ref{diagram:equivalent cohomologous uu}) can be defined by \[\phi(e_{1},0)=(e_{1},0),\phi(0,e_{2})=(\bar{g}_{0}e_{1},\bar{h}_{0}e_{2}).\]
By Theorem ~\ref{thm:extension to unified}, the isomorphism $\psi$ is defined by $\varphi\circ\phi\circ\varphi^{-1}$, where $\varphi$ is defined by Eq. ~\eqref{formulas:map eu}, refer to the following diagram:
\begin{align*}
\begin{split}
\xymatrix{(\bar{g}_{0},\bar{h}_{0})\ar[r]^{\mathrm{Eq.~\eqref{formulas:flag and D}}}&(g_{0},\bar{h}_{0})\ar[r]^{\mathrm{Eq.~\eqref{formulas:flag and extending equivalent}}}\ar[l]&(g,h)\ar[r]^{\quad\quad\mathrm{Lemma ~\ref{lem:datum equivalent cohomologous}}}\ar[l]&\ar[l]
}
\begin{split}
\end{split}
\makecell{\xymatrix{D\natural V \ar[r]^{\varphi\quad\,\,} &\cee\\
D\natural^{\prime} V \ar[u]^{\phi}\ar[r]^{\varphi\quad\,\,} &\ceep\ar[u]^{\psi}}
}
\end{split}
\end{align*}
\end{exam}

\newpage
\section{Special cases of unified products for dendriform algebras}
In this section, we mainly consider bicrossed products, cocycle semidirect products and nonabelian semidirect products as special cases of unified products. Moreover, we consider the splitting of short exact sequences.

\subsection{Matched pairs and bicrossed products of dendriform algebras}
In this subsection, we mainly consider the factorization problem for dendriform algebras, it is a subproblem of the ES problem. To solve the factorization problem, we consider matched pairs and bicrossed products of dendriform algebras, which are special cases of dendriform extending structures and unified products, respectively.

\begin{defn}
A {\bf matched pair} of dendriform algebras is a system $(D,V,\triangleright_{1},\triangleright_{2},\triangleleft_{1},\triangleleft_{2},\rightharpoonup_{1},\rightharpoonup_{2},\leftharpoonup_{1},\leftharpoonup_{2})$ consisting of two dendriform algebras $D$ and $V$, and eight bilinear maps:
\begin{align*}
\begin{split}
\triangleright_{1},\triangleright_{2}&:\,D\times V\rightarrow V,\\
\rightharpoonup_{1},\rightharpoonup_{2}&:\,V\times D\rightarrow D,
\end{split}
\begin{split}
\triangleleft_{1},\triangleleft_{2}&:\,V\times D\rightarrow V,\\
\leftharpoonup_{1},\leftharpoonup_{2}&:\,D\times V\rightarrow D
\end{split}
\end{align*}
satisfying the following conditions for all $a,b\in D$ and $x,y\in V$:
\allowdisplaybreaks{
\begin{align*}
        (M1)\,\,&\text{$\big(V,\triangleright_{1},\triangleright_{2},\triangleleft_{1},\triangleleft_{2}\big)$ is a $D$ bimodule, $\big(D,\rightharpoonup_{1},\rightharpoonup_{2},\leftharpoonup_{1},\leftharpoonup_{2}\big)$ is a $V$ bimodule;}\\
        (M2)\,\,&(a\star b)\leftharpoonup_{1} x=a\succ(b\leftharpoonup_{1} x)+a\leftharpoonup_{1}(b\triangleright_{1} x);\\
        &(a\prec b)\leftharpoonup_{2} x=a\prec (b\leftharpoonup x)+a\leftharpoonup_{2}(b\triangleright x);\\&(a\succ b)\leftharpoonup_{2} x=a\succ(b\leftharpoonup_{2} x)+a\leftharpoonup_{1}(b\triangleright_{2} x);\\
        (M3)\,\,&(x\rightharpoonup a)\succ b+(x\triangleleft a)\rightharpoonup_{1} b=x\rightharpoonup_{1}(a\succ b);\\
        &(x\rightharpoonup_{2} a)\prec b+(x\triangleleft_{2} a)\rightharpoonup_{2} b=x\rightharpoonup_{2}(a\star b);\\&(x\rightharpoonup_{1} a)\prec b+(x\triangleleft_{1} a)\rightharpoonup_{2} b=x\rightharpoonup_{1}(a\prec b);\\
        (M4)\,\,&(a\leftharpoonup x)\succ b+(a\triangleright x)\rightharpoonup_{1} b=a\succ(x\rightharpoonup_{1} b)+a\leftharpoonup_{1}(x\triangleleft_{1} b);\\
        &(a\leftharpoonup_{2} x)\prec b+(a\triangleright_{2} x)\rightharpoonup_{2} b=a\prec (x\rightharpoonup b)+a\leftharpoonup_{2}(x\triangleleft b);\\&(a\leftharpoonup_{1} x)\prec b+(a\triangleright_{1} x)\rightharpoonup_{2} b=a\succ(x\rightharpoonup_{2} b)+a\leftharpoonup_{1}(x\triangleleft_{2} b);\\
        (M5)\,\,&(x\star_{V} y)\triangleleft_{1} a=x\triangleleft_{1}(y\rightharpoonup_{1} a)+x\succ_{V}(y\triangleleft_{1} a);\\
                &(x\prec_{V} y)\triangleleft_{2} a=x\triangleleft_{2}(y\rightharpoonup a)+x\prec_{V}(y\triangleleft a);\\
                &(x\succ_{V} y)\triangleleft_{2} a=x\triangleleft_{1}(y\rightharpoonup_{2} a)+x\succ_{V}(y\triangleleft_{2} a);\\
        (M6)\,\,&(a\leftharpoonup x)\triangleright_{1} y+(a\triangleright x)\succ_{V} y= a\triangleright_{1} (x\succ_{V} y);\\
                &(a\leftharpoonup_{2} x)\triangleright_{2} y+(a\triangleright_{2} x)\prec_{V} y= a\triangleright_{2} (x\star_{V} y);\\
                &(a\leftharpoonup_{1} x)\triangleright_{2} y+(a\triangleright_{1} x)\prec_{V} y= a\triangleright_{1} (x\prec_{V} y);\\
        (M7)\,\,&(x\rightharpoonup a)\triangleright_{1} y+(x\triangleleft a)\succ_{V} y=x\triangleleft_{1}(a\leftharpoonup_{1} y)+ x\succ_{V}(a\triangleright_{1} y);\\
        &(x\rightharpoonup_{2} a)\triangleright_{2} y+(x\triangleleft_{2} a)\prec_{V} y=x\triangleleft_{2}(a\leftharpoonup y)+ x\prec_{V}(a\triangleright y);\\
        &(x\rightharpoonup_{1} a)\triangleright_{2} y+(x\triangleleft_{1} a)\prec_{V} y=x\triangleleft_{1}(a\leftharpoonup_{2} y)+ x\succ_{V}(a\triangleright_{2} y).
    \end{align*}
    }
\end{defn}

\begin{prop}\label{prop:matched pair unified product}
The extending datum $\Omega(D,V)=\big(\triangleright_{1},\triangleright_{2},\triangleleft_{1},\triangleleft_{2},
\rightharpoonup_{1},\rightharpoonup_{2},\leftharpoonup_{1},\leftharpoonup_{2},\succ_{V},\prec_{V})$ is a dendriform extending structure of $D$ through $V$, if and only if, $(D,V,\triangleright_{1},\triangleright_{2},\triangleleft_{1},\triangleleft_{2},\rightharpoonup_{1},\rightharpoonup_{2},\leftharpoonup_{1},\leftharpoonup_{2})$ is a matched pair of dendriform algebras.
\end{prop}

\begin{proof}
In Theorem ~\ref{thm:datum and unified product}, if the maps $f_{1}$ and $f_{2}$ are trivial, then Eqs. (D1)-(D12) of the dendriform extending structure $\Omega(D,V)=\big(\triangleright_{1},\triangleright_{2},\triangleleft_{1},\triangleleft_{2},
\rightharpoonup_{1},\rightharpoonup_{2},\leftharpoonup_{1},\leftharpoonup_{2},\succ_{V},\prec_{V})$ are equivalent to Eqs. (M1)-(M7) of the associated matched pair $(D,V,\triangleright_{1},\triangleright_{2},\triangleleft_{1},\triangleleft_{2},\rightharpoonup_{1},\rightharpoonup_{2},\leftharpoonup_{1},\leftharpoonup_{2})$.
\end{proof}
\begin{exam} \label{exam:matched pair}
To continue Example~\ref{exam:ES problem}, define $e_{2}\succ_{V} e_{2}:=e_{2}$, $e_{2}\prec_{V} e_{2}:=0$ on $V={\bf k}\{e_{2}\}$. Then $(V,\succ_{V},\prec_{V})$ is a dendriform algebra by Example~\ref{exam:dendriform algebra}~\ref{exam:dendriform algebra1}.
By Table ~\ref{table:classifying flag datum} in Example~\ref{exam:ES problem}, we obtain all the matched pairs of $D$ and $V$ as follows: by Proposition ~\ref{prop:matched pair unified product} and Eqs. ~(\ref{formulas:flag and extending structure})~(\ref{formulas:flag and D}), collecting all the flag datums $(\bar{l}_{1},\bar{l}_{2},\bar{r}_{1},\bar{r}_{2},\bar{p}_{1},\bar{p}_{2},\bar{q}_{1},\bar{q}_{2},
  \bar{a}_{1},\bar{a}_{2},\bar{k}_{1},\bar{k}_{2})$ such that $\bar{a}_{1}=\bar{a}_{2}=\bar{k}_{2}=0$ and $\bar{k}_{1}=1$, then we obtain Table ~\ref{table:matched pairs} and each row in Table ~\ref{table:matched pairs} defines a matched pair $(D,V,\triangleright_{1},\triangleright_{2},\triangleleft_{1},\triangleleft_{2},$ $\rightharpoonup_{1},\rightharpoonup_{2},\leftharpoonup_{1},\leftharpoonup_{2})$.

\begin{table}[htbp]
  \centering
  \renewcommand\arraystretch{1.6}
\begin{tabular}{|c|c|c|c|c|c|c|c|c|c|c|c|c|c|c|c|c|c|c|c|c|c|c|c|c|c|}
  \hline
  $\bar{l}_{1}$ & $\bar{l}_{2}$ & $\bar{r}_{1}$ & $\bar{r}_{2}$ & $\bar{p}_{1}$ & $\bar{p}_{2}$ & $\bar{q}_{1}$ & $\bar{q}_{2}$ & $\bar{a}_{1}$ & $\bar{a}_{2}$ & $\bar{k}_{1}$ & $\bar{k}_{2}$ \\
  \hline
$1$ & $-1$ & $0$ & $0$ & $1$ & $0$ & $0$ & $1$  &\multirow{10}*{0}&\multirow{10}*{0}&\multirow{10}*{1}&\multirow{10}*{0}\\
  \cline{1-8}
$0$ & $0$ & $1$ & $0$ & $0$ & $0$ & $1$ & $0$&&&&\\
  \cline{1-8}
$1$ & $0$ & $0$ & $0$ & $1$ & $0$ & $0$ & $0$&&&&\\
  \cline{1-8}
$1$ & $0$ & $0$ & $1$ & $1$ & $-1$ & $0$ & $0$&&&&\\
  \cline{1-8}
$1$ & $-1$ & $0$ & $1$ & $1$ & $-1$ & $0$ & $1$&&&&\\
  \cline{1-8}
 $0$ & $0$ & $0$ & $0$ & $1$ & $0$ & $1$ & $0$&&&&\\
  \cline{1-8}
 $0$ & $0$ & $0$ & $0$ & $1$ & $0$ & $0$ & $1$&&&&\\
  \cline{1-8}
 $1$ & $0$ & $1$ & $0$ & $0$ & $0$ & $0$ & $0$&&&&\\
  \cline{1-8}
 $1$ & $0$ & $0$ & $1$ & $0$ & $0$ & $0$ & $0$&&&&\\
  \cline{1-8}
  $0$ & $0$ & $0$ & $0$ & $0$ & $0$ & $0$ & $0$&&&&\\
  \hline
\end{tabular}
\\[5pt]
  \caption{matched pairs of $D$ and $V$}\label{table:matched pairs}
\end{table}
\end{exam}
\allowdisplaybreaks{
\begin{defn}\label{defn:bicrossed product}
Let $\Omega(D,V)=\big(\triangleright_{1},\triangleright_{2},\triangleleft_{1},\triangleleft_{2},
\rightharpoonup_{1},\rightharpoonup_{2},\leftharpoonup_{1},\leftharpoonup_{2},\succ_{V},\prec_{V})$ be a dendriform extending structure of $D$ through $V$, the associated unified product $\dnv$ is called the {\bf bicrossed product} associated to the matched pair $(D,V,\triangleright_{1},\triangleright_{2},\triangleleft_{1},\triangleleft_{2},\rightharpoonup_{1},\rightharpoonup_{2},\leftharpoonup_{1},\leftharpoonup_{2})$, denoted by $D\bowtie V$.
\end{defn}
\begin{rem}\label{remark:bicrossed product}
According to Theorem ~\ref{thm:datum and unified product}, the bicrossed product $D\bowtie V$ consists of the vector space $D\times V$ together with a dendriform algebra structure, for $a,b\in D$ and $x,y\in V$, defined as follows:
\begin{align*}
(a,x)\succeq(b,y)&:=(a\succ b+a\leftharpoonup_{1} y+x\rightharpoonup_{1} b,\,a\triangleright_{1} y+x\triangleleft_{1} b+x\succ_{V} y);\\
(a,x)\preceq(b,y)&:=(a\prec b+a\leftharpoonup_{2} y+x\rightharpoonup_{2} b,\,a\triangleright_{2} y+x\triangleleft_{2} b+x\prec_{V} y).
\end{align*}
\end{rem}

Under the bijection map $\Upsilon_{1}$ (or map $\Upsilon_{2}$) in Proposition~\ref{prop:bijection of extension and unified product}, we have the following proposition.
\begin{prop}\label{prop:factor through matched pair}
Let $D$ and $V$ be two dendriform algebras, vector space $E=D+V$, $D\cap V=\{0\}$, and $\rho:E\rightarrow D$ the retraction associated to V. Suppose that $D\subset E $ is an extension of dendriform algebras, then the following conditions are equivalent:
\begin{enumerate}
\item\label{prop:factor through matched pair1} the dendriform algebra $E$ factorizes through $D$ and $V$.
\item\label{prop:factor through matched pair2} the unified product corresponding to the dendriform extending structure $\Omega(D,V)=\Upsilon_{1}(E)$ is a bicrossed product $D\bowtie V$.
\end{enumerate}
\end{prop}

\begin{proof}
First, we prove that ~\ref{prop:factor through matched pair1} $\Rightarrow$ ~\ref{prop:factor through matched pair2}.
Since $V$ is a subalgebra of $E$, for all $x,y\in V$ and by Eq. ~(\ref{formulas:extension to unified}), we have
\begin{align*}
f_{1}(x,y)&=\rho(x\succ_{E} y)=\rho(x\succ_{V} y)=0,\text{ i.e., $f_{1}$ is trivial};\\
f_{2}(x,y)&=\rho(x\prec_{E} y)=\rho(x\prec_{V} y)=0,\text{ i.e., $f_{2}$ is trivial}.
\end{align*}
According to Proposition ~\ref{prop:matched pair unified product} and Definition ~\ref{defn:bicrossed product}, the corresponding unified product is a bicrossed product $D\bowtie V$.

Second, we prove that ~\ref{prop:factor through matched pair2} $\Rightarrow$ ~\ref{prop:factor through matched pair1}. Here we just need to prove that $V$ is a subalgebra of $E$, by Eq. ~(\ref{formulas:unified to extension}), we know that $x\succ_{E} y=x\succ_{V} y$, $x\prec_{E} y=x\prec_{V} y$, $x,y\in V$, so $V$ is a subalgebra of $E$.
\end{proof}
}


\begin{exam}\label{exam:matched pairs2}
To continue Example~\ref{exam:ES problem}, define $e_{2}\succ_{V} e_{2}:=0$ and $e_{2}\prec_{V} e_{2}:=e_{2}$ on $V={\bf k}\{e_{2}\}$. Then $(V,\succ_{V},\prec_{V})$ is a dendriform algebra by Example~\ref{exam:dendriform algebra}~\ref{exam:dendriform algebra2}. Similar to Example ~\ref{exam:matched pair}, we obtain all the matched pairs of $D$ and $V$ in Table ~\ref{table:matched pairs2}. Since each flag datum $(\bar{l}_{1},\bar{l}_{2},\bar{r}_{1},\bar{r}_{2},\bar{p}_{1},\bar{p}_{2},\bar{q}_{1},\bar{q}_{2},
  \bar{a}_{1},\bar{a}_{2},$ $\bar{k}_{1},\bar{k}_{2})$ comes from different equivalent classes, these matched pairs are not equivalent to each other.
\begin{table}[H]
  \centering
  \renewcommand\arraystretch{1.6}
\begin{tabular}{|c|c|c|c|c|c|c|c|c|c|c|c|c|c|c|c|c|c|c|c|c|c|c|c|c|c|}
  \hline
  $\bar{l}_{1}$ & $\bar{l}_{2}$ & $\bar{r}_{1}$ & $\bar{r}_{2}$ & $\bar{p}_{1}$ & $\bar{p}_{2}$ & $\bar{q}_{1}$ & $\bar{q}_{2}$ & $\bar{a}_{1}$ & $\bar{a}_{2}$ & $\bar{k}_{1}$ & $\bar{k}_{2}$ \\
  \hline
$1$ & $-1$ & $0$ & $0$ & $0$ & $0$ & $0$ & $0$  &\multirow{3}*{0}&\multirow{3}*{0}&\multirow{3}*{0}&\multirow{3}*{1}\\
  \cline{1-8}
$1$ & $0$ & $0$ & $1$ & $0$ & $0$ & $0$ & $0$&&&&\\
  \cline{1-8}
$0$ & $0$ & $0$ & $0$ & $0$ & $0$ & $0$ & $0$&&&&\\
  \hline
\end{tabular}
  \caption{matched pairs of $D$ and $V$}\label{table:matched pairs2}
\end{table}
\noindent By Eq.~(\ref{formulas:flag to extension}), for each row, we can obtain an extension $\cee$ on $E={\bf k}\{e_{1},e_{2}\}$. By Proposition~\ref{prop:factor through matched pair}, we obtain all the dendriform algebras $\cee$ that factorizes through $D$ and $V$.

\end{exam}
\subsection{Cocycle semidirect products of dendriform algebras}
In this subsection, we mainly consider cocycle semidirect systems and cocycle semidirect products, which are special cases of dendriform extending structures and unified products, respectively.

\begin{defn}
A {\bf cocycle semidirect system} of dendriform algebras is a system $(D,V,\triangleright_{1},\triangleright_{2},\triangleleft_{1}$, $\triangleleft_{2},f_{1},f_{2},\succ_{V},\prec_{V})$ consisting of a dendriform algebra $D$, a $D$ bimodule $\big(V,\triangleright_{1},\triangleright_{2},\triangleleft_{1},
 \triangleleft_{2}\big)$, and four bilinear maps:
\begin {align*}
f_{1},f_{2}:\,V\times V\rightarrow D,\,\,\,\,\,\,\succ_{V},\prec_{V}:\,V\times V\rightarrow V,
\end{align*}
satisfying the following conditions for all $a\in D$ and $x,y,z\in V$:
\allowdisplaybreaks{
         \begin{align*}
        (C1)\,\,&f(x,y)\succ a=f_{1}\big( x,y\triangleleft_{1} a\big);\,\,\,\,
        f_{2}(x,y)\prec a=f_{2}\big( x,y\triangleleft a\big);\,\,\,\,
        f_{1}(x,y)\prec a=f_{1}\big( x,y\triangleleft_{2} a\big);\\
        (C2)\,\,&(x\star_{V} y)\triangleleft_{1} a=x\succ_{V}(y\triangleleft_{1} a);\,\,\,\,
                (x\prec_{V} y)\triangleleft_{2} a=x\prec_{V}(y\triangleleft a);\,\,\,\,
                (x\succ_{V} y)\triangleleft_{2} a=x\succ_{V}(y\triangleleft_{2} a);\\
        (C3)\,\,&f_{1}\big(a\triangleright x, y\big)= a\succ f_{1}(x,y);\,\,\,\,
                f_{2}\big(a\triangleright_{2} x, y\big)= a\prec f(x,y);\,\,\,\,                     f_{2}\big(a\triangleright_{1} x, y\big)= a\succ f_{2}(x,y);\\
        (C4)\,\,&(a\triangleright x)\succ_{V} y= a\triangleright_{1} (x\succ_{V} y);\,\,\,\,
                (a\triangleright_{2} x)\prec_{V} y= a\triangleright_{2} (x\star_{V} y);\,\,\,\,
                (a\triangleright_{1} x)\prec_{V} y= a\triangleright_{1} (x\prec_{V} y);\\
        (C5)\,\,&f_{1}\big(x\triangleleft a,y\big)=f_{1}\big(x,a\triangleright_{1} y\big);\,\,\,\,
               f_{2}\big(x\triangleleft_{2} a,y\big)=f_{2}\big(x,a\triangleright y\big);\,\,\,\,
        f_{2}\big(x\triangleleft_{1} a,y\big)=f_{1}\big(x,a\triangleright_{2} y\big);\\
        (C6)\,&(x\triangleleft a)\succ_{V} y=x\succ_{V}(a\triangleright_{1} y);\,\,\,\,
        (x\triangleleft_{2} a)\prec_{V} y=x\prec_{V}(a\triangleright y);\,\,\,\,
        (x\triangleleft_{1} a)\prec_{V} y=x\succ_{V}(a\triangleright_{2} y);\\
        (C7)\,&f_{1}\big(x\star_{V} y,z\big)=f_{1}\big(x,y\succ_{V} z\big);\,\,\,\,
        f_{2}\big(x\prec_{V} y,z\big)=f_{2}\big(x,y\star_{V} z\big);\,\,\,\,
        f_{2}\big(x\succ_{V} y,z\big)=f_{1}\big(x,y\prec_{V} z\big);\\
        (C8)\,&f(x,y)\triangleright_{1} z=x\triangleleft_{1} f_{1}(y,z)+x\succ_{V} (y\succ_{V} z)-(x\star_{V} y)\succ_{V} z; \\
        &f_{2}(x,y)\triangleright_{2} z=x\triangleleft_{2} f(y,z)+x\prec_{V} (y\star_{V} z)-(x\prec_{V} y)\prec_{V} z;\\
        &f_{1}(x,y)\triangleright_{2} z=x\triangleleft_{1} f_{2}(y,z)+x\succ_{V} (y\prec_{V} z)-(x\succ_{V} y)\prec_{V} z.
    \end{align*}
    }
\end{defn}
\begin{exam}
   To continue Example~\ref{exam:dendriform algebra}~\ref{exam:dendriform algebra1} and suppose $V={\bf k}\{e_{2}\}$, define eight bilinear maps as follows:
\begin{align*}
\begin{aligned}
e_{1}\triangleright_{1}e_{2}&=e_{2},\\
f_{1}(e_{2},e_{2})&=e_{1},\\
\end{aligned}\,\,\,\,\,\,\,
\begin{aligned}
e_{1}\triangleright_{2}e_{2}&=0,\\
f_{2}(e_{2},e_{2})&=0,\\
\end{aligned}\,\,\,\,\,\,\,
\begin{aligned}
e_{2}\triangleleft_{1}e_{1}&=e_{2},\\
e_{2}\succ_{V} e_{2}&=2e_{2},\\
\end{aligned}\,\,\,\,\,\,\,
\begin{aligned}
e_{2}\triangleleft_{2}e_{1}&=0,\\
e_{2}\prec_{V} e_{2}&=0.
\end{aligned}
\end{align*}
By directly computing, the system $(D,V,\triangleright_{1},\triangleright_{2},\triangleleft_{1}$, $\triangleleft_{2},f_{1},f_{2},\succ_{V},\prec_{V})$ is a cocycle semidirect product of dendriform algebras.

\end{exam}

\begin{prop}
\label{cor:system and cocycle semidirect product}
 The extending datum $\Omega(D,V)=\big(\triangleright_{1},\triangleright_{2},\triangleleft_{1},
 \triangleleft_{2},f_{1},f_{2},\succ_{V},\prec_{V})$ is a dendriform extending structure of $D$ through $V$, if and only if, $(D,V,\triangleright_{1},\triangleright_{2},\triangleleft_{1},
 \triangleleft_{2},f_{1},f_{2},\succ_{V},\prec_{V})$ is a cocycle semidirect system of dendriform algebras.
\end{prop}
\begin{proof}
It's obtained directly by Theorem ~\ref{thm:datum and unified product} if the maps $\rightharpoonup_{1},\rightharpoonup_{2},\leftharpoonup_{1}$ and $\leftharpoonup_{2}$ are all trivial.
\end{proof}
\begin{defn}Let $\Omega(D,V)=\big(\triangleright_{1},\triangleright_{2},\triangleleft_{1},
 \triangleleft_{2},f_{1},f_{2},\succ_{V},\prec_{V})$ be a dendriform extending structure of $D$ through $V$. The associated unified product $\dnv$ is called the {\bf cocycle semidirect product} associated to the cocycle semidirect system $(D,V,\triangleright_{1},\triangleright_{2},\triangleleft_{1}$, $\triangleleft_{2},f_{1},f_{2},\succ_{V},\prec_{V})$, denoted by $D\propto^{f} V$.
\end{defn}
\begin{rem}
According to Theorem ~\ref{thm:datum and unified product}, the cocycle semidirect product $D\propto^{f} V$ consists of the vector space $D\times V$, together with a dendriform algebra structure defined as follows with $a,b\in D$ and $x,y\in V$,
\allowdisplaybreaks{
\begin{align*}
(a,x)\succeq(b,y)&:=\big(a\succ b+f_{1}(x,y),a\triangleright_{1} y+x\triangleleft_{1} b+x\succ_{V} y\big);\\
(a,x)\preceq(b,y)&:=\big(a\prec b+f_{2}(x,y),a\triangleright_{2} y+x\triangleleft_{2} b+x\prec_{V} y\big).
\end{align*}
}
\end{rem}

By Theorem ~\ref{thm:datum and unified product}, we know that $\big(V,\triangleright_{1},\triangleright_{2},\triangleleft_{1},\triangleleft_{2}\big)$ is a $D$ bimodule and $D$ is a regular module. Then $E$ has a natural $D$ bimodule structure. Applying the bijection map $\Upsilon_{1}$ (or map $\Upsilon_{2}$) in Proposition~\ref{prop:bijection of extension and unified product}, we have the following proposition.

\begin{prop}\label{prop:left module morphism retraction}
Let $D\subset E $ be an extension of dendriform algebras with a retraction $\rho:E\rightarrow D$, and $V=\mathrm{ker}\,\rho$.
Then the following conditions are equivalent.
\begin{enumerate}
\item\label{prop:left module morphism retraction1} The map $\rho$ is a left (resp. right) $D$ module morphism.
\item\label{prop:left module morphism retraction2} The maps $\leftharpoonup_{1}$ and $\leftharpoonup_{2}$ (resp. $\rightharpoonup_{1}$ and $\rightharpoonup_{2}$) of the dendriform extending structure $\Omega(D,V)$ defined by $\Upsilon_{1}(E)$ are both trivial.
\end{enumerate}
\end{prop}

%

%

\begin{proof}
The linear retraction $\rho$ is a left $D$ module morphism if and only if the following Eq. ~(\ref{zformulas(not):module morphism in prop}) holds for all $a,b\in D$ and $x\in V$,
\allowdisplaybreaks{
\begin{align}\label{zformulas(not):module morphism in prop}
\begin{split}
    \rho\big(b\succ_{E}(a+x)\big)&=b\succ\, \rho(a+x);\\
    \rho\big(b\prec_{E}(a+x)\big)&=b\prec\, \rho(a+x).
\end{split}
\end{align}
By Eq. ~(\ref{formulas:unified to extension}), we have
\begin{align*}
    0=&\rho\big(b\succ_{E}(a+x)\big)-b\succ\,\rho(a+x)\\
    =&\rho(b\succ a+b\leftharpoonup_{1} x+b\triangleright_{1} x)-b\succ a\\
    =&b\succ a+b\leftharpoonup_{1} x-b\succ a\\
    =&b\leftharpoonup_{1} x;\\
    0=&\rho\big(b\prec_{E}(a+x)\big)-b\prec\,\rho(a+x)\\
    =&\rho(b\prec a+b\leftharpoonup_{2} x+b\triangleright_{2} x)-b\prec a\\
    =&b\prec a+b\leftharpoonup_{2} x-b\prec a\\
    =&b\leftharpoonup_{2} x.
    \end{align*}
}Hence the linear retraction $\rho$ is a left $D$ module morphism if and only if the maps $\leftharpoonup_{1}$ and $\leftharpoonup_{2}$ are both trivial. The right case can be similarly proved.
    \end{proof}

\begin{cor}\label{cor:bimodule retraction}
Let $D\subset E $ be an extension of dendriform algebras with a retraction $\rho:E\rightarrow D$, and $V=\mathrm{ker}\,\rho$.
Then the following conditions are equivalent.
\begin{enumerate}
\item The map $\rho$ is a $D$ bimodule morphism.
\item The unified product corresponding to the dendriform extending structure $\Omega(D,V)=\Upsilon_{1}(E)$ is a cocycle semidirect product $D\propto^{f} V$.
\end{enumerate}
\end{cor}
\begin{proof}
The corollary can be obtained by Proposition ~\ref{prop:left module morphism retraction} and Proposition~\ref{cor:system and cocycle semidirect product}.
\end{proof}

\subsection{Nonabelian semidirect products of dendriform algebras}
In this subsection, we consider nonabelian semidirect systems and nonabelian semidirect products, which are special cases of dendriform extending structures and unified products, respectively.

\begin{defn}
A {\bf nonabelian semidirect system} of dendriform algebras is a system $(D,V,\triangleright_{1}$, $\triangleright_{2},\triangleleft_{1},\triangleleft_{2},\succ_{V},\prec_{V})$ consisting of two dendriform algebras $D$ and $V$, and four bilinear maps
\begin {align*}
\triangleleft_{1},\triangleleft_{2}:\,V\times D\rightarrow V,\,\,\,\,\,\,\triangleright_{1},\triangleright_{2}:\,D\times V\rightarrow V
\end{align*}
such that $\big(V,\triangleright_{1},\triangleright_{2},\triangleleft_{1},\triangleleft_{2}\big)$ is a $D$ bimodule, and the following conditions hold for all $a\in D$ and $x,y\in V$,
\allowdisplaybreaks{
    \begin{align*}
        (S1)\,\,&(x\star_{V} y)\triangleleft_{1} a=x\succ_{V}(y\triangleleft_{1} a);\,\,\,\,
                (x\prec_{V} y)\triangleleft_{2} a=x\prec_{V}(y\triangleleft a);\,\,\,\,
                (x\succ_{V} y)\triangleleft_{2} a=x\succ_{V}(y\triangleleft_{2} a);\\
        (S2)\,\,&(a\triangleright x)\succ_{V} y= a\triangleright_{1} (x\succ_{V} y);\,\,\,\,
                (a\triangleright_{2} x)\prec_{V} y= a\triangleright_{2} (x\star_{V} y);\,\,\,\,
                (a\triangleright_{1} x)\prec_{V} y= a\triangleright_{1} (x\prec_{V} y);\\
        (S3)\,\,&(x\triangleleft a)\succ_{V} y=x\succ_{V}(a\triangleright_{1} y);\,\,\,\,
        (x\triangleleft_{2} a)\prec_{V} y=x\prec_{V}(a\triangleright y);\,\,\,\,
        (x\triangleleft_{1} a)\prec_{V} y=x\succ_{V}(a\triangleright_{2} y).
    \end{align*}
    }
\end{defn}

\begin{prop}\label{cor:system and semidirect product}
The extending datum $\Omega(D,V)=(\triangleright_{1},\triangleright_{2},\triangleleft_{1},\triangleleft_{2},\succ_{V},\prec_{V})$ is a dendriform extending structure of $D$ through $V$, if and only if, $(D,V,\triangleright_{1},\triangleright_{2},\triangleleft_{1},\triangleleft_{2},\succ_{V},\prec_{V})$ is a nonabelian semidirect system of dendriform algebras.
\end{prop}

\begin{proof}
It's obtained directly by Theorem ~\ref{thm:datum and unified product} when the maps $\rightharpoonup_{1},\rightharpoonup_{2},\leftharpoonup_{1},\leftharpoonup_{2},f_{1}$ and $f_{2}$ are all trivial.
\end{proof}

\begin{defn}Let $\Omega(D,V)=(\triangleright_{1},\triangleright_{2},\triangleleft_{1},\triangleleft_{2},\succ_{V},\prec_{V})$ be a dendriform extending structure of $D$ through $V$. The associated unified product $\dnv$ is called the {\bf nonabelian semidirect product} associated to the nonabelian semidirect system $(D,V,\triangleright_{1},\triangleright_{2},\triangleleft_{1},\triangleleft_{2},\succ_{V},\prec_{V})$, denoted by $D\propto V$.
\end{defn}
\begin{rem}
By Eq.~\eqref{formulas:unified product}, the nonabelian semidirect product $D\propto V$ consists of the vector space $D\times V$, together with a dendriform algebra structure defined as follows:
\allowdisplaybreaks{
\begin{align*}
(a,x)\succeq(b,y)&:=(a\succ b,a\triangleright_{1} y+x\triangleleft_{1} b+x\succ_{V} y);\\
(a,x)\preceq(b,y)&:=(a\prec b,a\triangleright_{2} y+x\triangleleft_{2} b+x\prec_{V} y),\,\text{for}\, a,b\in D, x,y\in V.
\end{align*}
}
\end{rem}

%
%
%

Under the bijection map $\Upsilon_{1}$ (or map $\Upsilon_{2}$) in Proposition~\ref{prop:bijection of extension and unified product}, we have the following proposition.

\begin{prop}\label{prop:algebra morphism retraction}
Let $D\subset E $ be an extension of dendriform algebras with a retraction $\rho:E\rightarrow D$, and $V=\mathrm{ker}\,\rho$.
Then the following conditions are equivalent.
\begin{enumerate}
\item The map $\rho$ is a morphism of dendriform algebras.
\item The unified product corresponding to the dendriform extending structure $\Omega(D,V)=\Upsilon_{1}(E)$ is a nonabelian semidirect product $D\propto V$, the triple $(V,\succ_{V},\prec_{V})$ automatically becomes a subalgebra of $E$.
\end{enumerate}
In this case, the canonical projection $\pi:E\rightarrow V$ is a morphism of dendriform algebras, if and only if, the corresponding unified product is a direct sum $D\oplus V$.
\end{prop}

%

\begin{proof}
The linear retraction $\rho$ is a morphism of dendriform algebras if and only if the following Eq. ~(\ref{zformula(not):algebra morphism in prop}) holds for all $a,b\in D$ and $x,y\in V$,
\allowdisplaybreaks{
\begin{align}\label{zformula(not):algebra morphism in prop}
\begin{split}
    \rho\big((a+x)\succ_{E}(b+y)\big)&=\rho(a+x)\succ\,\rho(b+y);\\
    \rho\big((a+x)\prec_{E}(b+y)\big)&=\rho(a+x)\prec\,\rho(b+y).
\end{split}
\end{align}
}By direct computation, the linear retraction $\rho$ is a morphism of dendriform algebras if and only if the maps $\rightharpoonup_{1},\rightharpoonup_{2},\leftharpoonup_{1},\leftharpoonup_{2},f_{1}$ and $f_{2}$ are all trivial. According to Proposition ~\ref{cor:system and semidirect product}, the corresponding unified product $\dnv$ is a nonabelian semidirect product $D\propto V$, and the triple $(V,\succ_{V},\prec_{V})$ is a subalgebra of $E$ by Proposition ~\ref{prop:matched pair unified product} and Proposition ~\ref{prop:factor through matched pair}.

Similarly, we can prove that the projection $\pi$ is a morphism of dendriform algebras if and only if the maps $\triangleright_{1},\triangleright_{2},\triangleleft_{1}$ and $\triangleleft_{2}$ are all trivial, i.e., the corresponding unified product is a direct sum $D\oplus V$.
\end{proof}
\subsection{Splitting of short exact sequences}
In this subsection, we mainly consider the splitting of short exact sequences and always assume that the short exact sequence can be defined in the domain category.
\begin{defn}~\cite{R09} Let $\mathfrak{C}$ be a category. Suppose that
$$\xymatrix{
   0\ar[r] &A \ar@<1ex>[r]^{\iota} & B \ar@<1ex>[r]^{\pi}\ar@{-->}@<1ex>[l]^{\rho}& C \ar[r]\ar@{-->}@<1ex>[l]^{s} &0} $$
is a short exact sequence in the category $\mathfrak{C}$.
\begin{enumerate}
\item The short exact sequence is called {\bf (right) splitting} if there exists a morphism $s:C\rightarrow B$, such that $\pi\,s=\mathrm{Id}_{C}$.
\item The short exact sequence is called {\bf left splitting} if there exists a morphism $\rho:B\rightarrow A$, such that $\rho\,\iota=\mathrm{Id}_{A}$.
\end{enumerate}
\end{defn}
In fact, we have studied the left splitting of the short exact sequence in the category $_{D}\mathfrak{M}\,(\mathfrak{M}_{D})$ of all left (right) $D$ modules in Proposition ~\ref{prop:left module morphism retraction}. More specifically, the maps $i$ and $\pi$ are both $D$ bimodule morphisms in Diagram ~(\ref{diagram:equivalent cohomologous ee}), when the map $\rho$ is a left (right) $D$ module morphism, then the first row is left splitting in the category $_{D}\mathfrak{M}\,(\mathfrak{M}_{D})$. Moreover, we have studied the left splitting in the category $_{D}\mathfrak{M}_{D}$ of all $D$ bimodules in Corollary ~\ref{cor:bimodule retraction}.


Analogous to the category of all left $R$ (a ring) modules~\cite{R09}, we have the following proposition.

\begin{prop}\label{prop:module left right splitting}Let D be a dendriform algebra. Suppose that 
$$\xymatrix{
   0\ar[r] &(A,\triangleright^A_{1},\triangleright^A_{2}) \ar@<1ex>[r]^{\iota} & (B,\triangleright^B_{1},\triangleright^B_{2}) \ar@<1ex>[r]^{\pi}\ar@{-->}@<1ex>[l]^{\rho}& (C,\triangleright^C_{1},\triangleright^C_{2}) \ar[r]\ar@{-->}@<1ex>[l]^{s} &0} $$
is a short exact sequence in the category $_{D}\mathfrak{M}$. Then the left splitting of the sequence is equivalent to the right splitting.
\end{prop}
\begin{proof}\begin{enumerate}
\item \label{left right splitting1} We prove that the left splitting of the sequence implies the right splitting.
Since the sequence is left splitting, i.e., there exists a morphism of dendriform algebras $\rho:B\rightarrow A$ such that $\rho\,\iota=\mathrm{Id}_{A}$.
We establish a map:
\begin{align*}
s:C&\rightarrow B\\
x&\mapsto s(x):=(\mathrm{id}_{B}-\iota\,\rho)\big(\pi^{-1}(x)\big),\,x\in C,
\end{align*}
where the notation $\pi^{-1}(x)$ means to select a preimage of $x$, and the linear map $(\mathrm{id}_{B}-\iota\,\rho):B\rightarrow B$ is defined by $(\mathrm{id}_{B}-\iota\,\rho)(y):=\mathrm{id}_{B}(y)-\iota\big(\rho(y)\big)$ with $y\in B$. Then
\begin{align}
\pi\,(\mathrm{id}_{B}-\iota\,\rho)=&\pi-(\pi\,\iota)\,\rho=\pi;\label{formulas:pi and rho1}\\
\rho\,(\mathrm{id}_{B}-\iota\,\rho)=&\rho-(\rho\,\iota)\,\rho=\rho-\rho=0.\label{formulas:pi and rho2}
\end{align}
%

First, we prove that the map $s$ is well defined, i.e., it is independent of the choice of the preimages of $x$. Suppose that $y_{1}, y_{2}\in B$ are preimages of $x \in C$, then $\pi(y_{1})=\pi(y_{2})=x$. By Eq.~\eqref{formulas:pi and rho1}, we have
\begin{align*}
\pi\Big((\mathrm{id}_{B}-\iota\,\rho)(y_{1})-(\mathrm{id}_{B}-\iota\,\rho)(y_{2})\Big)
=&\pi\Big((\mathrm{id}_{B}-\iota\,\rho)(y_{1})\Big)-\pi\Big((\mathrm{id}_{B}-\iota\,\rho)(y_{2})\Big)
=\pi(y_{1})-\pi(y_{2})
=0.
\end{align*}
Hence $(\mathrm{id}_{B}-\iota\,\rho)(y_{1})-(\mathrm{id}_{B}-\iota\,\rho)(y_{2})\in \mathrm{ker}\, \pi=\mathrm{im}\,\iota$, then there exists $z\in A$, such that $(\mathrm{id}_{B}-\iota\,\rho)(y_{1})-(\mathrm{id}_{B}-\iota\,\rho)(y_{2})=\iota(z)$. By Eq.~\eqref{formulas:pi and rho2}, we have
\begin{align*}
z=\rho\big(\iota(z)\big)=\rho\Big((\mathrm{id}_{B}-\iota\,\rho)(y_{1})-(\mathrm{id}_{B}-\iota\,\rho)(y_{2})\Big)
=\rho\Big((\mathrm{id}_{B}-\iota\,\rho)(y_{1})\Big)-\rho\Big((\mathrm{id}_{B}-\iota\,\rho)(y_{2})\Big)
=0.
\end{align*}
Hence $(\mathrm{id}_{B}-\iota\,\rho)(y_{1})-(\mathrm{id}_{B}-\iota\,\rho)(y_{2})=\iota(z)=0$, so the map $s$ is well defined.

Second, we prove that $s$ is a linear map. Suppose that $y,y^{\prime}\in B$ are preimages of $x,x^{\prime}\in C$, respectively. Then $\pi(k_{1}y+k_{2}y^{\prime})=k_{1}\,\pi(y)+k_{2}\,\pi(y^{\prime})=k_{1}x+k_{2}x^{\prime}$ with $k_{1},k_{2} \in \bf k$, hence $k_{1}y+k_{2}y^{\prime}$ is a preimage of $k_{1}x+k_{1}x^{\prime}$, so
\begin{align*}
s(k_{1}x+k_{2}x^{\prime})=&(\mathrm{id}_{B}-\iota\,\rho)\big(\pi^{-1}(k_{1}x+k_{2}x^{\prime})\big)\\
=&(\mathrm{id}_{B}-\iota\,\rho)(k_{1}y+k_{2}y^{\prime})\\
=&k_{1}y+k_{2}y^{\prime}-\iota\big(\rho(k_{1}y+k_{2}y^{\prime})\big)\\
=&k_{1}y-k_{1}\,\iota\big(\rho(y)\big)+k_{2}y^{\prime}-k_{2}\,\iota\big(\rho(y^{\prime})\big)\\
=&k_{1}\,s(x)+k_{2}\,s(x^{\prime}).
\end{align*}

Third, we prove that $\pi\,s=\mathrm{Id}_{C}$ and $\rho\,s=0$. Let $y$ be a preimage of $x$ with $x\in C$, by Eqs.~\eqref{formulas:pi and rho1}-\eqref{formulas:pi and rho2}, we have
\begin{align*}
\pi(s(x))=&\pi\big((\mathrm{id}_{B}-\iota\,\rho)(y)\big)
=\pi(y)=x;\\
\rho(s(x))=&\rho\big((\mathrm{id}_{B}-\iota\,\rho)(y)\big)
=0.
\end{align*}

Finally, we prove that the map $s$ is a morphism of left $D$ modules. Suppose that $y$ is a preimage of $x$ with $x\in C$, since $\pi(a \triangleright_{1}^{B} y)=a \triangleright_{1}^{C} \pi(y)=a \triangleright_{1}^{C} x$, we know that $a \triangleright_{1}^{B} y$ is a preimage of $a \triangleright_{1}^{C} x$. Then we have
\allowdisplaybreaks{
\begin{align*}
s(a \triangleright_{1}^{C} x)=&(\mathrm{id}_{B}-\iota\,\rho)(a \triangleright_{1}^{B} y)\\
=&a \triangleright_{1}^{B} y-\iota\big(\rho(a \triangleright_{1}^{B} y)\big)\\
=&a \triangleright_{1}^{B} y-\iota\big(a \triangleright_{1}^{A} \rho(y)\big)\\
=&a \triangleright_{1}^{B} y-a \triangleright_{1}^{B}\iota\big(\rho(y)\big)\\
=&a \triangleright_{1}^{B} (\mathrm{id}_{B}-\iota\,\rho)(y)\\
=&a \triangleright_{1}^{B} s(x).
\end{align*}
Analogously we can prove that $s(a \triangleright_{2}^{C} x)=a \triangleright_{2}^{B} s(x)$. Hence the sequence is right splitting.

\item We prove that the right splitting of the sequence implies the left splitting. Since the sequence is right splitting, i.e., there exists a left $D$ module morphism $s:C\rightarrow B$ such that $\pi\,s=\mathrm{Id}_{C}$.
    We establish a map:
\begin{align*}
\rho:B&\rightarrow A\\
y&\mapsto \rho(y):=\iota^{-1}\Big(y-s\big(\pi(y)\big)\Big),\,y\in B.
\end{align*}
where the notation $\iota^{-1}(y)$ means the preimage of $y$ with $y\in B$ (the preimage of $y$ may not exist!).

First, we prove that the map $\rho$ is well defined, i.e., we need to prove that $y-s\big(\pi(y)\big)\in \mathrm{im}\,\iota$ with $y\in B$. Since $$\pi\Big(y-s\big(\pi(y)\big)\Big)=\pi(y)-\pi\Big(s\big(\pi(y)\big)\Big)=\pi(y)-(\pi\, s)\big(\pi(y)\big)=\pi(y)-\pi(y)=0,$$
we obtain that $y-s\big(\pi(y)\big)\in \mathrm{ker}\,\pi=\mathrm{im}\,\iota$.

Second, we prove that $\rho$ is a linear map and $\rho\,\iota=\mathrm{Id}_{A}$. Suppose $z=\rho(y)=\iota^{-1}\Big(y-s\big(\pi(y)\big)\Big)$ and $z^{'}=\rho(y^{'})=\iota^{-1}\Big(y^{'}-s\big(\pi(y^{'})\big)\Big)$ (with $y,y^{'}\in B$), then $\iota(z)=y-s\big(\pi(y)\big)$ and $\iota(z^{'})=y^{'}-s\big(\pi(y^{'})\big)$. For all $k_{1},k_{2} \in \bf k$, we have
\begin{align*}
\iota(k_{1}z+k_{2}z^{'})
=&k_{1}\,\iota(z)+k_{2}\,\iota(z^{'})\\
=&k_{1}\,\Big(y-s\big(\pi(y)\big)\Big)+k_{2}\,\Big(y^{'}-s\big(\pi(y^{'})\big)\Big)\\
=&(k_{1}y+k_{2}y^{'})-s\big(\pi(k_{1}y+k_{2} y^{'})\big).
\end{align*}
Hence \[k_{1}z+k_{2}z^{'}=\iota^{-1}\Big((k_{1}y+k_{2}y^{'})-s\big(\pi(k_{1}y+k_{2} y^{'})\big)\Big)=\rho(k_{1}y+k_{2}y^{'}).\]
Also we have
 \[\rho\big(\iota(z)\big)=\iota^{-1}\Big(\iota(z)-s\Big(\pi\big(\iota(z)\big)\Big)\Big)
=\iota^{-1}\Big(\iota(z)-s\big((\pi\,\iota)(z)\big)\Big)
=\iota^{-1}\big(\iota(z)\big)
=z, z\in A.\]

Finally, we prove that the map $\rho$ is a left module morphism. Suppose $z=\rho(y)=\iota^{-1}\Big(y-s\big(\pi(y)\big)\Big)$ with $y\in B$, then $\iota(z)=y-s\big(\pi(y)\big)$. Then we obtain
\begin{align*}
\iota\big(a\triangleright_{1}^{A} \rho(y)\big)
=&\iota(a\triangleright_{1}^{A} z)
=a\triangleright_{1}^{B} \iota(z)\\
=&a\triangleright_{1}^{B} \Big(y-s\big(\pi(y)\big)\Big)\\
=&a\triangleright_{1}^{B} y-a\triangleright_{1}^{B} s\big(\pi(y)\big)\\
=&a\triangleright_{1}^{B} y- s\big(a\triangleright_{1}^{C} \pi(y)\big)\\
=&a\triangleright_{1}^{B} y- s\big(\pi(a\triangleright_{1}^{B} y)\big).
\end{align*}
}
Hence $a\triangleright_{1}^{A} \rho(y)=\iota^{-1}\Big(a\triangleright_{1}^{B} y- s\big(\pi(a\triangleright_{1}^{B} y)\big)\Big)=\rho(a\triangleright_{1}^{B} y)$. Analogously we can prove that $a\triangleright_{2}^{A} \rho(y)=\rho(a\triangleright_{2}^{B} y)$.
\end{enumerate}
\end{proof}

In fact, we have studied the left splitting of the short exact sequence in the category $\mathbf{Dend}$ of all dendriform algebras in Proposition ~\ref{prop:algebra morphism retraction}. More specifically, the map $i$ is also a morphism of dendriform algebras in Diagram ~(\ref{diagram:equivalent cohomologous ee}), when the maps $\rho,\pi$ are both morphisms of dendriform algebras, then the first row is left splitting in the category $\mathbf{Dend}$. In contrast to Proposition ~\ref{prop:module left right splitting}, the right splitting short exact sequence in the category $\mathbf{Dend}$ may not be left splitting, refer to the following example.
\begin{exam}
To continue Example ~\ref{exam:ES problem}, selecting case 14 in Table ~\ref{table:classifying flag datum} and suppose that $\bar{q}_{1}=0,\bar{k}_{1}=1$, we obtain an extension $\cee$ of $D$ by Eq. ~\eqref{formulas:flag to extension}. In the case, the first row of Diagram ~(\ref{diagram:equivalent cohomologous ee}) is a short exact sequence in the category $\mathbf{Dend}$. In fact, the maps $\triangleright_{1},\triangleright_{2},\triangleleft_{1},\triangleleft_{2}$ are trivial by Eqs. ~(\ref{formulas:flag and extending structure}) and~(\ref{formulas:flag and D}), then $(V,\succ_{V},\prec_{V})$ is a dendriform algebra by Eq. (D12) of Theorem ~\ref{thm:datum and unified product}, and $\pi$ is a morphism of dendriform algebras by Proposition ~\ref{prop:algebra morphism retraction}.

To explain that the sequence is right splitting, we define a linear map $s:V\rightarrow E$ by $s(e_{2})=e_{2}$. By directly computing, the map $s$ is a morphism of dendriform algebras and $\pi\,s=\mathrm{Id}_{V}$.

Now we explain that the sequence can not be left splitting. By contradiction, assume that there exists a morphism $\rho:E\rightarrow D$ of dendriform algebras such that $\rho\, i=\mathrm{Id}_{D}$, then $\mathrm{ker}\,\rho$ is a space complement of $D$ in $E$ by Remark ~\ref{remark:retraction and V} ~\ref{remark:retraction and V1}, and $\mathrm{ker}\,\rho$ is an ideal of dendriform algebra $E$. Let $B$ be a space complement of $D$ in $E$, then the basis of $B$ is the form of $\lambda e_{1}+e_{2}(\lambda\in {\bf k})$. By directly computing, we find each space complement of $D$ in $E$ is not an ideal of dendriform algebra $E$, and this is a contradiction.
\end{exam}




\begin{prop}\label{prop:algebra left right splitting}
Suppose that 
$$\xymatrix{
   0\ar[r] &(A,\succ_{A},\prec_{A}) \ar@<1ex>[r]^{\iota} & (B,\succ_{B},\prec_{B}) \ar@<1ex>[r]^{\pi}\ar@{-->}@<1ex>[l]^{\rho}& (C,\succ_{C},\prec_{C}) \ar[r]\ar@{-->}@<1ex>[l]^{s} &0} $$
is a short exact sequence in the category $\mathbf{Dend}$. Then the left splitting of the sequence implies the right splitting.
\end{prop}

\begin{proof}
\allowdisplaybreaks{
Similar to Item ~\ref{left right splitting1} in the proof of Proposition ~\ref{prop:module left right splitting}, we obtain $\pi\,s=\mathrm{Id}_{C},\,\rho\,s=0$. Here we only need to prove that the map $s$ is a morphism of dendriform algebras. For all $x,x^{\prime}\in C$, we have
\begin{align*}
\pi(s(x\succ_{C} x^{\prime})-s(x)\succ_{B} s(x^{\prime}))
=\pi(s(x\succ_{C} x^{\prime}))-\pi(s(x))\succ_{C} \pi(s(x^{\prime}))
=x\succ_{C} x^{\prime}-x\succ_{C} x^{\prime}
=0.
\end{align*}
Hence $s(x\succ_{C} x^{\prime})-s(x)\succ_{B} s(x^{\prime})\in \mathrm{ker}\, \pi=\mathrm{im}\,\iota$, then there exists $z\in D$, such that $s(x\succ_{C} x^{\prime})-s(x)\succ_{B} s(x^{\prime})=\iota(z)$. Then
\begin{align*}
z=\rho\big(\iota(z)\big)=\rho\big(s(x\succ_{C} x^{\prime})-s(x)\succ_{B} s(x^{\prime})\big)
=\rho\big(s(x\succ_{C} x^{\prime}))-\rho\big(s(x))\succ_{B} \rho\big(s(x^{\prime})\big)
=0.
\end{align*}
Hence $s(x\succ_{C} x^{\prime})-s(x)\succ_{B} s(x^{\prime})=\iota(z)=0$. Similarly, we have $s(x\prec_{C} x^{\prime})-s(x)\prec_{B} s(x^{\prime})=0$.
}
\end{proof}

\section{Classifying complements for dendriform algebras}
 In this section, we define the deformation map and the deformation of dendriform algebras, and theoretically solve the classifying complements problem. The deformation map is often defined on a matched pair, while we define it on a dendriform extending structure (more general case), which is more practical in the classifying complements problem.
%

\begin{defn}
Let $\Omega(D,V)$ be a dendriform extending structure of $D$ through $V$. A linear map $d:V\rightarrow D$ is called a {\bf deformation map} of $\Omega(D,V)$, if the following conditions hold for all $x,y\in V$,
\allowdisplaybreaks{
\begin{align}
d(x)\succ\,d(y)-d(x\succ_{V}y)&=d\big(d(x)\triangleright_{1} y+x \triangleleft_{1} d(y)\big)-d(x)\leftharpoonup_{1} y-x\rightharpoonup_{1} d(y)-f_{1}(x,y);\label{formulas:deformation in defn1}\\
d(x)\prec\,d(y)-d(x\prec_{V}y)&=d\big(d(x)\triangleright_{2} y+x \triangleleft_{2} d(y)\big)-d(x)\leftharpoonup_{2} y-x\rightharpoonup_{2} d(y)-f_{2}(x,y).\label{formulas:deformation in defn2}
\end{align}
}
\end{defn}
We denote by $\mathcal{D}(D,V)$ the set of all deformation maps of $\Omega(D,V)$.

\begin{prop}\label{prop:demormation of X}
Let $\Omega(D,V)$ be a dendriform extending structure of $D$ through $V$, and  $d:V\rightarrow D$ a deformation map of $\Omega(D,V)$. Define two bilinear maps as follows:
\begin{align}
x\succ_{d}y:=&x\succ_{V}y+d(x)\triangleright_{1} y+x \triangleleft_{1} d(y);\label{formulas:deformation1}\\
x\prec_{d}y:=&x\prec_{V}y+d(x)\triangleright_{2} y+x \triangleleft_{2} d(y),\,x,y\in V.\label{formulas:deformation2}
\end{align}
Then $V_{d}:=(V,\succ_{d},\prec_{d})$ is a dendriform algebra, and we call $V_{d}$  the {\bf deformation} of $V$.
\end{prop}

\begin{proof}

By Theorem~\ref{thm:datum and unified product}, here we give the proof by the method of unified product $D \natural V$.  

First, we give some useful identities.
In Diagram ~(\ref{diagram:equivalent cohomologous ue}), define the canonical projection $\pi_{D}:\dnv\rightarrow D$ and the injection $i_{V}:V\rightarrow \dnv$ as follows:
$$\pi_{D}(a,x)=a,\quad i_{V}(x)=(0,x),\,\, a\in D,\,x\in V.$$
Then for all $x,y\in V$, we have
\begin{align}
x\succ_{d}y=&x\succ_{V}y+d(x)\triangleright_{1} y+x \triangleleft_{1} d(y)\nonumber\\
           =&\pi^{\prime}((d(x),x)\succeq (d(y),y));\label{zdeformation in proof1}\quad\text{(by Eq.~\eqref{formulas:unified product})}\\
d(x\succ_{d}y)=&d(x\succ_{V}y+d(x)\triangleright_{1} y+x \triangleleft_{1} d(y))\nonumber\\
              =&d(x)\succ d(y)+x\rightharpoonup_{1} d(y)+d(x)\leftharpoonup_{1} y+f_{1}(x,y)\quad\text{(by Eq.~\eqref{formulas:deformation in defn1})}\nonumber\\
              =&\pi_{D}((d(x),x)\succeq (d(y),y)).\label{zdeformation in proof2}\quad\text{(by Eq.~\eqref{formulas:unified product})}
\end{align}
Similarly, we have
\begin{align}
x\prec_{d}y=&\pi^{\prime}((d(x),x)\preceq (d(y),y));\label{zdeformation in proof3}\\
d(x\prec_{d}y)=&\pi_{D}((d(x),x)\preceq (d(y),y)).\label{zdeformation in proof4}
\end{align}

Second, we prove that $ V_{d}:=(V,\succ_{d},\prec_{d})$ is a dendriform algebra. Here we only prove Eq.~\eqref{definition:dendriform:formula1}, Eqs.~\eqref{definition:dendriform:formula2} -~\eqref{definition:dendriform:formula3} can be proved similarly. For all $x,y,z\in V$, we have
\allowdisplaybreaks{
\begin{align*}
(x\star_{d}y)\succ_{d}z
=&\pi^{\prime}\Big(\big(d(x\star_{d}y),x\star_{d}y\big)\succeq\big(d(z),z\big)\Big)\quad\text{(by Eqs.~\eqref{zdeformation in proof1}})\\
=&\pi^{\prime}\Big(\big(d(x\succ_{d}y+x\prec_{d}y),x\succ_{d}y+x\prec_{d}y\big)\succeq\big(d(z),z\big)\Big)\\
=&\pi^{\prime}\Big(\big(\pi_{D}((d(x),x)\succeq (d(y),y))+\pi_{D}((d(x),x)\preceq (d(y),y)),\\
&\pi^{\prime}((d(x),x)\succeq (d(y),y))+\pi^{\prime}((d(x),x)\preceq (d(y),y))\big)\succeq\big(d(z),z\big)\Big)\\
&\hspace{3cm}\text{(by Eqs.~\eqref{zdeformation in proof1}-\eqref{zdeformation in proof4})}\\
=&\pi^{\prime}\Big(\big(\pi_{D}((d(x),x)\underline{\star} (d(y),y)),\pi^{\prime}((d(x),x)\underline{\star} (d(y),y))\big)\succeq\big(d(z),z\big)\Big)\\
=&\pi^{\prime}\Big(((d(x),x)\underline{\star} (d(y),y)\big)\succeq\big(d(z),z\big)\Big)\\
=&\pi^{\prime}\Big((d(x),x)\succeq \big((d(y),y)\succeq(d(z),z)\big)\Big)\\
&\hspace{3cm}\text{(by $D \natural V$ being a dendriform algebra)}\\
=&\pi^{\prime}\Big((d(x),x)\succeq \big(\pi_{D}\big((d(y),y)\succeq(d(z),z)\big),\pi^{\prime}\big((d(y),y)\succeq(d(z),z)\big)\big)\Big)\\
=&\pi^{\prime}\Big((d(x),x)\succeq \big(d\big(y\succ_{d}z\big),y\succ_{d}z\big)\Big)\quad\text{(by Eqs.~\eqref{zdeformation in proof1} and ~\eqref{zdeformation in proof2})}\\
=&x\succ_{d}(y\succ_{d}z).\quad\text{(by Eq.~\eqref{zdeformation in proof1})}
\end{align*}}
\end{proof}

%

\begin{thm}\label{thm:complement to deformation}
Let $D \subset E$ be an extension of dendriform algebras with a retraction $\rho:E\rightarrow D$, $V:=\mathrm{ker}\,\rho$ and the dendriform extending structure $\Omega(D,V)=\Upsilon_{1}(E)$ by {\rm Remark~\ref{remark:retraction and V}~\ref{rem:eu map}}.  Suppose that $B$ is a $D$ dendriform complement of $E$, then there exists a deformation map $d:V\rightarrow D$ of $\Omega(D,V)$, such that $B\cong V_{d}$ as dendriform algebras.
\end{thm}

\allowdisplaybreaks{
\begin{proof}
First, we construct a map $d:V\rightarrow D$ as follows:
let $\tilde{d}:E\rightarrow D$ be the linear retraction associated to $B$,
and $d:=-\tilde{d}|_{V}$. To prove that $d$ is a deformation map, we only prove Eq.~\eqref{formulas:deformation in defn1}, Eq.~\eqref{formulas:deformation in defn2} can be similarly proved.

For all $x\in V$, since $\tilde{d}(x-\tilde{d}(x))=\tilde{d}(x)-\tilde{d}(x)=0$, we have $x-\tilde{d}(x)=x+d(x)\in \mathrm{ker}\,\tilde{d}=B$ by Remark ~\ref{remark:retraction and V}~\ref{remark:retraction and V1}. Since $B$ is a $D$ dendriform complement of $E$, for all $x,y\in V$, we have \[\big(x+d(x)\big)\succ_{E}\big(y+d(y)\big)\in B=\mathrm{ker}\,\tilde{d}.\]
Also we have
\begin{align*}
0=&\tilde{d}\Big(\big(x+d(x)\big)\succ_{E}\big(y+d(y)\big)\Big)\\
=&\tilde{d}\big(d(x)\succ d(y)+x\rightharpoonup_{1} d(y)+d(x)\leftharpoonup_{1} y+f_{1}(x,y)+x\triangleleft_{1} d(y)+d(x)\triangleright_{1} y+x\succ_{V}y\big)\\
&\hspace{3cm}\text{(by Eq.~\eqref{formulas:unified to extension})}\\
=&d(x)\succ d(y)+x\rightharpoonup_{1} d(y)+d(x)\leftharpoonup_{1} y+f_{1}(x,y)-d(x\triangleleft_{1} d(y)+d(x)\triangleright_{1} y)-d(x\succ_{V}y).
\end{align*}

Second, we construct a linear map:
\begin{align*}
\phi:V_{d}&\rightarrow B\\
x&\mapsto x+d(x),\text{ $x\in V$,}
\end{align*}
the linear map $\phi$ is well defined as $x+d(x)\in B$ for all $x\in V$.

Third, we prove that the map $\phi$ is a linear isomorphism as follows: Suppose $\phi(x)=x+d(x)=0$, $x\in V$, then $x=-d(x)\in D$, we have $x\in D\cap V=\{0\}$ by Remark ~\ref{remark:retraction and V}~\ref{remark:retraction and V1}. Hence the map $\phi$ is injective.
By Remark ~\ref{remark:retraction and V}~\ref{remark:retraction and V1}, we have $\rho(x-\rho(x))=\rho(x)-\rho(x)=0$, hence $x-\rho(x)\in \mathrm{ker}\,\rho=V.$ For any $y \in B$, then we have
$$\phi\big(y-\rho(y)\big)
=y-\rho(y)+d\big(y-\rho(y)\big)
=y-\rho(y)-\tilde{d}\big(y-\rho(y)\big)
=y-\rho(y)+\rho(y)
=y,$$
i.e., the map $\phi$ is surjective.

Finally, we need to prove that the map $\phi$ is a morphism of dendriform algebras. For all $x,y\in V$, we have
\begin{align*}
\phi(x\succ_{d}y)=&x\succ_{d}y+d(x\succ_{d}y)\\
=&x\succ_{V}y+d(x)\triangleright_{1} y+x \triangleleft_{1} d(y)+d(x\succ_{V}y+d(x)\triangleright_{1} y+x \triangleleft_{1} d(y))\quad\text{(by Eq.~\eqref{formulas:deformation1})}\\
=&x\succ_{V}y+d(x)\triangleright_{1} y+x \triangleleft_{1} d(y)+d(x\succ_{V}y)+d(d(x)\triangleright_{1} y+x \triangleleft_{1} d(y))\\
=&x\succ_{V}y+d(x)\triangleright_{1} y+x \triangleleft_{1} d(y)+d(x)\succ d(y)+x\rightharpoonup_{1} d(y)+d(x)\leftharpoonup_{1} y+f_{1}(x,y)\\
&\hspace{3cm}\text{(by Eq.~\eqref{formulas:deformation in defn1})}\\
=&\big(x+d(x)\big)\succ_{E}\big(y+d(y)\big)\quad\text{(by Eq.~\eqref{formulas:unified to extension})}\\
=&\phi(x)\succ_{E}\phi(y).
\end{align*}
Similarly, we can prove $\phi(x\prec_{d}y)
=\phi(x)\prec_{E}\phi(y)$.
%
\end{proof}
}

\begin{rem}\label{rem:complement to deformation}
Under the condition of Theorem ~\ref{thm:complement to deformation}, we denote by $\mathcal{C}(D,E)$ the set of all $D$ dendriform complements of $E$, then we establishes a map:
\begin{align*}
\Delta_{1}:\mathcal{C}(D,E)&\rightarrow \mathcal{D}(D,V)\\
B&\mapsto d:=-\tilde{d}|_{V},
\end{align*}
where $\tilde{d}$ is the retraction associated to $B$.
\end{rem}

\begin{prop}\label{prop:bijection complement to deformation}
The map $\Delta_{1}:\mathcal{C}(D,E)\rightarrow \mathcal{D}(D,V)$ defined in {\rm Remark ~\ref{rem:complement to deformation}} is a bijection.
\end{prop}

\begin{proof}
First, we establish a map:
\begin{align*}
\Delta_{2}:\mathcal{D}(D,V)&\rightarrow \mathcal{C}(D,E)\\
d&\mapsto B:=\mathrm{ker}\, \xi,
\end{align*}
where the linear map $\xi:E\rightarrow D$ is defined by $\xi(a+x)=a-d(x)$,\,$a\in D$, $x\in V$.

Second, we need to prove that the map $\Delta_{2}$ is well defined, i.e., $B$ is a $D$ dendriform complement of $E$. Suppose that the inclusion map $i:D\rightarrow E$ is defined by $i(a)=a,\,a\in D$, then $\xi(i(a))=\xi(a)=a$, i.e., the map $\xi$ is a linear retraction. By Remark ~\ref{remark:retraction and V}~\ref{remark:retraction and V1}, $B$ is a space complement of $D$ in $E$. We only need to prove that the operations $\succ_{E}$ and $\prec_{E}$ is closed on $B$. In fact,
\begin{align*}
B=&\mathrm{ker}\, \xi\\
 =&\{a+x|\xi(a+x)=a-d(x)=0,\,a\in D, x\in V\}\\
  =&\{a+x|a=d(x),\,a\in D, x\in V\}\\
=&\{d(x)+x|x\in V\}.
\end{align*}
Also we have
\begin{align*}
&\big(x+d(x)\big)\succ_{E}\big(y+d(y)\big)\\
=&d(x)\succ d(y)+x\rightharpoonup_{1} d(y)+d(x)\leftharpoonup_{1} y+f_{1}(x,y)+x\succ_{V}y+d(x)\triangleright_{1} y+x \triangleleft_{1} d(y)\\
&\hspace{3cm}\text{(by Eq.~\eqref{formulas:unified to extension})}\\
=&d(d(x)\triangleright_{1} y+x \triangleleft_{1} d(y))+x\succ_{V}y+d(x)\triangleright_{1} y+x \triangleleft_{1} d(y)+d(x\succ_{V}y)\\
&\hspace{3cm}\text{(by Eq.~\eqref{formulas:deformation in defn1})}\\
=&d(x\succ_{V}y+d(x)\triangleright_{1} y+x \triangleleft_{1} d(y))+x\succ_{V}y+d(x)\triangleright_{1} y+x \triangleleft_{1} d(y)\in B.
\end{align*}

Similarly, we can prove that the operations $\prec_{E}$ is closed on $B$.

Third, we consider the composition map $\Delta_{1}\circ\Delta_{2}$
\begin{align*}
\Delta_{1}\circ\Delta_{2}:\mathcal{D}(D,V)&\rightarrow\,\,\,\mathcal{C}(D,E)\,\,\,\rightarrow \,\,\,\mathcal{D}(D,V)\\
d\,\,\,\,\,\,&\mapsto B:=\mathrm{ker}\, \xi\mapsto d^{\prime}:=-\tilde{d}|_{V}.
\end{align*}
For all $x\in V$, $\xi(x+d(x))=-d(x)+d(x)=0$, then we have $x+d(x)\in \mathrm{ker}\, \xi=B$. By Remark ~\ref{remark:retraction and V}~\ref{remark:retraction and V1}, we have $\tilde{d}(x+d(x))=0$, then we have $$d^{\prime}(x)=-\tilde{d}(x)=-\tilde{d}(-d(x)+(x+d(x)))=-(-d(x))=d(x),$$
i.e., $\Delta_{1}\circ\Delta_{2}=\mathrm{id}_{\mathcal{D}(D,V)}$.

Finally, we consider the composition map $\Delta_{2}\circ\Delta_{1}$
\begin{align*}
\Delta_{2}\circ\Delta_{1}:\mathcal{C}(D,E)&\rightarrow\,\,\mathcal{D}(D,V)\,\rightarrow\,\,\,\mathcal{C}(D,E)\\
B\,\,\,\,\,\,&\mapsto d:=-\tilde{d}|_{V}\mapsto B^{\prime}:=\mathrm{ker}\, \xi.
\end{align*}
For all $a\in D,x\in V$, we have \[\xi(a+x)=a-d(x)=a-(-\tilde{d}(x))=\tilde{d}(a)+\tilde{d}(x)=\tilde{d}(a+x).\]
Then $B^{\prime}=\mathrm{ker}\, \xi=\mathrm{ker}\, \tilde{d}=B$, i.e., $\Delta_{2}\circ\Delta_{1}=\mathrm{id}_{\mathcal{C}(D,E)}$.
\end{proof}
We can convert the classifying problem on the set $\mathcal{C}(D,E)$ to the one on the set $\mathcal{D}(D,V)$.



\allowdisplaybreaks{
\begin{defn}\label{definition:equivalent deformation}
Let $\Omega(D,V)$ be a dendriform extending structure of $D$ through $V$. Two deformation maps $d,d^{\prime}:V\rightarrow D$ of $\Omega(D,V)$ are called {\bf equivalent} (denote by $d\sim d^{\prime}$), if there exists a linear automorphism $\delta:V\rightarrow V$ satisfying the following conditions with $x,y\in V$,
\begin{align}
\delta(x)\succ_{V}\delta(y)-\delta(x\succ_{V}y)&=\delta\big(d(x)\triangleright_{1} y\big)+\delta\big(x \triangleleft_{1} d(y)\big)-d^{\prime}\big(\delta(x)\big)\triangleright_{1} \delta(y)-\delta(x)\triangleleft_{1} d^{\prime}\big(\delta(y)\big);\label{formulas:equivalent deformation1}\\
\delta(x)\prec_{V}\delta(y)-\delta(x\prec_{V}y)&=\delta\big(d(x)\triangleright_{2} y\big)+\delta\big(x \triangleleft_{2} d(y)\big)-d^{\prime}\big(\delta(x)\big)\triangleright_{2} \delta(y)-\delta(x)\triangleleft_{2} d^{\prime}\big(\delta(y)\big).\label{formulas:equivalent deformation2}
\end{align}
\end{defn}
}

\begin{prop}\label{proposition:equivalent deformation}
Let $\Omega(D,V)$ be a dendriform extending structure of $D$ through $V$. Suppose that $d,d^{\prime}:V\rightarrow D$ are deformation maps of $\Omega(D,V)$, then $V_{d}\cong V_{d^{\prime}}$ as dendriform algebras if and only if $d\sim d^{\prime}$.
\end{prop}

\allowdisplaybreaks{
\begin{proof}We know that $V_{d}\cong V_{d^{\prime}}$ if and only if there exists an isomorphism of dendriform algebras $\delta:V_{d}\rightarrow V_{d^{\prime}}$, i.e., $\delta:V\rightarrow V$ is a linear isomorphism satisfying the following conditions
    \begin{align}
    \delta(x\succ_{d} y)&=\delta(x)\succ_{d^{\prime}} \delta(y);\label{formulas:morphism in proof1}\\
    \delta(x\prec_{d} y)&=\delta(x)\prec_{d^{\prime}} \delta(y),\,x,y\in V.\label{formulas:morphism in proof2}
    \end{align}
We need to prove that Eqs. ~\eqref{formulas:equivalent deformation1}-\eqref{formulas:equivalent deformation2} hold if and only if Eqs. ~\eqref{formulas:morphism in proof1}-\eqref{formulas:morphism in proof2} hold.
In fact,
    \begin{align*}
    &\delta(x\succ_{d} y)-\delta(x)\succ_{d^{\prime}}\delta(y)\\
    =&\delta\big(x\succ_{V}y+d(x)\triangleright_{1} y+x \triangleleft_{1} d(y)\big)-\big(\delta(x)\succ_{V}\delta(y)+d^{\prime}\big(\delta(x)\big)\triangleright_{1} \delta(y)+\delta(x)\triangleleft_{1} d^{\prime}\big(\delta(y)\big)\big)\\
    &\hspace{3cm}\text{(by Eq.~\eqref{formulas:deformation in defn1} and Eq.~\eqref{formulas:deformation1})}\\
    =&\delta\big(x\succ_{V}y\big)+\delta\big(d(x)\triangleright_{1} y\big)+\delta\big(x \triangleleft_{1} d(y)\big)-\delta(x)\succ_{V}\delta(y)-d^{\prime}\big(\delta(x)\big)\triangleright_{1} \delta(y)-\delta(x)\triangleleft_{1} d^{\prime}\big(\delta(y)\big)\\
=&\delta(x)\succ_{V}\delta(y)-\delta(x\succ_{V}y)-\big(\delta\big(d(x)\triangleright_{1} y\big)+\delta\big(x \triangleleft_{1} d(y)\big)-d^{\prime}\big(\delta(x)\big)\triangleright_{1} \delta(y)-\delta(x)\triangleleft_{1} d^{\prime}\big(\delta(y)\big)\big).
    \end{align*}
Hence Eq. ~\eqref{formulas:equivalent deformation1} holds if and only if Eq. ~\eqref{formulas:morphism in proof1} holds. Similarly, we can prove that Eq. ~\eqref{formulas:equivalent deformation2} holds if and only if Eq.~\eqref{formulas:morphism in proof2} holds.
\end{proof}
}

Now we arrive at our main result in this section.
\begin{thm}
Let $D \subset E$ be an extension of dendriform algebras with a retraction $\rho:E\rightarrow D$, $V:=\mathrm{ker}\,\rho$ and the dendriform extending structure $\Omega(D,V)=\Upsilon_{1}(E)$ by {\rm Remark~\ref{remark:retraction and V}~\ref{rem:eu map}}.  Then
\begin{enumerate}
\item the relation $\sim$ is an equivalence relation on the set $\mathcal{D}(D,V)$.

\item there exists a bijection between the sets $CH^{2}(D,E):=\mathcal{C}(D,E)/\cong$ and $H^{2}(D,V):=\mathcal{D}(D,V)/\sim$. Define
\begin{align*}
\Delta:CH^{2}(D,E)&\rightarrow H^{2}(D,V)\\
     \overline{B}\,\,\,\,\,\,\,\,\,\,\,\,&\mapsto \,\,\,\,\,\,\,\,\,\,\,\,\bar{d}
\end{align*}
 where the map $d=\Delta_{1}(B)$ by {\rm Remark ~\ref{rem:complement to deformation}}, $\overline{B}$ and $\bar{d}$ denote the equivalence class of $B$ and $d$ via $\cong$ and $\sim$, respectively. In particular, the index of $D$ in $E$ is computed by the formula $[E:D]=|H^{2}(D,V)|$.
 \end{enumerate}
\end{thm}

\allowdisplaybreaks{
\begin{proof}
According to Proposition ~\ref{prop:bijection complement to deformation}, there exists a bijective map
\begin{align*}
\Delta_{1}:\mathcal{C}(D,E)&\rightarrow \mathcal{D}(D,V)\\
B\,\,\,\,\,\,&\mapsto \,\,\,\,\,\,d\\
B^{\prime}\,\,\,\,\,\,&\mapsto \,\,\,\,\,\,d^{\prime}.
\end{align*}

By Theorem ~\ref{thm:complement to deformation}, we have $V_{d}\cong B,V_{d^{\prime}}\cong B^{\prime}$, then $V_{d}\cong V_{d^{\prime}}$ if and only if $B\cong B^{\prime}$. By Proposition ~\ref{proposition:equivalent deformation}, the relation $d\sim d^{\prime}$ if and only if $V_{d}\cong V_{d^{\prime}}$, so $d\sim d^{\prime}$ if and only if $B\cong B^{\prime}$. Hence $\sim$ is an equivalence relation on the set $\mathcal{D}(D,V)$, and the map $\Delta_{1}$ induces a bijection $\Delta:CH^{2}(D,E)\rightarrow H^{2}(D,V)$, and $[E:D]=|H^{2}(D,V)|$.
\end{proof}
}

Now we give an example of classifying complements for dendriform algebras.

\begin{exam}\label{exam:deformation map}

To continue Example ~\ref{exam:ES problem}, suppose that $d(e_{2})=\bar{d}e_{1}\,(\bar{d} \in {\bf k})$, $\delta(e_{2})=\bar{\delta}e_{2}\,(\bar{\delta} \neq 0\in {\bf k})$ and $x=y=e_{2}$. For each flag datum of $D$, we obtain an extension $\cee$ of $D$ by Eq. ~\eqref{formulas:flag to extension}. Computing Eqs. ~\eqref{formulas:deformation in defn1}-\eqref{formulas:deformation in defn2} by Eqs.~\eqref{formulas:flag and extending structure} and \eqref{formulas:flag and D}, we obtain all the deformation maps, refer to the third column of Table ~\ref{table:deformation maps and complements}. Computing Eqs.~\eqref{formulas:equivalent deformation1}-\eqref{formulas:equivalent deformation2} by Eqs.~\eqref{formulas:flag and extending structure} and \eqref{formulas:flag and D}, we obtain the equivalence relations of different deformation maps and the results refer to the fourth column of Table ~\ref{table:deformation maps and complements}.
\begin{table}[htbp]
  \centering
  \renewcommand\arraystretch{1.7}
\begin{tabular}{|c|c|c|c|c|c|c|}
 \hline
\multirow{2}*{case} & \multirow{2}*{ $(\bar{l}_{1},\bar{l}_{2},\bar{r}_{1},\bar{r}_{2},\bar{p}_{1},\bar{p}_{2},\bar{q}_{1},\bar{q}_{2},\bar{a}_{1},\bar{a}_{2},\bar{k}_{1},\bar{k}_{2})$} &\multirow{2}*{$\bar{d}$}  &\multirow{2}*{$\bar{\delta}$}   &equivalent &\multirow{2}*{[E:D]}\\
~&~ &~ &~   &class &\\
  \hline
 1 & $(1,-1,0,0,\bar{p}_{1},0,0,\bar{p}_{1},0,\bar{p}^{2}_{1},\bar{p}_{1},-\bar{p}_{1})$ & $-\bar{p}_{1}$   &\diagbox{}{}& $-\bar{p}_{1}$  &1\\

    \hline
 \multirow{2}*{2} &   $(1,-1,0,0,\bar{p}_{1},0,0,\bar{p}_{1},0,\bar{p}^{2}_{1}-\bar{k}_{2}\bar{p}_{1},$ & $-\bar{p}_{1}$    &\multirow{2}*{\diagbox{}{}} & $-\bar{p}_{1}$ & \multirow{2}*{2}  \\
  \cline{3-3} \cline{5-5}
     ~ & $\bar{p}_{1},\bar{k}_{2}-\bar{p}_{1})$, $\bar{k}_{2}\neq 0$ & $\bar{k}_{2}-\bar{p}_{1}$    & & $\bar{k}_{2}-\bar{p}_{1}$    & \\
      \hline

 \multirow{2}*{3} & \multirow{2}*{$(0,0,1,0,0,0,\bar{q}_{1},0,0,0,\bar{q}_{1},0)$} & $-\bar{q}_{1}$    &\diagbox{}{} & $-\bar{q}_{1}$  &\multirow{2}*{2} \\
\cline{3-5}
~ &~ & $-\bar{q}_{1}\neq \bar{d}\in {\bf k}$   &$\bar{d}+\bar{q}_{1}$ & $1-\bar{q}_{1}$   &  \\
  \hline
 4 & $(0,0,0,1,\bar{p}_{1},-\bar{p}_{1},\bar{p}_{1},0,\bar{p}^{2}_{1},-\bar{p}^{2}_{1},0,\bar{p}_{1})$ & $-\bar{p}_{1}$    &\diagbox{}{} & $-\bar{p}_{1}$  &1\\
\hline
\multirow{2}*{5} & \multirow{2}*{$(1,0,0,0,\bar{p}_{1},0,0,0,0,0,\bar{p}_{1},0)$} & $-\bar{p}_{1}$   &\diagbox{}{} & $-\bar{p}_{1}$  &\multirow{2}*{2}\\
\cline{3-5}
~ & ~ & $-\bar{p}_{1}\neq \bar{d}\in {\bf k}$   &$\bar{d}+\bar{p}_{1}$ & $1-\bar{p}_{1}$  &\\
\hline
 6 & $(1,0,1,0,0,0,0,0,-\frac{1}{4}\bar{k}^{2}_{1},0,\bar{k}_{1},0)$ & -$\frac{\bar{k}_{1}}{2}$   &\diagbox{}{} &-$\frac{\bar{k}_{1}}{2}$  &1\\
  \hline
   \multirow{2}*{7}  &  $(1,0,1,0,0,0,0,0,\bar{a}_{1}-\frac{1}{4}\bar{k}^{2}_{1},0$, & $-\frac{\bar{k}_{1}}{2}+ \sqrt{\bar{a}_{1}}$   &\multirow{2}*{-1} &  \multirow{2}*{$-\frac{\bar{k}_{1}}{2}+ \sqrt{\bar{a}_{1}}$}   & \multirow{2}*{1} \\
  \cline{3-3}
   ~ &   $\bar{k}_{1},0)$, $\bar{a}_{1}\neq 0$  & $-\frac{\bar{k}_{1}}{2}-\sqrt{\bar{a}_{1}}$   & &  & \\
\hline
 8 & $(1,0,0,1,\bar{p}_{1},-\bar{p}_{1},0,0,0,-\bar{p}_{1}^{2},\bar{p}_{1},\bar{p}_{1})$ & $-\bar{p}_{1}$   &\diagbox{}{} & $-\bar{p}_{1}$ &1\\

  \hline
   \multirow{2}*{9}  &  $(1,0,0,1,\bar{p}_{1},-\bar{p}_{1},0,0,0,-\bar{p}_{1}\bar{k}_{2}-\bar{p}_{1}^{2}$,& $-\bar{p}_{1}$   &\multirow{2}*{\diagbox{}{}} ~&  $-\bar{p}_{1}$   & \multirow{2}*{2} \\
  \cline{3-3} \cline{5-5}
      ~ & $\bar{p}_{1},\bar{k}_{2}+\bar{p}_{1})$, $\bar{k}_{2}\neq 0$  & $-\bar{k}_{2}-\bar{p}_{1}$   &~ & $-\bar{k}_{2}-\bar{p}_{1}$   & \\

  \hline
   10   &  \makecell{ $(1,0,0,1,\bar{p}_{1},-\bar{p}_{1},0,0,-\bar{p}_{1}\bar{k}_{1},-\bar{p}_{1}^{2}$,\\$\bar{k}_{1}+\bar{p}_{1},\bar{p}_{1})$, $\bar{k}_{1}\neq 0$} & $-\bar{p}_{1}$   &\diagbox{}{}  & $-\bar{p}_{1}$ & 1\\
  \hline

 \multirow{2}*{11} & \multirow{2}*{$(1,-1,0,1,\bar{p}_{1},-\bar{p}_{1},0,\bar{p}_{1},0,0,\bar{p}_{1},0)$} & $-\bar{p}_{1}$  &\diagbox{}{} & $-\bar{p}_{1}$  &\multirow{2}*{2}\\
  \cline{3-5}
 ~ & ~ & $-\bar{p}_{1}\neq \bar{d}\in {\bf k}$   &$\bar{d}+\bar{p}_{1}$ & $1-\bar{p}_{1}$ & \\
\hline
  12 &  $(0,0,0,0,\bar{q}_{1},0,\bar{q}_{1},0,\bar{q}^{2}_{1},0,0,0)$ & $-\bar{q}_{1}$   &\diagbox{}{}  & $-\bar{q}_{1}$  &1\\

\hline
   \multirow{2}*{13}  &    $(0,0,0,0,\bar{q}_{1},0,\bar{q}_{1},0,\bar{q}_{1}(\bar{q}_{1}-\bar{k}_{1}),0$,  & $-\bar{q}_{1}$    &\multirow{2}*{1}  & \multirow{2}*{$-\bar{q}_{1}$} & \multirow{2}*{1} \\
  \cline{3-3}
    ~ &  $\bar{k}_{1},0)$, $\bar{k}_{1}\neq 0$ & $\bar{k}_{1}-\bar{q}_{1}$   &~ & ~  & \\

  \hline
   14 &    \makecell{ $(0,0,0,0,\bar{q}_{1}+\bar{k}_{1},0,\bar{q}_{1},\bar{k}_{1},\bar{q}^{2}_{1},\bar{q}_{1}\bar{k}_{1}$,\\$\bar{k}_{1},0)$, $\bar{k}_{1}\neq 0$} & $-\bar{q}_{1}$   &\diagbox{}{}  & $-\bar{q}_{1}$ & 1\\

   \hline
 15  &  \makecell{ $(0,0,0,0,\bar{q}_{1},0,\bar{q}_{1},0,\bar{q}^{2}_{1},-\bar{q}_{1}\bar{k}_{2}$,\\$0,\bar{k}_{2})$, $\bar{k}_{2}\neq 0$} & $-\bar{q}_{1}$   &\diagbox{}{}  & $-\bar{q}_{1}$  & 1\\

    \hline
   \multirow{2}*{16}  &    $(0,0,0,0,\bar{q}_{1}+\bar{k}_{2},0,\bar{q}_{1},\bar{k}_{2},\bar{q}^{2}_{1}+\bar{q}_{1}\bar{k}_{2},0$, & $-\bar{q}_{1}$   &\multirow{2}*{1}  & \multirow{2}*{$-\bar{q}_{1}$} &  \multirow{2}*{1} \\
  \cline{3-3}
    ~ &  $0,\bar{k}_{2})$, $\bar{k}_{2}\neq 0$ & $-\bar{q}_{1}-\bar{k}_{2}$   & & & \\
    \hline
\end{tabular}
\\[5pt]
  \caption{deformation maps and complements}\label{table:deformation maps and complements}
\end{table}
Now we classify all the deformation maps, the results refer to the fifth and sixth columns of Table ~\ref{table:deformation maps and complements}.
For each deformation map $d$ and by Proposition ~\ref{prop:bijection complement to deformation}, we obtain a $D$ dendriform complement $B$ of $\cee$ as follows:
 $B=\mathrm{ker}\, \xi=\{d(x)+x|x\in V\}={\bf k}\{e_{2}+\bar{d}e_{1}\}$.
\end{exam}

\section{Further questions}
In this section we mainly give two problems that may be worthy to study in the sequel.
\smallskip

 From the case 3 of Table ~\ref{table:deformation maps and complements}, for each flag datum (i.e., for the different data $\bar{q_1}$) and by Eq.~(\ref{formulas:flag to extension}), we obtain an extension $\cee$.
 Moreover, we find all space complements of $D$ in $E$ are dendriform complements of $D$ in $E$. But in case 1 of Table ~\ref{table:deformation maps and complements}, for each flag datum and by Eq.~(\ref{formulas:flag to extension}), we obtain an extension $\cee$,  and there exist only one dendriform complement of $D$ in $E$. Hence, not all  space complements of $D$ in $E$ are dendriform complements of $D$ in $E$. For this phenomenon, we propose the following definition.
\begin{defn}
Let $D\subset E$ be an extension of dendriform algebras. The extension is called {\bf good} if each space complement of $D$ in $E$ is a dendriform complement of $D$ in $E$.
\end{defn}

The first problem can be expressed as follows.
\smallskip

{\bf The Good Extending Structure Problem:} Let $D$ be a dendriform algebra, $E$ a vector space containing $D$ as a subspace. Describe and classify all good extensions $\cee$ of $D$.  

\smallskip
From the results of Example ~\ref{exam:deformation map}, we find that for each flag datum and by Eq.~(\ref{formulas:flag to extension}), there exsits an extension $\cee$ and at least one $D$ dendriform complement in $E$, where $D$ is a dendriform algebra defined in Example~\ref{exam:dendriform algebra}~\ref{exam:dendriform algebra1}. It's natural to ask whether this phenomenon is suit for other cases. It is not easy to answer the question, but it may be worthy to consider in the sequel.

\smallskip
The second problem can be expressed as follows.

Let $D\subset E$ be an extension of dendriform algebras, does the dendriform complement of $D$ in $E$ always exist?

\smallskip
\noindent
{{\bf Acknowledgments.}
This work is supported in part by Natural Science Foundation of China (No. 12101183).
The authors are thankful to Yi Zhang for helpful suggestions.
}

\noindent
{\bf Declaration of interests.} The authors have no conflicts of interest to disclose.

\smallskip

\noindent
{\bf Data availability.} Data sharing is not applicable as no new data were created or analyzed.

\end{document}